%% file: ex_article.tex
\documentclass[hidelinks,onefignum,onetabnum]{siamart251216}

\input{ex_shared}
\usepackage{comment}

\ifpdf
\hypersetup{
  pdftitle={An Example Article},
  pdfauthor={D. Doe, P. T. Frank, and J. E. Smith}
}
\fi

\begin{document}

\maketitle

\begin{abstract}
We consider a class of smooth $N$-player noncooperative games, where players' objectives are expectation-valued and potentially nonconvex. In such a setting, we consider the largely open question of efficiently computing a quasi-Nash equilibrium (QNE) via a single-step gradient-response framework. First, under a suitably defined quadratic growth property, we prove that both the stochastic synchronous gradient-response (\textbf{SSGR}) scheme and its asynchronous counterpart (\textbf{SAGR}) are characterized by almost sure convergence to a QNE and a sublinear rate guarantee. Second, under a quasi sharpness property, we show that the deterministic synchronous variant displays a linear rate of convergence to a QNE by leveraging a geometric decay in steplengths. This paves the way for developing a practically implementable two-stage scheme that combines sublinearly convergent schemes with a locally linearly convergent second phase. Notably, when the game admits a pseudoconvex potential function, the above convergence claims can be strengthened to a Nash equilibrium (NE), rather than merely to a QNE.  Third, when player problems are convex but the associated concatenated gradient map is potentially non-monotone, we propose a stochastic asynchronous modified gradient-response (\textbf{SAMGR}) scheme which can efficiently obtain an NE under the strict copositivity condition. Collectively, our findings represent some of the first inroads into the tractable computation of QNE/NE in nonconvex settings, leading to a set of single-step schemes that are characterized by broader reach while providing rate guarantees. We present applications satisfying the prescribed requirements where preliminary empirical studies appear promising.
\end{abstract}



\section{Introduction}\label{Sec-1}

In the last several decades, Nash equilibrium (NE) \cite{nash-1951} has assumed growing relevance in  engineered and economic systems, complicated by the presence of competition between a set of self-interested entities. Managing such systems has necessitated the study of algorithms for computing an NE \cite{facchinei-pang-2009}. Specifically, we consider the $N$-player noncooperative game ${\cal G}(\mathbf{f},X, \bxi)$, where ${\bf f}$ denotes the collection of player-specific objectives, i.e.
${\bf f} \triangleq \{f_i\}_{i=1}^N$, $X \triangleq \prod_{i=1}^N X_i$ denotes the Cartesian product of player-specific strategy sets, and the randomness is captured by the random variable $\bxi: \Omega \to
\mathbb{R}^m$ defined on the probability space $(\Omega, \mathcal{F}, \mathbb{P})$. In this game, for any $i \in [N] \triangleq \left\{ 1, 2, \cdots, N \right\}$, the $i$th player solves the optimization problem
\begin{equation*}\label{SQNE}
    \hspace*{-0.4in} \min_{x_i \in X_i} f_i(x_i, x_{-i}) \triangleq \mathbb{E} \left[ \, \tilde{f}_i(x_i, x_{-i}, \boldsymbol{\xi}) \, \right], ~ \, \forall \, i \in [N], \tag{\text{Player$_i(x_{-i})$}}
\end{equation*}
where $X_{i} \subseteq \mathbb{R}^{n_{i}}$, $X_{-i} \triangleq {\prod_{j \neq i}}\ X_j$, and $x_{-i} \triangleq (x_j)_{j \ne i} \in X_{-i}$. In addition, $\Xi \triangleq \left\{ \bxi(\omega) \mid \omega \in \Omega \right\}$ and each $\Tilde{f}_{i}: X \times \Xi \to \mathbb{R}$ is a real-valued function. While most algorithms require convexity and smoothness in player objectives, significantly limiting the applicability of such models, Pang and Scutari \cite{pang-scutari-2011} introduced a weaker equilibrium solution concept in 2011, referred to as a quasi-Nash equilibrium (QNE), that aligns with B-stationarity in optimization problems \cite[Definition 6.1.1]{cui-pang-2021}. 

{\bf Related work.} The NE computation under uncertainty has prompted a study of stochastic gradient-response (GR) \cite{koshal-nedic-shanbhag-2013, lei-shanbhag-2022, yousefian-nedic-shanbhag-2016} as well as best-response (BR) schemes \cite{lei-shanbhag-2022, lei-shanbhag-pang-sen-2020}. In  deterministic nonconvex games, QNE computation has leveraged BR schemes enabled by surrogation~\cite{cui-pang-2021, pang-razaviyayn-2016, razaviyayn-2014,
xiao-shanbhag-2026-BR}. Recall that a QNE of a smooth nonconvex game can be captured by a solution of the associated non-monotone variational inequality problem, which belongs to a challenging class of problems that has seen some recent study. Table \ref{table-1} details extragradient-type (EG) schemes for solving
nonmonotone VIs under the Minty condition and its
variants~\cite{alacaoglu-kim-wright-2024, diakonikolas-daskalakis-jordan-2021,
hsieh-iutzeler-malick-mertikopoulos-2020, huang-zhang-2024, vankov-nedic-sankar-2023}, pseudomonotonicity and its variants
\cite{dang-lan-2015, iusem-jofre-oliveira-thompson-2017, iusem-jofre-oliveira-thompson-2019, kannan-shanbhag-2019, kotsalis-lan-li-2022, yousefian-nedic-shanbhag-2017}, and cohypomonotonicity
\cite{alacaoglu-kim-wright-2024}. Yet several key questions persist.

\vspace{0.1in}
\begin{custombox} 
\textbf{Questions.} (i) Can we develop synchronous/asynchronous SGR schemes with last-iterate guarantees for computing deterministic/stochastic QNE in nonconvex \\games under conditions that go beyond Minty-type conditions and variants of monotonicity? (ii) Can (locally) linear rates be achieved without relying on strong monotonicity? (iii) Are there conditions under which convergence can be strengthened from QNE to NE, despite the scourge of nonconvexity? 
\end{custombox}
\vspace{0.1in}

We resolve the above questions by deriving asymptotic convergence as well as sublinear/linear rate guarantees under diverse assumptions, as summarized in Table~\ref{table-1}.

\begin{table}[htb]
    \centering
    \tiny
    \renewcommand{\arraystretch}{1.0}

    \begin{tabular}{@{}c@{}}

    \resizebox{0.96\textwidth}{!}{%
    \begin{tabular}{|c|ccccc|}
        \Xhline{1.2pt}
        \multicolumn{6}{|c|}{\textbf{Recent work}} \\
        \Xhline{1.2pt}
        \textbf{condition} & \textbf{scheme} & \textbf{stoch.} & \textbf{constr.} & \textbf{conv.} & \textbf{rate} \\
        \hline

        PM/SPM
        & SEG \cite{iusem-jofre-oliveira-thompson-2017, iusem-jofre-oliveira-thompson-2019, kannan-shanbhag-2019} 
        & \ding{51} & \ding{51} & \makecell{a.s.; in expectation} & sublinear \\
        \hline

        GM
        & EG \cite{dang-lan-2015} 
        & \ding{55} & \ding{51} & \ding{51} & sublinear \\
        \hline

        \multirow{2}{*}{GSM} 
        & SOE \cite{kotsalis-lan-li-2022} 
        & \ding{51} & \ding{51} & in expectation & linear \\
        \cline{2-6}
        & SGDA \cite{loizou-berard-gidel-mitliagkas-lacoste-julien-2021} 
        & \ding{55} & \ding{55} & \ding{51} & linear \\
        \hline

        $p$-QS
        & Popov \cite{vankov-nedic-sankar-2023} 
        & \ding{51} & \ding{51} & \makecell{a.s.; in expectation} & sublinear \\
        \hline

        \multirow{3}{*}{MVI} 
        & EG \cite{arefizadeh-nedic-2024} 
        & \ding{55} & \ding{51} & limit points & \ding{55} \\
        \cline{2-6}
        & ARE/PGR/EG \cite{huang-zhang-2024}
        & \ding{55} & \ding{51} & \ding{51} & sublinear \\
        \cline{2-6}
        & DSEG \cite{hsieh-iutzeler-malick-mertikopoulos-2020} 
        & \ding{51} & \ding{55} & \makecell{a.s.; in expectation} & sublinear \\
        \hline

        {WMVI} 
        & EG+ \cite{diakonikolas-daskalakis-jordan-2021} 
        & \ding{55} & \ding{55} & \ding{51} & sublinear \\
        \hline

        CHM/MVI
        & KM \cite{alacaoglu-kim-wright-2024} 
        & \ding{51} & \ding{51} & in expectation & sublinear \\
        \hline
    \end{tabular}
    }

    \\

    \resizebox{0.96\textwidth}{!}{%
    \begin{tabular}{|c|ccccc|}
        \Xhline{1.2pt}
        \multicolumn{6}{|c|}{\textbf{This work}} \\
        \Xhline{1.2pt}
        \textbf{condition} & \textbf{scheme} & \textbf{stoch.} & \textbf{constr.} & \textbf{conv.} & \textbf{rate} \\
        \hline

        AA 
        & SSGR/SAGR 
        & \ding{51} & \ding{51} 
        & \makecell{a.s. subsequential\\(Theorems~\ref{SSGR-AA},~\ref{SAGR-AA})} 
        & \ding{55} \\
        \hline

        QG 
        & SSGR/SAGR 
        & \ding{51} & \ding{51} 
        & \makecell{a.s.; in expectation\\(Theorems~\ref{SSGR-QG},~\ref{SAGR-QG})} 
        & \makecell{sublinear (Theorems~\ref{SSGR-QG},~\ref{SAGR-QG}):\\$\mathcal{O}(k^{-1})$, $\widetilde{\mathcal{O}} (k^{-\min\{2\alpha,\,1/2-\widetilde{d}\}})$} \\
        \hline

        QS
        & SGR 
        & \ding{55} & \ding{51} 
        & \ding{51} 
        & \makecell{locally linear (Theorem~\ref{SGR-QS}):\\
        $\mathcal{O} ((1-(1-\lambda)(\beta/M_F)^2)^k)$} \\
        \hline

        SC
        & SAMGR 
        & \ding{51} & \ding{51} 
        & \makecell{a.s.; in expectation\\(Theorem~\ref{a.s.-SAMGR})} 
        & \makecell{sublinear (Theorem~\ref{rate-complexities-SAMGR}):\\
        $\mathcal{O} ( n^{2-h(a,b,e)}k^{-1} )$} \\
        \Xhline{1.2pt}
    \end{tabular}
    }

    \end{tabular}
    \vspace{-0.1in}
	\captionsetup{font=tiny} 
    \caption{A summary of some recent schemes for non-monotone VIs and the results established in this work. (S)PM: (strongly) pseudomonotone; G(S)M: generalized (strongly) monotone; $p$-QS: $p$-quasi sharpness; (W)MVI: (weak) Minty VI; CHM: cohypomonotone. Other problem-dependent constants are specified in the formal theorem statements.}
    \label{table-1}
    \vspace{-0.2in}
\end{table}

{\bf Outline.} After providing some preliminaries in Section \ref{Sec-2}, we present stochastic synchronous and asynchronous gradient-response (\textbf{SSGR} \& \textbf{SAGR}) schemes with their convergence guarantees in Sections \ref{Sec-3} and \ref{Sec-4}, respectively. Under the additional requirement of pseudoconvexity of the potential function, our guarantees can be strengthened to  the attainment of an NE. In Section \ref{Sec-5}, for stochastic convex but potentially non-monotone games, we present a stochastic asynchronous modified gradient-response (\textbf{SAMGR}) scheme with a.s. convergence and a sublinear rate for computing an NE under the strict copositivity requirement.  We present applications satisfying the prescribed properties in Section \ref{Sec-6} and conclude in Section \ref{Sec-7}.

\textbf{Notation.} The expectation and a realization of a random variable
$\bxi$ are denoted by $\mathbb{E}[\bxi]$ and $\xi$, respectively. The notation $\Pi_{X}[x]$ denotes the Euclidean projection of $x$ onto set $X$. The interior of set $X$ is denoted by $\mathrm{int}\:(X)$. A real-valued function $f:\mathbb{R}^{n}\to\mathbb{R}$ is called a C$^{1}$ function if it is continuously differentiable. The distance from a vector $x\in \mathbb{R}^{n}$ to $A$ is defined as $\mathrm{dist} \: (x, A) \triangleq \inf_{y\in A} \| y-x \|$. We denote the tangent cone and  Fréchet normal cone to $X$ at $x$ by $\mathcal{T}_{X}(x)$ and $\mathcal{N}_{X}(x)$, respectively.

\section{Preliminaries}\label{Sec-2}

We impose the following assumption throughout this paper.

\vspace{5pt}

\emph{Assumption $\mathrm{A}$.} (A1) Each $X_{i}\subseteq \mathbb{R}^{n_{i}}$ is convex and closed for any $i\in [N]$. (A2) Given $x_{-i} \in X_{-i}$, each $f_{i}(\bullet, x_{-i}) \triangleq \mathbb{E}[\Tilde{f}_{i}(\bullet, x_{-i}, \bxi)]$ is C$^{1}$ on an open set $\mathcal{O}_{i} \supseteq X_{i}$ such that $\nabla_{x_{i}} f_{i}(x_{i}, x_{-i}) = \mathbb{E}[\nabla_{x_{i}} \Tilde{f}_{i}(x_{i}, x_{-i},\bxi)]$ for any $i\in [N]$. (A3) The solution set $X^{*}$ of $\mathrm{VI}\:(X, F)$ is closed and nonempty, where $F(x) \triangleq \left(\nabla_{x_i} f_i(x_{i}, x_{-i})\right)_{i=1}^N$. $\hfill \Box$

\subsection{QNE and VIs}

When for any $i \in [N]$, the $i$th player-specific objective $f_{i}(\bullet, x_{-i})$ loses convexity in \eqref{SQNE}, both establishing existence and computing equilibria become challenging. This has led to the weaker solution concept of the {\em quasi-Nash equilibrium} (QNE), based on leveraging B-stationarity \cite[Definition 6.1.1]{cui-pang-2021}. Given an optimization problem $\min_{x\in X} f(x)$, where $f$ is directionally differentiable, we say that $x^{\ast}\in X$ is a B-stationary point of $f$ on $X$ if $f^{\prime}(x^{\ast};v)\ge 0$ for all $v \in {\cal T}_{X}(x^{\ast})$, where $f^{\prime}(x^{\ast};v)$ represents the directional derivative at $x^{\ast}$ along the direction $v$ and ${\cal T}_{X}(x^{\ast})$ denotes the tangent cone to set $X$ at $x^{\ast}$. If $f$ is differentiable and $X$ is convex, B-stationarity of $x^{\ast}$ reduces to
\begin{equation}\label{B-stationarity}
    \nabla f(x^{\ast})^\top (x-x^{\ast}) \geq 0, ~ \, \forall \, x \in X.
\end{equation}
Inspired by this setup, Pang and Scutari \cite{pang-scutari-2011} introduced the QNE. In the definition below, we focus exclusively on the smooth case.

\begin{definition}[\mbox{\cite[Definition 2]{pang-scutari-2011}}]
     Consider the $N$-player game $\mathcal{G}({\bf f},X, \bxi)$. For any $i \in [ N ]$, suppose that each $f_{i}(\bullet, x_{-i})$ is C$^{1}$ for any $x_{-i} \in X_{-i}$. Then we say that $x^{\ast} \triangleq (x^{\ast}_{i})_{i=1}^{N}$ is a quasi-Nash equilibrium (QNE) if for any $i\in [N]$, we have
\begin{equation}\label{QNE}
    \nabla_{x_{i}}f_{i}(x^{\ast}_{i},x^{\ast}_{-i})^\top\left( x_{i}-x^{\ast}_{i} \right) \, \geq \, 0, ~ \, \forall \, x_{i} \in X_{i}.
\end{equation}
\end{definition}

We observe that $x^{\ast}$ is a QNE if and only if $x^{\ast}$ solves the stochastic variational inequality $\mathrm{SVI}\:(X, F)$, i.e., $x^{\ast}$ satisfies
    $F(x^{\ast})^{\top}(x-x^{\ast}) \geq 0$ for all $x \in  X,$
where $F$ is expectation-valued, defined as $F(x) \triangleq \left(\nabla_{x_i} f_i(x)\right)_{i=1}^N$. This facilitates the utilization of VI literature~\cite{facchinei-pang-2003}. We first provide an existence guarantee for QNE, extending the classical existence result of an NE~\cite{nash-1951}, where $\Tilde{F}(x, \xi) \triangleq (\nabla_{x_{i}}\Tilde{f}_{i}(x, \xi))_{i=1}^{N}$ for any $x, \xi$.

\begin{theorem}[QNE existence]\label{QNE-existence}\em
    Consider the $N$-player game $\mathcal{G}({\bf f},X, \bxi)$. Suppose that assumptions $\mathrm{(A1)}$-$\mathrm{(A2)}$ hold. Then a QNE exists if (i) or (ii) holds: (i) If $X$ is bounded; (ii) If there exists $\widehat{x}\in X$ such that
    \begin{equation*}
        \liminf\limits_{\substack{\|x\|\to\infty, x\in X}} \Tilde{F}(x,\xi)^{\top}(x-\widehat{x}) \geq 0,~ \textrm{a.s.}.
    \end{equation*}
\end{theorem}
\begin{proof}
     (i) is directly from~\cite[Corollary 2.2.5]{facchinei-pang-2003}, while (ii) was proven in \cite[Proposition 3.5]{ravat-shanbhag-2011} by combining Lebesgue convergence theorems with variational analysis.
\end{proof}

\begin{remark} Assumption $\mathrm{(A2)}$ imposes smoothness of $f_{i}(\bullet, x_{-i})$, without which QNE may fail to exist despite each $X_{i}$ being compact and convex \cite[p.\,105]{pang-razaviyayn-2016}. $\hfill \Box$
\end{remark}

\subsection{Three properties}\label{Sec-2.2}

Throughout Sections \ref{Sec-3}-\ref{Sec-4}, we derive convergence and rate guarantees for {\bf SSGR} and {\bf SAGR} under three properties: (i) acute angle \eqref{AA}; (ii) quadratic growth \eqref{QG}; and (iii) quasi sharpness \eqref{QS}. We present their definitions below where $X^{\ast}$ denotes the solution set of $\mathrm{VI}\:(X, F)$.

\begin{definition}[Three properties]\label{def:property_VI}
    We say that $\mathrm{VI}\:(X, F)$ satisfies the

    \emph{(i) (\cite[p.\,166]{konnov-2007})} acute angle \eqref{AA} property if 
    \begin{equation}\label{AA}
        (x-x^{*})^{\top}F(x) > 0, ~ \forall x\in X\setminus X^{*}, ~ \forall x^{*}\in X^{*}. \tag{AA}
    \end{equation}

    \emph{(ii) (\cite[Assumption 3]{vankov-nedic-sankar-2023})} quadratic growth \eqref{QG} property if there exists $\alpha > 0$ such that
    \begin{equation}\label{QG}
        (x-x^{*})^{\top}F(x) \geq \alpha \: \mathrm{dist}^{2}(x, X^{*}), ~ \forall x\in X, ~ \forall x^{*}\in X^{*}. \tag{QG}
    \end{equation}

    \emph{(iii) (\cite[Assumption 3]{vankov-nedic-sankar-2023})} quasi sharpness \eqref{QS} if there exists $\beta > 0$ such that
    \begin{equation}\label{QS}
        (x-x^{*})^{\top}F(x) \geq \beta \: \mathrm{dist}\:(x, X^{*}), ~ \forall x\in X, ~ \forall x^{*}\in X^{*}. \tag{QS}
    \end{equation}
\end{definition}

We provide some remarks on the above properties and discuss their connections with several common conditions in the literature. Note that the $p$-quasi sharpness property \cite[Assumption 3]{vankov-nedic-sankar-2023} requires that there exists some $p > 0$, $c > 0$ such that $(x-x^{*})^{\top}F(x) \geq c \: \mathrm{dist}^{p}(x, X^{*})$ holds for all $x\in X$ and all $x^{*}\in X^{*}$. This property reduces to \ref{QG} and \ref{QS} properties when $p = 2$ and $p = 1$, respectively. 

The quadratic growth \eqref{QG} property may be more suitable than a similar property introduced in~\cite{kotsalis-lan-li-2022, loizou-berard-gidel-mitliagkas-lacoste-julien-2021} when $\mathrm{VI}\:(X, F)$ admits multiple solutions. We say that $F$ satisfies generalized strong monotonicity \eqref{GSM} \cite{kotsalis-lan-li-2022} or quasi-strong monotonicity (QSM) \cite{loizou-berard-gidel-mitliagkas-lacoste-julien-2021} on $X$ if there exists $\alpha > 0$ such that
\begin{equation}\label{GSM}
    (x-x^*)^{\top}F(x) \geq \alpha \|x-x^*\|^{2}, ~ \forall x\in X. \tag{GSM}
\end{equation}
Note that such a definition implicitly requires that $X^{*}$ is a singleton. Consider two distinct solutions $x^{*}\neq \widehat{x}$. Then we may see that the left-hand side $(\widehat{x}-x^*)^{\top}F(\widehat{x})$ is nonpositive (from $\widehat{x}$ being a solution)  while the right-hand side $\alpha \|\widehat{x}-x^*\|^{2}$ is strictly positive, which is a contradiction. We say that $\mathrm{VI}\:(X, F)$ satisfies strong pseudomonotonicity \eqref{SP} \cite[Definition 6.4]{karamardian-schaible-1990} if there exists $\alpha > 0$ such that
\begin{equation}\label{SP}
    (x-y)^\top F(y) \geq 0 \implies (x-y)^\top F(x) \geq \alpha \|x-y\|^2, ~ \forall x, y\in X. \tag{SP}
\end{equation}
From the definition of $\mathrm{VI}\:(X,F)$ and by choosing $y = x^\ast \in X$, \ref{SP} implies \ref{GSM}.  We now turn to the quasi sharpness \eqref{QS} property. If $X$ is compact, then $\mathrm{dist}\:(x,X^*)$ is bounded for any $x\in X$ and hence \ref{QS} implies \ref{QG}. Recall that a mapping $F$ is monotone on $X$ if $(x-y)^{\top}(F(x)-F(y)) \geq 0$ holds for any $x, y\in X$. Recall  that $\mathrm{VI}\:(X, F)$ satisfies \eqref{WS} \cite[Definition 2]{yousefian-nedic-shanbhag-2014} if there exists $\beta > 0$ such that
\begin{equation}\label{WS}
    (x-x^{*})^{\top}F(x^{*}) \geq \beta \: \mathrm{dist}\:(x, X^{*}), ~ \forall x\in X, ~ \forall x^{*}\in X^{*}. \tag{WS}
\end{equation}
Under monotonicity of $F$, one can easily verify that \ref{WS} implies \ref{QS}.  The Minty variational inequality (MVI) consists in finding an $x^\ast \in X$ such that
\begin{equation}\label{MVI}
    F(x)^\top(x-x^\ast) \geq 0, ~ \forall x^{*}\in X^{*}. \tag{MVI}
\end{equation}
When $F$ is a monotone map, we may show that the solution set to $\mathrm{MVI}\:(X, F)$ is identical to that of $\mathrm{VI}\:(X, F)$. A particular focus has been on resolving the weak MVI (WMVI) with $\rho > 0$, defined as requiring an $x^{\ast} \in X^{\ast}$ such that 
\begin{equation}\label{WMVI}
    F(x)^{\top} (x-x^\ast) \geq -\frac{\rho}{2} \| F(x) \|^{2}, ~ \forall x^{*}\in X^{*}. \tag{WMVI}
\end{equation}
We may see that \eqref{WMVI} reduces to \eqref{MVI} when $\rho = 0$. For any $\rho > 0$, we say that $F$ is $\tfrac{\rho}{2}-$cohypomonotone \cite{combettes-pennanen-2004} on $X$ if
\begin{equation}\label{CHM}
    (F(x)-F(y))^\top(x-y) \, \ge \, -\tfrac{\rho}{2}\|F(x)-F(y)\|^2, ~ \forall x, y \in X. \tag{CHM}
\end{equation}
We may observe that if $F(x^\ast) = 0$ for some $x^\ast \in X^\ast$, then $\tfrac{\rho}{2}$-cohypomonotonicity of $F$ implies \eqref{WMVI}. All of the relations discussed above are summarized in Figure \ref{VI-properties-summary}.

\begin{figure}[htb]
    \centering
    \hspace*{-0.75cm}
    \begin{tikzpicture}[
        >=Stealth, 
        node distance=1.5cm
        ]
        
        \matrix (m) [
        matrix of nodes,
        row sep=2.8em, 
        column sep=4em,
        nodes={anchor=center}
        ] {
        \ref{QG} &[1cm] $p$-quasi sharpness & \ref{QS} \\
        \ref{GSM} &[1cm] \ref{MVI} & \ref{WS} + monotonicity \\
        \ref{SP} &[1cm] \ref{WMVI} & \ref{CHM} \\
        };
        
        \path[->, font=\small]
        
        (m-1-2) edge node[above] {$p=2$} (m-1-1)
        (m-1-2) edge node[above] {$p=1$} (m-1-3)

        ([xshift = -1mm] m-1-1.south) edge node[left, align=right] {$X^{*}$ singleton} ([xshift = -1mm] m-2-1.north)

        ([xshift = 1mm] m-2-1.north) edge node[right] {} ([xshift = 1mm] m-1-1.south)

        (m-1-1) edge (m-2-2)
        (m-1-3) edge (m-2-2)

        (m-2-1) edge (m-2-2)
        (m-2-3) edge (m-1-3)
        (m-2-2) edge (m-3-2)

        (m-3-1) edge (m-2-1)
        (m-3-1) edge (m-2-2)
        (m-3-1) edge (m-3-2)
        
        (m-1-3) edge[bend right=20] node[above, pos=0.52] {$X \text{ compact}$} (m-1-1);

        \path[<->, font=\small]
        (m-3-2) edge node[above] {$F(x^{*}) = 0$} (m-3-3);
    \end{tikzpicture}
    \caption{Relations between different properties for VIs.}
    \label{VI-properties-summary}
\vspace{-0.3in}
\end{figure}

\section{A synchronous gradient-response scheme}\label{Sec-3}

In this section, we consider QNE computation via a synchronous stochastic GR ({\bf SSGR}) architecture, described in Algorithm~\ref{SSGR}. Under such a framework, we derive convergence guarantees under the properties given in Definition \ref{def:property_VI}. In this section, we impose the assumption of boundedness of $F$ in Assumption $\mathrm{(B3)}$, instead of imposing the commonly used Lipschitz continuity. We denote the history generated by \textbf{SSGR} at iteration $k$ by $\mathcal{F}_{k}\triangleq \sigma \{ x^{0}, \cup_{t=0}^{k-1} \{ \nabla_{x_{i}} \tilde{f}_{i}(x^{t}, \xi^{t}_{i}) \}_{i=1}^{N} \}$.

\vspace{5pt}

\emph{Assumption $\mathrm{B}$.} (B1) For any $k\geq 0$, we have $\mathbb{E}[ w^{k} \mid \mathcal{F}_{k} ] = 0$, where $w^k \triangleq (w_i^k)_{i=1}^N$ and $w^{k}_{i} \triangleq \nabla_{x_{i}}\Tilde{f}_{i}(x^{k}_{i}, x^{k}_{-i}, \xi^{k}_{i}) - \nabla_{x_{i}} \mathbb{E}[ \Tilde{f}_{i}(x^{k}_{i}, x^{k}_{-i}, \bxi) ]$. (B2) For any $k\geq 0$, the second moment bound $\mathbb{E}[ \|w^{k}\|^2 \!\mid\! \mathcal{F}_{k} ]\leq M_{w}^{2}$ holds for some $M_{w} > 0$. (B3) The bound $\| F(x) \| \leq M_{F}$ holds for any $x\in X$ and some $M_{F} > 0$. $\hfill \Box$

\vspace{-0.05in}
\begin{algorithm}[htb]
    \caption{~Stochastic Synchronous GR ({\bf SSGR}) Scheme}
    Set $k = 0$. Initialize $x^{0}\in X$ and stepsizes $\{\gamma^{k}\}_{k\geq 0}$. Iterate until $k\geq K$.\\
    \textbf{Strategy update.} Generate an i.i.d. realization $\xi_{i}^{k}$ of $\bxi$ and update
    \begin{equation*}
        x^{k+1}_{i} \triangleq \Pi_{X_{i}} \left[ x^{k}_{i} - \gamma^{k}\nabla_{x_{i}}\Tilde{f}_{i}(x^{k}_{i}, x^{k}_{-i}, \xi_{i}^{k}) \right], ~ \forall i\in [N].
    \end{equation*}
    \textbf{Return.} $x^{K}$ as final estimate.
    \label{SSGR}
\end{algorithm}
\vspace{-0.05in}

We first derive two recursions for \textbf{SSGR} without imposing any properties from Definition \ref{def:property_VI}. These results play a crucial role in the subsequent analysis. Note that $X^\ast$ is not necessarily a convex set, implying that $\Pi_{X^{*}}\left[x^{k} \right]$ may not be unique.

\begin{lemma}[SSGR recursion]\label{SSGR-recursions}\em
    Consider the $N$-player game $\mathcal{G}({\bf f},X,\bxi)$. Suppose that Assumptions $\mathrm{A}$ and $\mathrm{B}$ hold. Consider the sequence of iterates $\{x^{k}\}_{k=0}^{\infty}$ generated by \textbf{SSGR} and suppose that the stepsizes $\{\gamma^{k}\}_{k\geq 0}$ satisfy $\sum_{k=0}^{\infty}\gamma^{k} = \infty$ and $\sum_{k=0}^{\infty}(\gamma^{k})^{2}<\infty$.   (i) Let $x^{*}\in X^{*}$ be any QNE. Then for any $k$, 
    \begin{equation*}
        \mathbb{E}[ \|x^{k+1}-x^{\ast}\|^2 \mid \mathcal{F}_{k} ] \leq \|x^{k}-x^{\ast}\|^2 -2\gamma^{k}(x^{k}-x^{\ast})^{\top}F(x^{k}) + 2(\gamma^{k})^{2} \left( M^{2}_{w} + M^{2}_{F} \right).
    \end{equation*}

    \noindent (ii) Let $x^{k, \ast} \in \Pi_{X^{*}}[x^{k}]$. Then we have for any $k$,
    \begin{equation*}
        \mathbb{E}[ \mathrm{dist}^{2} (x^{k+1}, X^{*}) \mid \mathcal{F}_{k} ] \leq \mathrm{dist}^{2} (x^{k}, X^{*}) -2\gamma^{k}(x^{k}-x^{k, \ast})^{\top} F(x^{k}) + 2(\gamma^{k})^{2} \left( M^{2}_{w} + M^{2}_{F} \right).
    \end{equation*}
\end{lemma}
\begin{proof}
    (i) For any $x^\ast \in X^\ast$ and any $k \ge 0$,
    \allowdisplaybreaks
    \begin{align*}
        & \|x^{k+1} - x^{\ast}\|^2 \leq \sum\limits_{i=1}^{N} \| (x^{k}_{i}-x^{\ast}_{i}) - \gamma^{k}\nabla_{x_{i}}\Tilde{f}_{i}(x^{k}_{i},x^{k}_{-i},\xi^{k}) \|^2 = \|x^{k}-x^{\ast}\|^2 \\
	    & - 2\gamma^{k}\sum_{i=1}^{N}(x^{k}_{i}-x^{\ast}_{i})^{\top}\nabla_{x_{i}}\Tilde{f}_{i}(x^{k}_{i}, x^{k}_{-i}, \xi_{i}^{k}) +(\gamma^{k})^2\sum\limits_{i=1}^{N}\|\nabla_{x_{i}}\Tilde{f}_{i}(x^{k}_{i},x^{k}_{-i},\xi^{k}_{i})\|^2 \\
        & \le \|x^{k}-x^{\ast}\|^2 - 2\gamma^{k}\sum_{i=1}^{N}(x^{k}_{i}-x^{\ast}_{i})^{\top}(\nabla_{x_{i}}\mathbb{E}[\Tilde{f}_{i}(x^{k}_{i},x^{k}_{-i},\bxi)]+w^{k}_{i}) + 2(\gamma^{k})^2\sum_{i=1}^{N}\|w^{k}_{i}\|^2 \\
        &+ 2(\gamma^{k})^2\sum_{i=1}^{N}\| \nabla_{x_{i}}\mathbb{E}[\Tilde{f}_{i}(x^{k}_{i},x^{k}_{-i},\bxi)] \|^2,
    \end{align*}
    where the first inequality follows from the nonexpansiveness of the Euclidean projection. By taking conditional expectations and invoking assumption $\mathrm{(B1)}$, we have
    \allowdisplaybreaks
    \begin{align*}
        \mathbb{E}[ \|x^{k+1} &- x^{\ast}\|^2 \mid \mathcal{F}_{k} ] \leq \|x^{k}-x^{\ast}\|^2 \underbrace{- 2\gamma^{k}\sum_{i=1}^{N}(x^{k}_{i}-x^{\ast}_{i})^{\top}(\nabla_{x_{i}}\mathbb{E}[\Tilde{f}_{i}(x^{k}_{i},x^{k}_{-i},\bxi)])}_{\textrm{Term (a)\,=\,$-2\gamma^{k}(x^{k}-x^{\ast})^{\top}F(x^{k})$}} \\
        &\underbrace{+ 2(\gamma^{k})^2\sum_{i=1}^{N} \mathbb{E}[ \|w^{k}_{i}\|^2 \mid \mathcal{F}_{k}]}_{\textrm{Term (b)}}\underbrace{+ 2(\gamma^{k})^2\sum_{i=1}^{N}\| \nabla_{x_{i}}\mathbb{E}[\Tilde{f}_{i}(x^{k}_{i},x^{k}_{-i},\bxi)] \|^2}_{\textrm{Term (c)}}.
    \end{align*}
    By Assumption $\mathrm{B}$, we may obtain the upper bounds $\textrm{Term (b)} \leq 2(\gamma^{k})^2 M^{2}_{w}$ and $\textrm{Term (c)} \leq 2(\gamma^{k})^2 M^{2}_{F}$. Then we complete the proof of (i). By invoking the definition of $x^{k, \ast}$ and choosing $x^\ast = x^{k,\ast}$, we have that $\|x^k - x^{k,\ast}\|^2 = \mathrm{dist}^2(x^k,X^\ast)$. The result in (ii) follows from noting that $\mathrm{dist}^{2} (x^{k+1}, X^{*}) \leq \|x^{k+1}-x^{k, \ast}\|^2$.
\end{proof}

\subsection{SSGR under \ref{AA} and \ref{QG}}\label{Sec-3.1}

We first recall three useful lemmas on the convergence of random variables.

\begin{lemma}\label{three-useful-lemma}\em
(i) (Robbins-Siegmund \cite{robbins-siegmund-1971}) Let $\{\nu^{k}\}_{k=0}^{\infty}$, $\{\theta^{k}\}_{k=0}^{\infty}$, $\{\varepsilon^{k}\}_{k=0}^{\infty}$ and $\{\delta^{k}\}_{k=0}^{\infty}$ be nonnegative sequences of random variables such that $\sum_{k=0}^{\infty}\delta^{k}<\infty$, $\sum_{k=0}^{\infty}\varepsilon^{k}<\infty$ and $\mathbb{E} [ \: \nu^{k+1} \mid \mathcal{F}_{k} \: ] \leq (1+\delta^{k})\nu^{k} - \theta^{k} + \varepsilon^{k}$, a.s. Then, $\sum_{k=0}^{\infty}\theta^{k}<\infty$ and $\nu^k \xrightarrow{a.s.} v$ as $k\to \infty$ where $v \ge 0$ is a random variable.

(ii) (\cite[Lemma 2.2.10]{polyak-1987}) Let $\{\nu^{k}\}_{k=0}^{\infty}$ be a nonnegative sequence of random variables and $\{\alpha^k\}$ and $\{\mu^k\}$ be deterministic sequences such that $0 \le \alpha^k \le 1$ and $\mu^k \ge 0$ for all $k$ and $\sum_{k=1}^{\infty} \alpha^k = \infty$ and $\lim_{k \to \infty} \tfrac{\mu^k}{\alpha^k} = 0$, and $\mathbb{E}[ \: \nu^{k+1} \mid {\cal F}_k \: ] \le (1-\alpha^k) \nu^k + \mu^k$ for $k \ge 0$. Then $\nu^k \xrightarrow{a.s.} 0$ as $k\to \infty$.

(iii) (\cite[Section 8.2]{shapiro-dentcheva-ruszczynski-2021}) Suppose that the nonnegative sequence $\{e^{k}\}_{k=0}^{\infty}$ satisfies $e^{k+1}\leq (1-2\alpha\gamma^{k})e^{k} + (\gamma^{k})^{2}M$ for all $k\geq 0$, where $\alpha, M>0$. Let $\gamma^{k} = \gamma^{0}/k$, where $\gamma^{0}>\tfrac{1}{2\alpha}$. Let $Q(\gamma^{0}) \triangleq \max \{\tfrac{(\gamma^{0})^{2}M}{2\alpha\gamma^{0}-1}, e^{0} \}$. Then for all $k \geq 1$, we have $e^{k} \leq \tfrac{Q(\gamma^{0})}{k}$. $\hfill \Box$
\end{lemma}

Based on Robbins-Siegmund lemma, we now derive asymptotic \emph{subsequential} a.s. convergence under \ref{AA} and an additional compactness assumption on $X^{\ast}$.

\begin{theorem}[SSGR under \ref{AA}]\label{SSGR-AA}\em 
    Consider the $N$-player game $\mathcal{G}({\bf f},X,\bxi)$. Suppose that Assumptions $\mathrm{A}$ and $\mathrm{B}$ hold, and that the \ref{AA} property is satisfied. Consider the sequence $\{x^k\}_{k=0}^{\infty}$ generated by \textbf{SSGR}. Suppose that the stepsizes $\{\gamma^{k}\}_{k=0}^{\infty}$ satisfy  $\sum_{k=0}^{\infty}\gamma^{k} = \infty$ and $\sum_{k=0}^{\infty}(\gamma^{k})^{2}<\infty$. If $X^{\ast}$ is compact, then \emph{some subsequence} of $\{x^{k}\}_{k=0}^{\infty}$ converges to a QNE a.s.
\end{theorem}
\begin{proof}
Beginning from the SSGR recursion in Lemma \ref{SSGR-recursions}-(i), we consider two cases: (I) There are infinitely many iterates $\{x^{k}\}$ in the solution set $X^{\ast}$. (II) There are finitely many iterates $\{x^{k}\}$ in the solution set $X^{\ast}$. In case (I), the final conclusion holds since $X^{\ast}$ is compact. In case (II), there are only finitely many $\{ x^{k} \}$ in the solution set $X^{\ast}$. Therefore for some large $K$, we have $x^{k}\in X\setminus X^{\ast}$ for any $k\geq K$.  By the square summability of $\{\gamma^k\}_{k=0}^{\infty}$ and \ref{AA} property $(x^{k}-x^{\ast})^{\top}F(x^{k})>0$ for any $x^k \in X \setminus X^\ast$ and $x^\ast \in X^\ast$, we may invoke Lemma~\ref{three-useful-lemma}-(i), whereby
\begin{equation*}
    \mathrm{(a)}\;\{ \|x^{k}-x^{\ast}\| \}_{k=0}^{\infty} \textrm{ converges a.s.} \; \textrm{ and } \; \mathrm{(b)}\;\sum_{k=0}^{\infty}2\gamma^{k}(x^{k}-x^{\ast})^{\top}F(x^{k})<\infty \textrm{ a.s.}
\end{equation*}
By  (a), we know that there exists some $a > 0$ such that $\lim_{k\to\infty} \|x^{k}-x^{\ast}\| = a \geq 0$ holds a.s. It implies that for some sufficiently large $\tilde K$, we have $\|x^{k}-x^{\ast}\|\leq a+1$ a.s. for any $k\geq \tilde{K}$, i.e., $\{x^{k}\}_{k = \tilde{K}}^{\infty}$ is bounded a.s. and consequently, the entire sequence $\{x^{k}\}_{k = 0}^{\infty}$ is bounded a.s. By the non-summability condition
$\sum_{k=0}^{\infty}\gamma^{k} = \infty$, (b) implies that $\liminf_{k\to\infty}(x^{k}-x^{\ast})^{\top}F(x^{k}) = 0$, i.e., there exists some subsequence $\{x^{k_{l}}\}_{l=0}^{\infty}$ such that $(x^{k_{l}}-x^{\ast})^{\top}F(x^{k_{l}})\to 0$ as $l\to \infty$. Since the sequence $\{x^{k}\}_{k = 0}^{\infty}$ is bounded a.s., it implies that the subsequence $\{x^{k_{l}}\}_{l=0}^{\infty}$ is bounded a.s. Without loss of generality, we may assume that $\{x^{k_{l}}\}_{l=0}^{\infty}$ is convergent a.s. (otherwise we may continue to take a subsequence if needed), i.e., $\lim_{l\to\infty} x^{k_{l}} = \Tilde{x}$ a.s. By continuity of $F$, it follows that
\begin{equation*}
    \lim_{l\to\infty} (x^{k_{l}}-x^{\ast})^{\top}F(x^{k_{l}}) = (\Tilde{x}-x^{\ast})^{\top}F(\Tilde{x}) = 0 \quad \mbox{a.s.}
\end{equation*}
By the \ref{AA} property, we have $\Tilde{x}\in X^{\ast}$ (where $\Tilde{x}$ may differ from $x^{\ast}$); if not, $(\Tilde{x} -x^{\ast})^{\top}F(\Tilde{x})>0$, violating the \ref{AA} property. This completes the proof.
\end{proof}

\begin{remark}
    One result on stochastic extragradient (SEG) methods \cite{hsieh-iutzeler-malick-mertikopoulos-2020} is particularly relevant to our findings. A stronger sequential (rather than subsequential) guarantee of almost sure convergence is established under a  weaker Minty condition \cite[Theorem 1]{hsieh-iutzeler-malick-mertikopoulos-2020}. However, this result leverages a ``double steplength'' framework and is provided for an unconstrained VI where a convenient residual function is available. Our schemes, while providing subsequential guarantees, are provided under constrained regimes and hold additional advantages in that the single-step structure requires only one oracle evaluation and one projection evaluation per iteration, making them cheaper and easier to implement than extragradient-type methods, especially in constrained, nonsmooth, large-scale, or distributed settings. $\hfill \Box$
\end{remark}

Having derived \emph{subsequential} a.s. convergence for \textbf{SSGR} under \ref{AA}, we now derive a sublinear rate statement for \textbf{SSGR} under \ref{QG}.

\begin{theorem}[SSGR under \ref{QG}]\label{SSGR-QG}\em
    Consider the $N$-player game $\mathcal{G}({\bf f}, X, \bxi)$. Suppose that Assumptions $\mathrm{A}$ and $\mathrm{B}$ hold, and that the \ref{QG} property is satisfied. Consider the sequence $\{x^k\}_{k=0}^{\infty}$ generated by \textbf{SSGR}. Suppose $\alpha>0$ is the \ref{QG} coefficient. Then we have the following.
    
    (i) (diminishing stepsize) Suppose $\gamma^k=\gamma^0/k$, where $\gamma^0>1/(2\alpha)$. Then we have that (i-1) $\lim_{k\to\infty} \mathrm{dist}\:(x^{k}, X^{*}) = 0$ a.s.; and (i-2) $\mathbb{E}[ \mathrm{dist}^{2} (x^{k}, X^{*}) ] \leq \tfrac{Q(\gamma^{0})}{k}$ holds for any $k\geq 1$, where $Q(\gamma^{0}) \triangleq \max \{ \tfrac{2(\gamma^{0})^{2}(M^{2}_{w} + M^{2}_{F})}{2\alpha\gamma^{0}-1}, \mathrm{dist}^{2} (x^{0}, X^{*}) \}$ with $M_{w}$ and $M_{F}$ given in Assumption $\mathrm{B}$.

    (ii) (constant stepsize) Consider the constant stepsize $\gamma^{k} \equiv \gamma \in (0, \tfrac{1}{2\alpha})$ for all $k\geq 0$. Then the following hold. (ii-1) For any sufficiently small $\gamma > 0$ such that $\log(\tfrac{1}{\gamma})>0$, we have $\mathbb{E}[ \mathrm{dist}^{2} (x^{k}, X^{*}) ] \leq \mathcal{O}(\gamma)$ for any $k > \lceil\tfrac{1}{2\alpha\gamma}\log(\tfrac{1}{\gamma})\rceil$. (ii-2)  For any fixed $K > \tilde{C} \triangleq \tfrac{C}{4\alpha^2}$, where $C \triangleq 2\left( M^{2}_{w} + M^{2}_{F} \right)$ and $M_{w}$ and $M_{F}$ are defined in Assumption $\mathrm{B}$, if we choose the constant stepsize $\gamma^{\ast}(K) \triangleq \tfrac{1}{2\alpha} (1 - (\tfrac{\tilde{C}}{K+1})^{1/K} )$, then we have $\mathbb{E}[ \mathrm{dist}^{2} (x^{K+1}, X^{*})] \leq \mathcal{O}(\tfrac{\log (K)}{K})$.  
\end{theorem}
\begin{proof} (i) By invoking the \ref{QG} property in the \textbf{SSGR} recursion given in Lemma \ref{SSGR-recursions}-(ii), it leads to
    \begin{equation}\label{SSGR-QG-conditional}
        \mathbb{E}[ \mathrm{dist}^{2} (x^{k+1}, X^{*}) \mid \mathcal{F}_{k} ] \leq (1-2\alpha\gamma^{k}) \, \mathrm{dist}^{2} (x^{k}, X^{*}) + 2(\gamma^{k})^{2} \left( M^{2}_{w} + M^{2}_{F} \right).
    \end{equation}
    When $k$ is sufficiently large, we have $0\leq 2\alpha\gamma^{k}\leq 1$. By Lemma \ref{three-useful-lemma}-(ii), we may claim the a.s. convergence (i-1). By taking unconditional expectations on both sides of (\ref{SSGR-QG-conditional}), it follows that
    \begin{equation}\label{SSGR-QG-unconditional}
        \mathbb{E}[ \mathrm{dist}^{2} (x^{k+1}, X^{*}) ] \leq (1-2\alpha\gamma^{k}) \, \mathbb{E}[ \mathrm{dist}^{2} (x^{k}, X^{*}) ] + 2(\gamma^{k})^{2} \left( M^{2}_{w} + M^{2}_{F} \right).
    \end{equation}
    By invoking Lemma \ref{three-useful-lemma}-(iii), we have that $\mathbb{E}[ \mathrm{dist}^{2} (x^{k}, X^{*}) ] \leq \tfrac{Q(\gamma^{0})}{k}$ for $k\geq 1$, where $Q(\gamma^{0}) \triangleq \max \{ \tfrac{2(\gamma^{0})^{2}(M^{2}_{w} + M^{2}_{F})}{2\alpha\gamma^{0}-1}, \mathrm{dist}^{2} (x^{0}, X^{*}) \}$. This completes the proof of (i-2).

    (ii) We define $\e^{k} \triangleq \mathbb{E}[ \mathrm{dist}^{2} (x^{k}, X^{*}) ]$ and $C \triangleq 2\left( M^{2}_{w} + M^{2}_{F} \right)$. By \eqref{SSGR-QG-unconditional} and the fact that $\gamma^{k} \equiv \gamma$ for all $k\geq 0$, we may obtain that $ \e^{k+1} \leq (1-2\alpha \gamma) \e^{k} + \gamma^{2} C$. Since $\gamma \in (0, \tfrac{1}{2\alpha})$ such that $v \triangleq 1-2\alpha \gamma \in (0, 1)$, we have
    \begin{align}\label{constant-log-factor}
	    \e^{k+1} & \leq v \e^{k} + \gamma^{2} C \leq v^2 \e^{k-1} + v \gamma^2 C + \gamma^2 C \notag \\
		& \leq v^{k+1} \e^0 + \gamma^2 C (1 + v + \cdots + v^{k}) \leq v^{k+1} \e^0 + \gamma^2 C \tfrac{1}{1-v} = v^{k+1} \e^{0} + \tfrac{\gamma C}{2\alpha}.
    \end{align}
    
    \noindent (ii-1) For any $k > \lceil \tfrac{1}{2\alpha\gamma}\log(\tfrac{1}{\gamma}) \rceil$, we have
    \begin{equation*}
        v^{k} = (1 - 2\alpha \gamma)^{k} \leq (1 - 2\alpha \gamma)^{\tfrac{1}{2\alpha\gamma}\log(\tfrac{1}{\gamma})} \leq (e^{-2\alpha})^{\tfrac{1}{2\alpha}\log(\tfrac{1}{\gamma})} = \gamma,
    \end{equation*}
    where the last inequality follows from $\log(\tfrac{1}{\gamma}) > 0$ and the fact that $(1-x/n)^n \le e^{-x}$ when $0\leq x < n$. Therefore, we have
    \begin{equation*}
        v^{k} \leq \gamma \implies \e^{k} \leq v^{k} \e^{0} + \tfrac{\gamma C}{2\alpha} = \mathcal{O}(\gamma).
    \end{equation*}
    Then we complete the proof.
    
    \noindent (ii-2) Consider any fixed $K > \tilde{C} \triangleq \tfrac{C}{4\alpha^2}$. We observe from \eqref{constant-log-factor} that if $\e^{0} \le 1$, then we have $\e^{K+1} \leq v^{K+1} + \tfrac{\gamma C}{2\alpha}$, while if $\e^{0} > 1$, it follows that $\e^{K+1} \leq v^{K+1} \e^{0} + \tfrac{\gamma C}{2\alpha} \e^{0}$. Consequently, we have that
    \begin{equation}\label{e_K+1_bound}
         \e^{K+1} \leq \underbrace{(v^{K+1}  + \tfrac{\gamma C}{2\alpha} )}_{\triangleq h_{K}(\gamma)} \max\left\{ \e^0,1\right\} \,\, \mbox{ and } \,\, h_{K}(\gamma) = (1-2\alpha\gamma)^{K+1} + \tfrac{\gamma C}{2\alpha}.
    \end{equation}
    We now examine $h_{K}(\gamma^{\ast})$ as follows. For any given $K > \tilde{C}$, we have that $(\frac{\tilde{C}}{K+1})^{1/K} \in (0, 1)$ and $\gamma^{\ast} \in (0, \tfrac{1}{2\alpha})$. Then we have
    \allowdisplaybreaks
    \begin{align*}
        h_{K}(\gamma^{\ast}) &= (1 - 2\alpha\gamma^{\ast})^{K+1}  + \tfrac{\gamma^{\ast} C}{2\alpha} = ( \tfrac{\tilde{C}}{K+1} )^{(K+1)/K} + \tilde{C} \left( 1 - (\tfrac{\tilde{C}}{K+1})^{1/K} \right) \\
        &= \underbrace{\tilde{C} ( 1-\tfrac{K}{K+1} ) (\tfrac{\tilde{C}}{K+1})^{1/K}}_{\textrm{Term (a)}} \,+\, \underbrace{\tilde{C} \left(1 - (\tfrac{\tilde{C}}{K+1})^{1/K}\right)}_{\textrm{Term (b)}}. 
    \end{align*}
    We now proceed to show that $h_{K}(\gamma^{\ast}) = \mathcal{O}(\tfrac{\log(K)}{K})$. We first examine the term $(\tfrac{\tilde{C}}{K+1})^{1/K}$. Since $\log(\tilde{C}) < \tilde{C} < K$, we have that $\log(\tilde{C})/K < 1$. Then by applying the inequality $e^x \leq 1+\frac{x}{1-x} ~ (\forall x < 1)$ where we select $x = \log(\tilde{C})/K < 1$ yields
    \allowdisplaybreaks
    \begin{align}\label{log-factor-proof-eqn1}
        \tilde{C}^{1/K} =  e^{\log(\tilde{C})/K} \leq 1 + \tfrac{\log(\tilde{C})}{K-\log(\tilde{C})}.
    \end{align} 
    Consequently, Term (a) can be bounded as follows. 
    \begin{align*}
        \textrm{Term (a)} = \tilde{C} \tfrac{1}{K+1} \tilde{C}^{1/K}{\underbrace{( \tfrac{1}{K+1} )^{1/K}}_{<1}} < \tilde{C} \tfrac{1}{K+1} \tilde{C}^{1/K} \overset{\eqref{log-factor-proof-eqn1}}{\leq} \tilde{C} \left( \tfrac{1}{K+1} + \tfrac{\log(\tilde{C})}{(K-\log(\tilde{C}))(K+1)} \right).
    \end{align*}
    Similarly, Term (b) can be bounded in a similar fashion. Observe that since $K > \tilde{C}$, we have that $\log(K+1) > \log(\tilde{C})$. By applying the inequality $1 - e^{-x} \leq x ~ (\forall x \geq 0)$ where we choose $x = \tfrac{\log(K+1) - \log(\tilde{C})}{K} \geq 0$ yields
    \begin{align*}
        \textrm{Term (b)} = \tilde{C} ( 1 - e^{\log(\frac{\tilde{C}}{K+1})/K} ) = \tilde{C} ( 1 - e^{-\frac{\log(K+1) - \log(\tilde{C})}{K}} ) \leq \tilde{C} \left(\tfrac{\log(K+1) - \log(\tilde{C})}{K}\right).
    \end{align*}
    By combining the above two bounds of Terms (a) and (b), we have that
    \begin{equation}\label{h_K_gamma_bound}
        h_{K}(\gamma^{\ast}) \leq \tilde{C} \left( \tfrac{1}{K+1} + \tfrac{\log(\tilde{C})}{(K-\log(\tilde{C}))(K+1)} \right) + \tilde{C} \left(\, \tfrac{\log(K+1) - \log(\tilde{C})}{K} \, \right)= \mathcal{O}\left(\tfrac{\log(K)}{K} \right).
    \end{equation}
    Therefore, for any given $K > \tilde{C}$, by choosing $\gamma = \gamma^{*}$, we know from \eqref{e_K+1_bound} and \eqref{h_K_gamma_bound} that $\e^{K+1} \leq \mathcal{O} (\tfrac{\log(K)}{K})$. 
\end{proof}

\begin{remark}
    We make two remarks here. First, unlike (ii-1), the claim in (ii-2) does not hold asymptotically since the constant stepsize $\gamma^{*}$ in (ii-2) depends on $K$. In fact, we may observe that $\gamma^{*}(K)$ minimizes $h_{K}(\gamma)$ over the open interval $(0, \tfrac{1}{2\alpha})$. Second, we discuss the relation between (ii-1) and (ii-2). By setting $\tfrac{\log(K)}{K} = \gamma$ in (ii-2), it follows from the Lambert function \cite{veberic-2010} that, for sufficiently small $\gamma$, the larger solution (associated with the equation $we^{w} = \gamma$) satisfies $ K = \mathcal{O}(\frac{1}{\gamma}\log\frac{1}{\gamma})$. This agrees with the minimum value of $k$ we established in (ii-1) up to the order of magnitude. $\hfill \Box$
\end{remark}

\subsection{A linear rate of SGR under \ref{QS}}\label{Sec-3.2}

In this subsection, we focus on the deterministic regime and establish a linear rate under the \ref{QS} property, drawing inspiration from recent results on weakly convex optimization \cite{davis-drusvyatskiy-macphee-paquette-2018}. This requires significant extension to contend with game-theoretic generalizations and rely on leveraging tools from the VI theory. In contrast with \cite{davis-drusvyatskiy-macphee-paquette-2018}, our convergence results are established without Lipschitz continuity or weak monotonicity assumptions. To this end, we consider the synchronous gradient-response (\textbf{SGR}) scheme for a given $x_0 \in X$:
\begin{align}\label{SGR}
    x^{k+1}_{i} = \Pi_{X_{i}} \left[ x^{k}_{i} - \gamma^{k}\nabla_{x_{i}}f_{i}(x^{k}_{i}, x^{k}_{-i}) \right],~ \forall i\in [N]. \tag{SGR}
\end{align}
The convergence analysis relies on (i) geometrically decaying stepsizes $\gamma^{k} = \gamma^{0}q^{k}$ for some $\gamma^{0}>0$ and $q\in (0,1)$; and (ii) a suitable initialization $\mathrm{dist}\:(x^{0}, X^{*}) \leq D$ for some \emph{known} $D>0$. 

\begin{remark}
    The stochastic extension is not straightforward. Davis et al.
    \cite{davis-drusvyatskiy-charisopoulos-2024} extended their early work \cite{davis-drusvyatskiy-macphee-paquette-2018} to the stochastic regime by a restart technique, which is distinct from our single-step setup here. We leave this for our future research. $\hfill \Box$
\end{remark}

Similar to \cite[Theorem 6.1]{davis-drusvyatskiy-macphee-paquette-2018}, we may derive a linear rate of \textbf{SGR} under \ref{QS}.

\begin{theorem}[SGR under \ref{QS}]\label{SGR-QS}\em
    Consider the deterministic specialization of the $N$-player game $\mathcal{G}({\bf f}, X, \bxi)$. Suppose that assumptions $\mathrm{(A3)}$ and $\mathrm{(B3)}$ hold, and the \ref{QS} property is satisfied. Consider the sequence $\{x^k\}_{k=0}^{\infty}$ generated by \textbf{SGR} with geometrically decaying stepsizes $\gamma^{k} = \gamma^{0}q^{k}$ for some $\gamma^{0}>0$ and $q \in (0, 1)$. Suppose that $\tfrac{\beta}{M_{F}} \leq \sqrt{\tfrac{1}{2-\lambda}} < 1$ for some $\lambda \in (0, 1)$, where $\beta > 0$ is the \ref{QS} coefficient and $M_{F}$ is the bound in assumption $\mathrm{(B3)}$. We set $(\gamma^{0}, q)$ and initialization as
    \begin{equation*}
        \gamma^{0} > 0, \, q \triangleq \sqrt{1-(1-\lambda)\tfrac{\beta^{2}}{M^{2}_{F}}}\in (0, 1), ~ \textrm{and} ~ \mathrm{dist}\:(x^{0}, X^{*}) \leq D \triangleq \tfrac{\gamma^{0} M_{F}}{\tfrac{\beta}{M_{F}} - \sqrt{\tfrac{\beta^{2}}{M^{2}_{F}}-(1-q^{2})}}.
    \end{equation*}
    Then for any $k\geq 0$, we have
    \begin{equation}\label{SGR-QS-main-result}
        \mathrm{dist}^{2}(x^{k}, X^{*}) \leq \max \left\{ \tfrac{(\gamma^{0})^{2} M^{4}_{F}}{\beta^{2}}, \mathrm{dist}^{2}(x^{0}, X^{*}) \right\} \left( 1 - (1-\lambda) \tfrac{\beta^{2}}{M^{2}_{F}} \right)^{k}.
    \end{equation}
\end{theorem}
\begin{proof}
    Akin to the derivation of Lemma \ref{SSGR-recursions}-(ii), we may obtain that
    \begin{equation*}
        \mathrm{dist}^{2} (x^{k+1}, X^{*}) \leq \mathrm{dist}^{2} (x^{k}, X^{*}) -2\gamma^{k}(x^{k}-x^{k, \ast})^{\top} F(x^{k}) + (\gamma^{k})^{2} M^{2}_{F}.
    \end{equation*}
    By invoking the \ref{QS} property, it leads to
    \begin{equation*}
        \mathrm{dist}^{2} (x^{k+1}, X^{*}) \leq \mathrm{dist}^{2} (x^{k}, X^{*}) -2 \beta \gamma^{k} \mathrm{dist}\:(x^{k}, X^{*}) + (\gamma^{k})^{2} M^{2}_{F}.
    \end{equation*}
    By dividing both sides by $M^{2}_{F}$ and using $\gamma^{k} = \gamma^{0} q^{k}$, it follows that
    \begin{equation}\label{SGR-QS-eqn1}
        \left( \tfrac{\mathrm{dist}\:(x^{k+1}, X^{*})}{M_{F}} \right)^{2} \leq \left( \tfrac{\mathrm{dist}\:(x^{k}, X^{*})}{M_{F}} \right)^{2} - \tfrac{2\beta\gamma^{0} q^{k}}{M_{F}}  \tfrac{\mathrm{dist}\:(x^{k}, X^{*})}{M_{F}} + (\gamma^{0})^{2} q^{2k}.
    \end{equation}
    
    We prove the result by induction. When $k = 0$, we can see that \eqref{SGR-QS-main-result} holds trivially. Suppose that \eqref{SGR-QS-main-result} holds for some positive $k$. It remains to show that the result \eqref{SGR-QS-main-result} also holds for $k + 1$. We define the constant
    \begin{equation*}
        E \triangleq \max \left\{ \tfrac{\gamma^{0} M_{F}}{\beta}, \tfrac{\mathrm{dist} \: (x^{0}, X^{*})}{M_{F}} \right\}.
    \end{equation*}
    Then we may rewrite \eqref{SGR-QS-main-result} equivalently as $\tfrac{\mathrm{dist}\:(x^{k}, X^{*})}{M_{F}} \leq E q^{k}$. Hence, there exists some $R\in [0, E]$ such that $\tfrac{\mathrm{dist}\:(x^{k}, X^{*})}{M_{F}} = R q^{k}$. We proceed to show that $\tfrac{\mathrm{dist}\:(x^{k+1}, X^{*})}{M_{F}} \leq E q^{k+1}$ holds. By \eqref{SGR-QS-eqn1}, we obtain that
    \begin{align}
        \left( \tfrac{\mathrm{dist}\:(x^{k+1}, X^{*})}{M_{F}} \right)^{2} &\leq R^{2} q^{2k} -2 \tfrac{\beta}{M_{F}} \gamma^{0} R q^{2k} + (\gamma^{0})^{2} q^{2k} \notag \\
        &\leq \max_{R\in [0, E]} \left\{ R^{2} q^{2k} -2 \tfrac{\beta}{M_{F}} \gamma^{0} R q^{2k} + (\gamma^{0})^{2} q^{2k} \right\} \notag \\
        &= q^{2k} \cdot \max \left\{ (\gamma^{0})^{2}, E^{2} -2 \tfrac{\beta}{M_{F}} \gamma^{0} E + (\gamma^{0})^{2} \right\}, \label{SGR-induction}
    \end{align}
    where the last equality is due to the fact that the maximum of a univariate convex quadratic function is attained either at $R = 0$ or $R = E$. It suffices to show that
    \begin{equation*}
        \mathrm{(a)} \; (\gamma^{0})^{2} \leq E^{2}q^{2}; \; \textrm{ and } \; \mathrm{(b)} \; E^{2} -2 \tfrac{\beta}{M_{F}} \gamma^{0} E + (\gamma^{0})^{2} \leq E^{2}q^{2}.
    \end{equation*}
    Observe that if (a) and (b) are true, then the required inductive claim holds by invoking \eqref{SGR-induction} and noting
    \begin{equation*}
        \max \left\{ (\gamma^{0})^{2}, E^{2} -2 \tfrac{\beta}{M_{F}} \gamma^{0} E + (\gamma^{0})^{2} \right\} \leq E^2 q^2.
    \end{equation*}
    The first requirement $\mathrm{(a)}$ can be proven by noting that
    \begin{equation*}
        \tfrac{\beta^{2}}{M^{2}_{F}} \overset{(*)}{\leq} 1 - (1 - \lambda) \tfrac{\beta^{2}}{M^{2}_{F}} = q^{2} \implies E^{2} \geq \tfrac{(\gamma^{0})^{2} M^{2}_{F}}{\beta^{2}} \geq \tfrac{(\gamma^{0})^{2}}{q^{2}},
    \end{equation*}
    where $(*)$ is a consequence of our assumption that $\tfrac{\beta}{M_{F}} \leq \sqrt{\tfrac{1}{2-\lambda}}$. Next we show that  $\mathrm{(b)}$ holds. We may rewrite $\mathrm{(b)}$ as a quadratic inequality, given by
    \begin{equation}\label{SGR-QS-eqn2}
        (1 - q^{2}) E^{2} -2 \tfrac{\beta}{M_{F}} \gamma^{0} E + (\gamma^{0})^{2} \leq 0.
    \end{equation}
    The discriminant $\Delta = 4(\gamma^{0})^{2}( \tfrac{\beta^{2}}{M^{2}_{F}} - (1-q^{2}) ) > 0$ since $\tfrac{\beta^{2}}{M^{2}_{F}} > 1-q^{2}$. We may then derive two distinct roots defined as
    \begin{equation*}
        E_{1} \triangleq \tfrac{\gamma^{0}}{\tfrac{\beta}{M_{F}}+\sqrt{\tfrac{\beta^{2}}{M^{2}_{F}}-(1-q^{2})}} < E_{2} \triangleq \tfrac{\gamma^{0}}{\tfrac{\beta}{M_{F}}-\sqrt{\tfrac{\beta^{2}}{M^{2}_{F}}-(1-q^{2})}}.
    \end{equation*}
    To ensure that \eqref{SGR-QS-eqn2} holds, we need to show that $E_{1} \leq E \leq E_{2}$ holds. The fact $E \geq E_{1}$ is immediate by definition, $E \geq \tfrac{\gamma^{0}}{\beta/M_{F}} \geq E_{1}$. We proceed to show that $E \leq E_{2}$ is also true. Indeed, if $E = \tfrac{\gamma^{0}}{\beta/M_{F}}$, then $E \leq E_{2}$ also holds. If $E = \tfrac{\mathrm{dist} \: (x^{0}, X^{*})}{M_{F}}$, then $E \leq E_{2}$ also holds by our initialization assumption. Therefore, the inductive claim holds for $k + 1$ and the proof is complete.
\end{proof}

\begin{remark}
    We make two remarks here. First, given $\beta>0$ and sufficiently large $M_{F}>0$, we may always ensure the rate factor is strictly less than one and obtain a linear rate. However, raising $M_{F}>0$ leads to a significant growth in the constant factor in the rate claim. Second, unlike \cite[Theorem 6.1]{davis-drusvyatskiy-macphee-paquette-2018}, we do not need to impose an upper bound requirement on $\gamma^{0}>0$ to ensure that the discriminant is nonnegative. By leveraging this advantage, the initialization condition can be made to hold trivially under compactness, thereby yielding a \emph{global} linear convergence rate, which is an encouraging result. We formalize this result below. $\hfill \Box$
\end{remark}

\begin{corollary}[Global linear convergence]\em
    Consider the same setting as Theorem \ref{SGR-QS}. Suppose $X$ is additionally compact with diameter $D_{X} > 0$. By setting $\gamma^{0} = \tfrac{D_{X}}{M_{F}}$, we have for any $k\geq 0$,
    \begin{equation}\label{SGR-QS-global-linear}
        \mathrm{dist}^{2}(x^{k}, X^{*}) \leq D^{2}_{X} \left( \tfrac{M^{2}_{F}}{\beta^{2}} \right) \left( 1 - (1-\lambda) \tfrac{\beta^{2}}{M^{2}_{F}} \right)^{k}. 
    \end{equation}
\end{corollary}
\begin{proof}
    By setting $\gamma^{0} = \tfrac{D_{X}}{M_{F}}$, we know the initialization condition
    \begin{equation*}
        \mathrm{dist} \: (x^{0}, X^{*}) \leq \tfrac{\gamma^{0} M_{F}}{\tfrac{\beta}{M_{F}} - \sqrt{\tfrac{\beta^{2}}{M^{2}_{F}}-(1-q^{2})}} \leq D_{X}
    \end{equation*}
    holds since $\tfrac{\beta}{M_{F}} < 1$. Similarly, we also have that $\tfrac{(\gamma^{0})^{2}M^{4}_{F}}{\beta^{2}} = \tfrac{D^{2}_{X}}{\beta^{2}/M^{2}_{F}} > D^{2}_{X} \geq \mathrm{dist}^{2}(x^{0}, X^{\ast})$. Then the final result \eqref{SGR-QS-global-linear} follows immediately from \eqref{SGR-QS-main-result}.
\end{proof}

In practice, however, the diameter $D_{X}$ is often unavailable, and verifying the initialization condition is challenging since $X^{*}$ is not known a priori. As discussed in subsection \ref{Sec-2.2}, \ref{QS} implies \ref{QG} under the compactness of $X$, which suggests the possibility of developing the \emph{two-stage} SGR (\textbf{2-SGR}) scheme. Specifically, in the first stage, we implement a sublinearly convergent scheme, whose global sublinear rate is established in Theorem \ref{SSGR-QG}, with the parameters prescribed as
\begin{align}\label{sub_rate} 
    \mathrm{dist}^{2}(x^{k}, X^{*})\leq \tfrac{Q}{k}, 
\end{align} 
for some known $Q > 0$ under \ref{QG}. Once the initialization condition is satisfied for some large $K^{+} > 0$, we enter the second stage and then switch to geometrically decaying stepsizes to achieve the \emph{locally} linear convergence. Note that such a two-stage scheme does not necessitate knowing $X^{\ast}$. We present the \textbf{2-SGR} scheme in Algorithm \ref{2SGR} and test this idea in the numerical section.

\begin{algorithm}[htb]
    \caption{~Two-Stage SGR (\textbf{2-SGR}) Scheme}
    Set $k = 0$. Given positive scalars $\beta, M_F, Q, q, D, \gamma^0, K^+$, where $\beta < M_F$, $Q > 0$ satisfies \eqref{sub_rate}, and $q$ and $D$ are specified in Theorem \ref{SGR-QS}, $\gamma^{0} > 0$, $K^{+} \triangleq \lfloor Q/D \rfloor$. Initialize $x^{0}\in X$ and stepsizes $\{\gamma^{k}\}_{k\geq 0}$. Iterate until $k \geq K$. \\
    \textbf{Strategy update.} For every $i\in [N]$, we update
    \begin{align*}
        x^{k+1}_{i}  = \Pi_{X_{i}} \left[ x^{k}_{i} - \gamma^{k}\nabla_{x_{i}}f_{i}(x^{k}_{i}, x^{k}_{-i}) \right], \mbox{ where } \gamma^{k} = 
        \begin{cases}
            \gamma^{0}/k, & \textrm{if }\, k \leq K^{+} \,\, \mbox{(Stage I)} \\
            \gamma^{0}q^{k-K^{+}}, \hspace{-0.1in} & \textrm{if }\, k > K^+ \,\, \mbox{(Stage II)}
        \end{cases}
    \end{align*}
    \textbf{Return.} $x^{K}$ as final estimate.
    \label{2SGR}
\end{algorithm}

\section{An asynchronous gradient-response scheme and NE computation}\label{Sec-4}

This section is divided into two parts. In subsection \ref{Sec-4.1}, we study the convergence guarantees of \textbf{SAGR} under \ref{AA} and \ref{QG} properties. In subsection \ref{Sec-4.2}, we show that one can obtain convergence guarantees for computing an NE (rather than QNE) under the additional potentiality assumption,  despite the presence of nonconvexity.

\subsection{SAGR under \ref{AA} and \ref{QG}}\label{Sec-4.1}

The \textbf{SAGR} scheme is presented in Algorithm \ref{SAGR}, where at each iteration $k$, only the randomly selected player $i(k)$ updates her strategy. Recall that in game-theoretic settings, the  qualifier ``asynchronous'' is often adopted to capture random player selection, either in gradient or best-response schemes~\cite{koshal-nedic-shanbhag-2016, lei-shanbhag-2020, pang-razaviyayn-2016}. In this subsection, we impose the following assumption, which is the asynchronous counterpart of Assumption $\mathrm{B}$. We denote the histories of \textbf{SAGR} by $\mathcal{F}_{k} \triangleq \sigma \{ x^{0}, \cup_{t=0}^{k-1} \{ i(t), \nabla_{x_{i(t)}} \tilde{f}_{i(t)}(x^{t},\xi_{i(t)}^{t}) \} \}$ and $\mathcal{F}_{k+1/2} \triangleq \mathcal{F}_{k} \cup \{i(k)\}$.

\vspace{5pt}

\emph{Assumption $\mathrm{C}$.} (C1) For any $k \geq 0$, we have $\mathbb{E}[ w_{i(k)}^{k} \!\mid\! \mathcal{F}_{k+1/2}] = 0$, where $w^{k}_{i(k)} = \nabla_{x_{i(k)}}\Tilde{f}_{i(k)}(x^{k}_{i(k)}, x^{k}_{-i(k)}, \xi^{k}_{i(k)}) - \nabla_{x_{i(k)}} \mathbb{E}[ \Tilde{f}_{i(k)}(x^{k}_{i(k)}, x^{k}_{-i(k)}, \bxi) \!\mid\! \mathcal{F}_{k+1/2} ]$. (C2) For any $k \geq 0$, the second moment bound $\mathbb{E}[ \|w^{k}_{i(k)}\|^2 \!\mid\! \mathcal{F}_{k+1/2} ] \leq M^{2}_{w, i(k)}$ holds for some $M_{w, i(k)} > 0$. (C3) For any $i \in [N]$, the bound $\| \nabla_{x_{i}} f_{i}(x_{i}, x_{-i}) \| \leq M_{F, i}$ holds for any $x\in X$ and some $M_{F, i} > 0$. $\hfill \Box$

\vspace{5pt}

\begin{algorithm}[htb]
    \caption{~Stochastic Asynchronous GR ({\bf SAGR}) Scheme}
    Set $k = 0$. Initialize $x^{0}\in X$ and stepsizes $\{\gamma^{k}\}_{k\geq 0}$. Iterate until $k\geq K$. 
    \\\textbf{Player selection.} Pick player $i(k)$ with probability $p_{i(k)} > 0$ where $\sum_{i=1}^{N} p_{i} = 1$.
    \\\textbf{Strategy update.} Player $i(k)$ updates strategy as follows:
    \begin{equation*}
        x^{k+1}_{i(k)} \triangleq \Pi_{X_{i(k)}} \left[ x^{k}_{i(k)} - \gamma^{k}_{i(k)}\nabla_{x_{i(k)}}\Tilde{f}_{i(k)}(x^{k}_{i(k)}, x^{k}_{-i(k)}, \xi_{i(k)}^{k}) \right].
    \end{equation*}
    \textbf{Return.} $x^{K}$ as final estimate.
    \label{SAGR}
\end{algorithm}

In this paper, we focus on the case where a single player is selected at random and updates her strategy at each iteration. Nevertheless, our results can be extended to allow simultaneous updates by a group of players, as well as delayed updates.  The \textbf{SAGR} recursions are more complicated than \textbf{SSGR} recursions. Naturally, deterministic stepsizes are harder to prescribe, leading to random stepsizes used in \textbf{SAGR} as presented in \cite{koshal-nedic-shanbhag-2016, nedic-2011}. To this end, for any $i\in [N]$ and any $k\geq 0$, we define the randomized stepsize $\gamma^k_i$ as
\begin{equation}\label{asyn_gamma}
    \gamma^{k}_{i} \triangleq \begin{cases}
        1/\Gamma_{k}(i), &\textrm{if } \, \Gamma_{k}(i)\neq 0,\\
        0, &\textrm{if } \, \Gamma_{k}(i) = 0,
    \end{cases}
\end{equation}
where $\Gamma_{k}(i(k))$ denotes the number of updates that player $i(k)$ (the player chosen at time $k$) has performed until and including the $k$th iteration. In the first iteration $k = 0$, one player, namely $i(0)$ is selected, implying that $\Gamma_{0}(i(0)) = 1$ while $\Gamma_{0}(j) = 0$ for all $j \ne i(0)$. Since player $j \ne i(0)$ does not update during iteration $0$, $\gamma_j^0 = 0$. For sufficiently large $k$ (each player has been selected at least once), we must have $\gamma^{k}_{i} = 1/\Gamma_{k}(i)$. This leads to an additional source of uncertainty, complicating the analysis. Similar to Lemma \ref{SSGR-recursions}, we next derive two crucial \textbf{SAGR} recursions. Again, note that $\Pi_{X^{*}}[x^{k}]$ may not be unique.

\begin{lemma}[SAGR recursions]\label{SAGR-recursions}\em
    Consider the $N$-player game $\mathcal{G}({\bf f},X,\bxi)$. Suppose that Assumptions $\mathrm{A}$ and $\mathrm{C}$ hold. Consider the sequence of iterates $\{x^{k}\}_{k=0}^{\infty}$ generated by \textbf{SAGR} and suppose that the stepsizes $\{ \gamma^{k}_{i(k)} \}_{k\geq 0}$ are defined as in \eqref{asyn_gamma}. Let $p_{i} > 0$ denote the probability that player $i$ is selected. Then the following hold.

    \noindent (i) Let $x^{*}\in X^{*}$ be any QNE. Then for any $k \geq 0$, we have
    \allowdisplaybreaks
    \begin{align*}
        \mathbb{E} [ \| x^{k+1} - x^{\ast} \|^2 \mid \mathcal{F}_{k} ] &\leq (1 + \max_{i} p_i|\gamma^{k}_{i}-\tfrac{1}{kp_{i}}|)\| x^{k} - x^{\ast} \|^2 + 2 \sum_{i=1}^{N} p_i(\gamma^{k}_{i})^{2} M^{2}_{w, i} \notag \\
        & ~~~+ \sum_{i=1}^{N} p_i( 2 (\gamma^{k}_{i})^{2} + |\gamma^{k}_{i}-\tfrac{1}{kp_{i}}|) M^{2}_{F, i} - \tfrac{2}{k}(x^{k}-x^{\ast})^{\top}F(x^{k}).
    \end{align*}

    \noindent (ii) Let $x^{k, \ast} \in \Pi_{X^{*}}[x^{k}]$. Then for any $k \geq 0$, we have
    \allowdisplaybreaks
    \begin{align*}
        \mathbb{E} [ \mathrm{dist}^{2}(x^{k+1}, X^{\ast}) \mid \mathcal{F}_{k} ] &\leq (1 + \max_{i} p_i|\gamma^{k}_{i}-\tfrac{1}{kp_{i}}|) \: \mathrm{dist}^{2}(x^{k}, X^{\ast}) + 2 \sum_{i=1}^{N} p_i(\gamma^{k}_{i})^{2} M^{2}_{w, i} \\
        & ~~~+ \sum_{i=1}^{N} p_i( 2 (\gamma^{k}_{i})^{2} + |\gamma^{k}_{i}-\tfrac{1}{kp_{i}}|) M^{2}_{F, i} - \tfrac{2}{k}(x^{k}-x^{k, \ast})^{\top}F(x^{k}).
    \end{align*}
\end{lemma}
\begin{proof}
    (i) Similar to the proof of Lemma \ref{SSGR-recursions}-(i), we can obtain that
    \allowdisplaybreaks
    \begin{align*}
        \mathbb{E}[ & \| x^{k+1}_{i(k)}-x^{\ast}_{i(k)} \|^2 \mid \mathcal{F}_{k+1/2} ] \leq \| x^{k}_{i(k)}-x^{\ast}_{i(k)} \|^2 \\
        & - 2\gamma^{k}_{i(k)}(x^{k}_{i(k)}-x^{\ast}_{i(k)})^{\top}\nabla_{x_{i(k)}}\mathbb{E}[\Tilde{f}_{i(k)}(x^{k}_{i(k)},x^{k}_{-i(k)},\bxi)] \\
        & + 2(\gamma^{k}_{i(k)})^{2}\mathbb{E}[\|w^{k}_{i(k)}\|^{2} \mid \mathcal{F}_{k+1/2} ] + 2(\gamma^{k}_{i(k)})^{2}\|\nabla_{x_{i(k)}}\mathbb{E}[\Tilde{f}_{i(k)}(x^{k}_{i(k)},x^{k}_{-i(k)},\bxi)]\|^{2}.
    \end{align*}
    By using $\gamma^{k}_{i(k)} = (\gamma^{k}_{i(k)}-\tfrac{1}{kp_{i(k)}}) + \tfrac{1}{kp_{i(k)}}$ and invoking Assumption $\mathrm{C}$, we obtain
    \allowdisplaybreaks
    \begin{align*}
        \mathbb{E} & [\| x^{k+1}_{i(k)}-x^{\ast}_{i(k)} \|^2 \mid \mathcal{F}_{k+1/2} ] \leq \| x^{k}_{i(k)}-x^{\ast}_{i(k)} \|^2  + 2(\gamma^{k}_{i(k)})^{2} \mathbb{E}[\|w^{k}_{i(k)}\|^{2} \mid \mathcal{F}_{k+1/2} ]\\
        & + 2(\gamma^{k}_{i(k)})^{2}\|\nabla_{x_{i(k)}}\mathbb{E}[{\Tilde{f}_{i(k)}(x^{k}_{i(k)},x^{k}_{-i(k)},\bxi)]\|^{2}} \\
        & -\tfrac{2}{kp_{i(k)}}(x^{k}_{i(k)}-x^{\ast}_{i(k)})^{\top}\nabla_{x_{i(k)}}\mathbb{E}[\Tilde{f}_{i(k)}(x^{k}_{i(k)},x^{k}_{-i(k)},\bxi)] \\
        & - 2(\gamma^{k}_{i(k)}-\tfrac{1}{kp_{i(k)}})(x^{k}_{i(k)}-x^{\ast}_{i(k)})^{\top}\nabla_{x_{i(k)}}\mathbb{E}[\Tilde{f}_{i(k)}(x^{k}_{i(k)},x^{k}_{-i(k)},\bxi)] \\
        & \leq (1 + | \gamma^{k}_{i(k)}-\tfrac{1}{kp_{i(k)}} |) \| x^{k}_{i(k)}-x^{\ast}_{i(k)} \|^2 + 2(\gamma^{k}_{i(k)})^{2} {M^{2}_{w, i(k)}}  + 2(\gamma^{k}_{i(k)})^{2} {M^{2}_{F, i(k)}} \\
        & + |\gamma^{k}_{i(k)}-\tfrac{1}{kp_{i(k)}}| {M^{2}_{F, i(k)}} 
        -\tfrac{2}{kp_{i(k)}}(x^{k}_{i(k)}-x^{\ast}_{i(k)})^{\top}\nabla_{x_{i(k)}}\mathbb{E}[\Tilde{f}_{i(k)}(x^{k}_{i(k)},x^{k}_{-i(k)},\bxi)].
    \end{align*}
    Consequently, it follows that
    \allowdisplaybreaks
    \begin{align*}
        & \mathbb{E}[ \|x^{k+1}_i - x^\ast_i\|^2 \mid \mathcal{F}_{k} ] \\
        =\: & p_i \: \mathbb{E}[ \|x^{k+1}_i - x^\ast_i\|^2 \mid \mathcal{F}_{k} \cup \{i(k) = i\} ] + (1-p_i) \: \mathbb{E}[ \|x^{k+1}_i - x^\ast_i\|^2 \mid \mathcal{F}_{k} \cup \{i(k) \neq i\} ] \\
        =\: & p_i \: \mathbb{E}\left[ \|x^{k+1}_i - x^\ast_i\|^2 \mid \mathcal{F}_{k+1/2} \right] + (1-p_i) \: \|x^{k}_i - x^\ast_i\|^2 \\
        \leq\: & (1 + p_i|\gamma^{k}_{i}-\tfrac{1}{kp_{i}}|) \| x^{k}_{i}-x^{\ast}_{i} \|^2 + 2 p_i(\gamma^{k}_{i})^{2} M^{2}_{w,i} + p_i( 2 (\gamma^{k}_{i})^{2} + |\gamma^{k}_{i}-\tfrac{1}{kp_{i}}|) M^{2}_{F,i} \\
        &-\tfrac{2}{k}(x^{k}_{i}-x^{\ast}_{i})^{\top}\nabla_{x_{i}}\mathbb{E}[\Tilde{f}_{i}(x^{k}_{i},x^{k}_{-i},\bxi)].
    \end{align*}
    By summing the inequality over $i\in [N]$, we obtain the desired result. Similarly, (ii) follows from the definition of $x^{k, \ast}$, choosing $x^\ast = x^{k, \ast}$, noting that $\mathrm{dist}^2 (x^{k}, X^{*}) = \|x^{k} - x^{k, \ast}\|^2$,  and invoking the claim that $\mathrm{dist}^{2} (x^{k+1}, X^{*}) \leq \|x^{k+1}-x^{k, \ast}\|^2$.
\end{proof}

Before discussing the convergence of \textbf{SAGR} under \ref{AA} and \ref{QG}, we first examine the asymptotics of $\gamma^{k}_{i}$ given in \eqref{asyn_gamma}. In contrast to \cite[Lemma 3]{nedic-2011} and \cite[Lemma 7]{koshal-nedic-shanbhag-2016}, where distributed settings over communication graphs are considered, we provide the centralized counterpart without an underlying graph; see the appendix for the proof.

\begin{lemma}\label{SAGR-stepsize-asymptotics}\em
    Let $\gamma^{k}_{i}$ be defined in (\ref{asyn_gamma}) for any $i\in [N]$ and any $k\geq 0$. Let $\Tilde{d} \in (0, \tfrac{1}{2})$ and $p_{i}$ denote the probability that player $i$ is selected. Define $p_{\min} \triangleq {\displaystyle \min_{i\in [N]}} p_{i} > 0$. Then there exists some sufficiently large $\Tilde{k} \triangleq \Tilde{k}(\Tilde{d}, N)$ such that for any $i\in [N]$ and any $k\geq \Tilde{k}$, with probability one: (i) $\gamma^{k}_{i} \leq \tfrac{2}{k p_{i}}$; (ii) $(\gamma^{k}_{i})^{2} \leq \tfrac{4}{k^{2} p^{2}_{\min}}$; and (iii) $| \gamma^{k}_{i} - \tfrac{1}{k p_{i}} | \leq \tfrac{3}{k^{3/2-\Tilde{d}} p^{2}_{\min}}$. $\hfill \Box$
\end{lemma}

Based on the above asymptotics lemma, we next derive the convergence guarantees of \textbf{SAGR} under \ref{AA} and \ref{QG}. 

\begin{theorem}[SAGR under \ref{AA}]\label{SAGR-AA}\em 
    Consider the $N$-player game $\mathcal{G}({\bf f},X,\bxi)$. Suppose that Assumptions $\mathrm{A}$ and $\mathrm{C}$ hold, and that the \ref{AA} property is satisfied. Consider the sequence $\{x^k\}_{k=0}^{\infty}$ generated by \textbf{SAGR}, where the stepsizes $\{\gamma^{k}\}_{k=0}^{\infty}$ are defined in \eqref{asyn_gamma}. If $X^{\ast}$ is additionally compact, then \emph{some subsequence} of $\{x^{k}\}_{k=0}^{\infty}$ converges to a QNE a.s.
\end{theorem}
\begin{proof}
    We may derive the recursion in Lemma \ref{SAGR-recursions}-(i). Let $\Tilde{d} \in (0, \tfrac{1}{2})$. By Lemma \ref{SAGR-stepsize-asymptotics}-(iii), for any random replication and an associated sufficiently large $\Tilde{k}(\omega)$, we have
\begin{align*}
    1 & = \mathbb{P}\left[ \omega\in \Omega \mid | \gamma_i^k-\tfrac{1}{kp_i} | \leq \tfrac{3}{k^{3/2-\Tilde{d}}p^{2}_{\min}} \textrm{ for } k \geq \Tilde{k}(\omega) \right] \\
    & = \mathbb{P}\left[ \omega\in \Omega \,\Bigg\vert\, \sum_{k=1}^{\infty}| \gamma_i^k-\tfrac{1}{k p_{i}}| \leq \sum_{k=1}^{\Tilde{k}(\omega)}|\gamma_i^k-\tfrac{1}{k p_{i}}| + \sum_{k=\Tilde{k}(\omega)+1}^{\infty} \tfrac{3}{k^{3/2-\Tilde{d}} p^{2}_{\min}} < \infty \right].
\end{align*}
In other words, $\sum_{k=1}^{\infty}| \gamma_i^k - \tfrac{1}{k p_{i}}| < \infty$ holds a.s. for any $i\in [N]$. Similarly, by Lemma \ref{SAGR-stepsize-asymptotics}-(ii), we can obtain $\sum_{k=1}^{\infty} (\gamma^{k}_{i})^{2} < \infty$ holds a.s. for any $i\in [N]$. Therefore, we have
\begin{equation*}
    \sum_{k=1}^{\infty} \max_{i} p_i| \gamma^{k}_{i}-\tfrac{1}{kp_{i}}| \leq \sum_{k=1}^{\infty} \max_{i} | \gamma^{k}_{i}-\tfrac{1}{kp_{i}}| < \infty \; \textrm{ and } \; \sum_{k = 1}^{\infty} p_i (\gamma^{k}_{i})^{2} \leq \sum_{k = 1}^{\infty} (\gamma^{k}_{i})^{2} < \infty.
\end{equation*}
The remaining part of the proof is similar to Theorem~\ref{SSGR-AA} and is omitted. 
\end{proof}

Only \emph{subsequential} a.s. convergence is available for \textbf{SAGR} scheme under \ref{AA}. We now present rate guarantees for \textbf{SAGR} under \ref{QG}. Before that, we derive a generalization of Chung's lemma \cite[Chapter 2.2]{polyak-1987}; see the appendix for the proof.

\begin{lemma}\label{Chung}\em
    Suppose for any $k$, $e_k\ge 0$ and $e^{k+1} \leq (1-\tfrac{C}{k}+\tfrac{A}{k^{1+t}}) e^{k} + \tfrac{B}{k^{t+1}}$ for any $k\geq 0$, where $A, B, C, t>0$.
    Then we have 
    \begin{equation*}
        e^{k} \begin{cases}
            \leq (A+B)(C-t)^{-1}k^{-t} + o\left( k^{-t} \right), &\textrm{if } C>t, \\
            = \mathcal{O}\left( k^{-C}\log{k} \right), &\textrm{if } C=t, \\
            = \mathcal{O}\left( k^{-C} \right), &\textrm{if } C<t.
        \end{cases}
    \end{equation*}
    for sufficiently large $k > 0$. $\hfill \Box$
\end{lemma}

\begin{theorem}[SAGR under \ref{QG}]\label{SAGR-QG}\em 
    Consider the $N$-player game $\mathcal{G}({\bf f}, X, \bxi)$. Suppose that
Assumptions $\mathrm{A}$ and $\mathrm{C}$ hold, and that the \ref{QG} property
is satisfied. Consider the sequence $\{x^k\}_{k=0}^{\infty}$ generated by
\textbf{SAGR}, where the stepsizes $\{\gamma^{k}\}_{k=0}^{\infty}$ are defined
in \eqref{asyn_gamma}. Then we have (i) ${\displaystyle \lim_{k\to\infty}}
\mathrm{dist}\:(x^{k}, X^{*}) = 0$ a.s.; and (ii) for sufficiently large $k$,
    \begin{equation*}
        \mathbb{E} \left[\, \mathrm{dist}^{2}(x^{k+1}, X^{\ast}) \, \right] \begin{cases}
            \leq \tfrac{3+8\sum_{i=1}^{N} M^{2}_{w, i} + 11\sum_{i=1}^{N} M^{2}_{F, i}}{p^{2}_{\min}(2\alpha-1/2+\Tilde{d})k^{1/2-\Tilde{d}}} + o( k^{-1/2+\Tilde{d}} ), \hspace{-0.1in} &\textrm{if } 2\alpha > 1/2-\Tilde{d}, \\
            = \mathcal{O}\left( k^{-2\alpha}\log{k} \right), &\textrm{if } 2\alpha = 1/2-\Tilde{d}, \\
            = \mathcal{O}\left( k^{-2\alpha} \right), &\textrm{if } 2\alpha < 1/2-\Tilde{d}.
        \end{cases}
    \end{equation*}
    where $\alpha > 0$ is the \ref{QG} parameter, $\Tilde{d} \in (0, \tfrac{1}{2})$, $p_{\min} = \min_{i\in [N]}p_{i}$, and $M_{w, i} > 0$ and $M_{F, i} > 0$ are given in Assumption $\mathrm{C}$.
\end{theorem}
\begin{proof}
\allowdisplaybreaks
     By invoking the \ref{QG} property in the \textbf{SAGR} recursion given in Lemma \ref{SAGR-recursions}-(ii), it leads to
     \begin{align*}
        \mathbb{E} [ \mathrm{dist}^{2}(x^{k+1}, X^{\ast}) \mid \mathcal{F}_{k} ] &\leq (1 - \tfrac{2\alpha}{k} + \max_{i} p_i|\gamma^{k}_{i}-\tfrac{1}{kp_{i}}|) \: \mathrm{dist}^{2}(x^{k}, X^{\ast}) \\
        & ~~~+ 2 \sum_{i=1}^{N} p_i(\gamma^{k}_{i})^{2} M^{2}_{w, i} + \sum_{i=1}^{N} p_i( 2 (\gamma^{k}_{i})^{2} + |\gamma^{k}_{i}-\tfrac{1}{kp_{i}}|) M^{2}_{F, i}.
    \end{align*}
    By Lemma \ref{SAGR-stepsize-asymptotics} and the fact that $p_{i} \leq 1$, for sufficiently large $k$, we have
    \begin{equation*}
        \mathbb{E} [ \mathrm{dist}^{2}(x^{k+1}, X^{\ast}) \mid \mathcal{F}_{k} ] \leq (1 - \tfrac{2\alpha}{k} + \tfrac{3}{k^{3/2-\Tilde{d}}p^{2}_{\min}}) \: \mathrm{dist}^{2}(x^{k}, X^{\ast}) + \tfrac{8\sum_{i=1}^{N} M^{2}_{w, i} + 11\sum_{i=1}^{N} M^{2}_{F, i}}{k^{3/2-\Tilde{d}}p^{2}_{\min}}.
    \end{equation*}
    Then (i) follows from Lemma \ref{three-useful-lemma}-(ii). By taking unconditional expectations on both sides of the above recursion and invoking Lemma \ref{Chung}, we can show (ii).
\end{proof}

\begin{remark}
    We make two remarks here. First, it follows from subsection \ref{Sec-2.2} that strong pseudomonotonicity \eqref{SP} is a sufficient condition for \ref{QG} when $X^{\ast}$ is a singleton. Therefore, the claims of Theorems \ref{SSGR-QG} and \ref{SAGR-QG} also hold under \ref{SP} when $X^{\ast}$ is a singleton. Second, the minimax problems are closely connected to variational inequalities, and the alternating gradient method (AGM) \cite{xu-zhang-xu-lan-2023} can be viewed as a two-player asynchronous gradient update scheme. However, existing convergence guarantees for AGM are limited to nonconvex-concave and convex-nonconcave regimes \cite{xu-zhang-xu-lan-2023}. The convergence behavior of AGM for nonconvex-nonconcave minimax problems remains largely unexplored. We intend to establish its convergence and rate guarantees under suitable growth and sharpness conditions in our future work. $\hfill \Box$
\end{remark}

\subsection{Computing NE for nonconvex potential games}\label{Sec-4.2}
In this subsection, we show that under the potentiality property, convergence and rate guarantees for the computation of NE (rather than QNE) can be provided, despite the presence of nonconvexity. Recall from assumption $\mathrm{(A1)}$ that, for given $x_{-i}\in X_{-i}$, each $f_i(\bullet, x_{-i})$ is C$^1$ on an open set $\mathcal{O}_{i} \supseteq X_i$. We say an $N$-player game is a potential game \cite{monderer-shapley-1996} if there exists a function $\mathcal{P}: \mathcal{O} \triangleq \Pi_{i=1}^{N} \mathcal{O}_{i} \to \mathbb{R}$ such that for any $i\in [N]$, we have 
\begin{align}\label{potential-def}
    f_{i}(x_{i}, x_{-i}) - f_{i}(y_{i}, x_{-i}) = \mathcal{P}(x_{i}, x_{-i}) - \mathcal{P}(y_{i}, x_{-i})
\end{align}
for any $x^{1}_{i}, x^{2}_{i}\in \mathcal{O}_{i}$ and any $x_{-i}\in X_{-i}$. Consequently, we have that $\nabla_{x} {\cal P}(x) = (\nabla_{x_i} f_i(x))_{i=1}^N$ by \cite[Proposition 3.1]{xiao-2026}. Recall that we say a smooth function ${\cal P}$ is pseudoconvex on $X$ \cite[Definition 3.2]{karamardian-schaible-1990} if $\nabla {\cal P}(x)^{\top}(y-x) \geq 0 \implies {\cal P}(y) \ge {\cal P}(x)$ for any $x,y \in X$. By imposing a potentiality property with a C$^1$ and pseudoconvex function ${\cal P}$, it can be shown that a QNE is indeed an NE. We formalize it below.

\begin{proposition}\label{QNE-implies_NE}\em
    Consider a smooth $N$-player game $\mathcal{G}({\bf f}, X, \bxi)$. Suppose this game admits a potential function ${\cal P}$. The following hold if  ${\cal P}$ is C$^{1}$ and pseudoconvex:
    \begin{equation*}
        x^{*} \mbox{ is a } \mbox{QNE} \implies x^{*} \mbox{is a B-stationary point} \mbox{ of } \min_{x\in X} {\cal P}(x) \implies  x^{*} \mbox{ is an } \mbox{NE}.
    \end{equation*}
\end{proposition}
\begin{proof}
    The first implication follows immediately from $\nabla_{x} {\cal P}(x) = \big(\nabla_{x_i} f_i(x)\big)_{i=1}^N$; see \cite[Proposition 3.1]{xiao-2026}. We consider the second implication. By definition, we have that $\nabla {\cal P}(x^*)^{\top}(x-x^*) \geq 0$ holds for any $x \in X$. From pseudoconvexity of $\mathcal{P}$ on $X$,  ${\cal P}(x) \geq {\cal P}(x^*)$, and hence ${\cal P}(x) - {\cal P}(x^*)\geq 0$ for any $x\in X$. Therefore, by the potentiality assumption, we have that
    \begin{align*}
        {\cal P}(x_{i}, x^*_{-i}) - {\cal P}(x^*_{i}, x^*_{-i})\geq 0 \implies f(x_{i}, x^*_{-i}) - f(x^*_{i}, x^*_{-i})\geq 0
    \end{align*}
    holds for any $x_{i}\in X_{i}$ and any $i\in [N]$, which implies that $x^{\ast}$ is an NE.
\end{proof}

Consequently, we may provide the following claim for the computation of an NE of the potential nonconvex $N$-player game $\mathcal{G}({\bf f}, X, \bxi)$.

\begin{theorem}[NE convergence]\label{NE-nonconvex-games}\em
    Consider the $N$-player game $\mathcal{G}({\bf f}, X, \bxi)$. Suppose this game admits a C$^1$ and pseudoconvex potential function ${\cal P}$. Then under the same required assumptions, all prior results established for QNE also apply to NE. $\hfill \Box$
\end{theorem}

\section{A stochastic asynchronous modified GR (SAMGR) scheme}\label{Sec-5}
In this section, we focus on computing an NE rather than a QNE in the setting where each $f_i(\bullet,x_{-i})$ is C$^1$ and convex for any $i\in [N]$, but the concatenated gradient map $F$ is not necessarily monotone  and prior growth and sharpness properties do not hold. Note that the failure of monotonicity of $F$ is closely tied to the cross-player interaction. We build a smoothing-enabled framework reliant on minimizing the regularized gap function $\theta_{c}$ of $\mathrm{SVI}\:(X, F)$. We provide some preliminaries on the regularized gap function and randomized smoothing in subsection \ref{Sec-5.1}, and analyze the rate and complexity guarantees of the \textbf{SAMGR} scheme in subsection \ref{Sec-5.2}.

\subsection{Regularized gap function and randomized smoothing}\label{Sec-5.1}

\subsubsection{Regularized gap function}
We consider $\mathrm{SVI}\:(X,F)$, where $X$ is closed and convex. For a given $c > 0$, the associated regularized gap function \cite[Definition 10.2.2]{facchinei-pang-2003}, denoted by $\theta_{c}$, is defined as
\begin{equation}\label{regularized-gap-function}
    \theta_{c}(x) \triangleq  \sup_{y \in X} \left( F(x)^\top (x-y) - \tfrac{c}{2}(x-y)^{\top}(x-y) \right).
\end{equation} 

We recall some properties of the regularized gap function $\theta_{c}(x)$.

\begin{lemma}[\mbox{\cite[Proposition 3.1]{li-ng-2009}} \& \mbox{\cite[Theorem 10.2.3]{facchinei-pang-2003}}]\label{lemma-regularized-gap-function}\em
    Consider $\mathrm{SVI}\:(X,F)$ where $X$ is closed and convex and $F(x) = \mathbb{E}[\Tilde{F}(x, \bxi)]$. Suppose that the associated regularized gap function is $\theta_{c}$ for given $c > 0$. Then the following hold.
    
    (i) For any $x\in X$, there exists a unique vector $y_{c}(x)\in X$ at which the supremum in \eqref{regularized-gap-function} is attained at $y_{c}(x) = \Pi_{X}[x-\tfrac{1}{c}F(x)]$, based on 
    \begin{align}
        y_{c}(x) \, &\triangleq \, \argmax_{y\in X} \, \left(\, F(x)^\top (x-y) - \tfrac{c}{2}(x-y)^\top (x-y) \,\right) \,\label{yc-x-1} \\
        &= \, \argmin_{y\in X} \, \mathbb{E}_{\bxi} \left[ \underbrace{\Tilde{F}(x, \bxi)^\top (y-x) + \tfrac{c}{2}(y-x)^\top (y-x)}_{\triangleq \, \widehat{F}(x, y, \bxi) } \right]. \label{yc-x-2}
    \end{align}
    
    (ii) $y_{c}$ and $\theta_{c}$ are continuous on $\mathbb{R}^{n}$ and $\theta_{c}(x)\geq 0$ for all $x\in X$.

    (iii) $\left[ \: x\, \in \, X,\:\theta_{c}(x)\, =\, 0 \: \right] \: \iff \:  \left[ \: x \, =\, y_{c}(x) \: \right] \: \iff \: \left[ \: x\, \in \, \mathrm{SOL}\:(X,F) \: \right]$.

    (iv) If $F$ is locally Lipschitz on $\mathbb{R}^{n}$, then $y_{c}$ and $\theta_{c}$ are locally Lipschitz on $\mathbb{R}^{n}$.
    
    (v) If $F$ is $\mathrm{C}^1$, then $\theta_{c}$ is $\mathrm{C}^1$ and $\nabla \theta_{c}$ is defined as 
    \begin{equation}\label{regularized-gap-function-gradient}
        \nabla\theta_{c}(x) = F(x) + \mathbf{J} F(x)^{\top} (x-y_{c}(x)) - c(x-y_{c}(x)),
    \end{equation}
    where $\mathbf{J} F(x)$ denotes the Jacobian of $F$ at $x$. $\hfill \Box$
\end{lemma}

By \eqref{yc-x-1}, the regularized gap function $\theta_{c}$ in \eqref{regularized-gap-function} can be written as
\begin{align}\label{gap-function}
    \theta_{c}(x) &= F(x)^\top (x-y_{c}(x)) - \tfrac{c}{2}(x-y_{c}(x))^{\top}(x-y_{c}(x)) \notag \\
    &= \mathbb{E}_{\bxi} \left[ \underbrace{\Tilde{F}(x, \bxi)^\top (x-y_{c}(x)) - \tfrac{c}{2}(x-y_{c}(x))^{\top}(x-y_{c}(x))}_{\triangleq \, \Tilde{\theta}_{c}(x, y_{c}(x), \bxi)} \right]. 
\end{align}
Observe that the regularized gap function $\theta_{c}$ may lose convexity. Furthermore, if $\mathbf{J}F(x)$ exists for any $x$, then $\nabla \theta_{c}(x)$ is well defined. By Lemma \ref{lemma-regularized-gap-function} (ii)-(iii), our goal lies in minimizing $\theta_{c}$ over $X$ with the intent of finding  feasible zeros of $\theta_c$ (i.e., which are global minimizers of $\theta_c$ on $X$). However, gradient-based methods will \emph{at best} provide guarantees for computing a stationary point of $\theta_{c}$ over $X$, i.e., $0\in \nabla \theta_{c}(x) +
{\cal N}_{X}(x)$, where ${\cal N}_{X}(x)$ denotes the normal cone to $X$ at $x$. But under some condition, a stationary point of $\theta_c(x)$ is indeed a feasible zero, and hence a solution of $\mathrm{VI}\:(X,F)$.

\begin{definition}[\mbox{\cite[Definition 2.5.1]{facchinei-pang-2003}}]
    We say a matrix $M$ is strictly copositive on a cone $C$ if for all $x\in C\setminus \{0\}$ we have $x^{\top}Mx>0$.
\end{definition}

\begin{lemma}[\mbox{\cite[Corollary 10.2.7]{facchinei-pang-2003}}]\label{copositivity}\em
    Consider $\mathrm{VI}\:(X,F)$ where $X$ is closed and convex. Suppose that the mapping $F$ is $\mathrm{C}^1$ on $O\supseteq X$. Suppose $x^{\ast}$ is a stationary point of $\theta_{c}(x)$ on $X$, i.e., $0\in \nabla \theta_{c}(x^{\ast}) + \mathcal{N}_{X}(x^{\ast})$. If $\mathbf{J} F(x^{\ast})$ is strictly copositive on $\mathcal{T}_{X}(x^{\ast})\cap (-F(x^{\ast}))^{\ast}$, where $\mathcal{T}_{X}(x^{\ast})$ denotes the tangent cone to $X$ at $x^{\ast}$ and the dual cone $(-F(x^{\ast}))^{\ast}$ is defined as $(-F(x^{\ast}))^{\ast} \triangleq \{ d \mid d^{\top}F(x^{\ast})\leq 0 \}$, then $x^{\ast}$ is an isolated solution of $\mathrm{VI}\:(X,F)$. $\hfill \Box$
\end{lemma}

\begin{remark}
    In contrast to growth and sharpness conditions, strict copositivity need not hold for every $x \in X$. It is instead a local property and it suffices to examine $\mathbf{J}F(x^{\ast})$. By \cite[Proposition 2.3.2]{facchinei-pang-2003}, the local \emph{strong} monotonicity of $F$ around $x^{\ast}$ provides a sufficient condition for the strict copositivity of $\mathbf{J}F(x^{\ast})$. An easily verifiable case arises when each $X_i$ is a polyhedral set of the form $X_i \triangleq \{x_i \in \mathbb{R}^{n_{i}} \mid {\bf A}_i x_i \leq {\bf b}_i\}$, where ${\bf A}_i\in \mathbb{R}^{m_i\times n_i}$ and ${\bf b}_i\in \mathbb{R}^{m_i}$. Consequently, by definition, we have $X \triangleq \left\{x \in \mathbb{R}^{n} \mid {\bf A}x\leq {\bf b}\right\}$, where
    \begin{equation*}
        {\bf A} \triangleq \pmat{
        {\bf A}_1 &  & \\
        & \ddots & \\
        & & {\bf A}_N} \mbox{ and } {\bf b} \triangleq \pmat{
        {\bf b}_1 \\ \vdots \\ {\bf b}_N}.
    \end{equation*}
    For a given $x\in X$, define the associated active set as ${\cal I}(x)\triangleq \{j \mid a_j^\top x=b_j\}$, where $a_{j}^{\top}$ denotes the $j$th row of ${\bf A}$ and $b_{j}$ denotes the $j$th element of ${\bf b}$. Then, by \cite[p.\,592]{rockafellar-wets-1998}, the tangent cone to $X$ at $x$ is given by ${\cal T}_{X}(x) = \{v\in \mathbb{R}^n \mid a_j^\top v\leq 0, ~ \forall j\in {\cal I}(x)\}$. Therefore, if ${\bf J}F(x^\ast)$ is strictly copositive on the polyhedral cone ${\cal T}_X(x^\ast)$, then it is also strictly copositive on the smaller intersection ${\cal T}_{X}(x^\ast)\cap (-F(x^\ast))^\ast$. It follows that $x^\ast$ is an isolated solution of $\mathrm{VI}\:(X, F)$. $\hfill \Box$
\end{remark}

We conclude by defining the stationarity residual $G_{t}$ of $\theta_{c}$ over $X$ as
\begin{equation}\label{residual-mapping}
    G_{t}(x) \triangleq t \left( x-\Pi_{X} [x - t^{-1} \nabla \theta_{c}(x) ] \right)
\end{equation}
for any $t > 0.$
By \cite[p.\,214]{beck-2023}, $G_{t}(x) = 0$ if and only if $x$ is a stationary point of $\theta_{c}(x)$ over $X$, i.e., $0\in \nabla \theta_{c}(x) + \mathcal{N}_{X}(x)$. Note that $G_t$ will be used as a metric of convergence later in this section.

\subsubsection{Randomized smoothing}

We now discuss a randomized smoothing framework that allows for constructing tractable gradient estimators of $\theta_c$. Note that $\nabla \theta_c(x)$ in \eqref{regularized-gap-function-gradient} is inaccessible, given its dependence on the expectation-valued matrix $\mathbf{J}F(x)$. To this end, we consider randomized smoothing~\cite{steklov-1907} of $\theta_c$ with $\eta > 0$, with the smoothed counterpart $\theta^{\eta}_{c}$ defined as
\begin{equation*}
    \theta_{c}^{\eta}(x) \triangleq \mathbb{E}_{\mathbf{u}\in \mathbb{B}} \left[ \, \theta_{c}(x+\eta \mathbf{u}) \, \right] \overset{\eqref{gap-function}}{=} \mathbb{E}_{\uu, \bxi} \left[ \, \Tilde{\theta}_{c}(x+\eta \uu, y_{c}(x+\eta \uu), \bxi) \, \right],
\end{equation*}
where $\mathbb{B}$ denotes the unit ball and $\mathbf{u}$ is uniformly distributed over $\mathbb{B}$. We first recall some basic properties of randomized smoothing, where we denote the surface of $\mathbb{B}$ by $\mathbb{S}$ and the Minkowski sum of $X$ and $\eta \mathbb{B}$ by $X_{\eta} \triangleq X+\eta \mathbb{B}$.

\begin{lemma}[\mbox{\cite[Lemma 2.4]{marrinan-shanbhag-yousefian-2026}}]\label{RS-property-grad}\em
    Consider the regularized gap function $\theta_{c}$ and its randomized smoothing $\theta_{c}^{\eta}$, where $\eta>0$. Then the following hold.
    
    (i) $\theta_{c}^{\eta}(x)$ is C$^{1}$ over $X$, satisfying the following for all $x\in X$.
    \begin{align}
       \nabla \theta_{c}^{\eta}(x) = (\tfrac{n}{2\eta}) \mathbb{E}_{\bv\in \eta\mathbb{S}}\left[\, (\theta_{c}(x+\bv) - \theta_{c}(x-\bv))\tfrac{\bv}{\|\bv\|}\, \right]
    \end{align}
     
     Suppose $\theta_{c}$ is $L_{0}$-Lipschitz continuous and $L_{1}$-smooth on $X_{\eta}$. For any $x, y\in X$, (ii)-(v) hold.

    (ii) $|\theta_{c}^{\eta}(x)-\theta_{c}^{\eta}(y)|\leq L_{0}\|x-y\|$. (iii) $|\theta_{c}^{\eta}(x)-\theta_{c}(x)|\leq L_{0}\eta$.
    
    (iv) $\|\nabla \theta_{c}^{\eta}(x) - \nabla \theta_{c}^{\eta}(y)\|\leq \tfrac{L_{0}\sqrt{n}}{\eta}\|x-y\|$. (v) $\| \nabla \theta_{c}^{\eta}(x) - \nabla \theta_{c}(x) \|\leq \eta L_{1} n$.

    (vi) $\mathbb{E}_{\bv \in \eta\mathbb{S}} \left[ \| g(x, \bv) \|^{2} \right] \leq 16\sqrt{2\pi}L^{2}_{0}n$, where $g(x, \bv) \triangleq \left( \tfrac{n( \theta_{c}(x+\bv) - \theta_{c}(x-\bv) )\bv}{2\eta\left\| \bv \right\|} \right)$. $\hfill \Box$

\end{lemma} 

Next we present the \textbf{SAMGR} scheme in which a diminishing smoothing parameter sequence $\{\eta_k\}_{k\geq 0}$ is employed, implying that $\nabla \theta_c^{\eta_k}(x)$ approximates the true gradient $\nabla \theta_c(x)$ with increasing accuracy as $k \to \infty.$ 

\subsection{SAMGR scheme}\label{Sec-5.2}

\subsubsection{Assumption and algorithm}
Recall that the claim in Lemma~\ref{copositivity} relies on strict
copositivity of ${\bf J}F(x^\ast)$. From Lemma~\ref{copositivity}, the convergence claim under strict copositivity relies on the existence of $\nabla \theta_c(x)$, which by \eqref{regularized-gap-function-gradient} exists as long as $\mathbf{J} F(x)$ exists. Therefore, we always impose such an assumption throughout this subsection.  Our proposed {\bf SAMGR} scheme is presented in Algorithm \ref{SAMGR}, inspired by recent efforts in zeroth-order methods for solving stochastic nonsmooth nonconvex optimization \cite{cui-shanbhag-yousefian-2023, marrinan-shanbhag-yousefian-2026, qiu-shanbhag-yousefian-2023, shanbhag-yousefian-2021}. Note that an exact solution $y_{c}(x)$ in \eqref{yc-x-2} is unavailable in finite time and we employ stochastic approximation (\textbf{SA}) \cite{robbins-monro-1951} to compute an inexact solution, which is described in Algorithm \ref{SA-lower-level}. We consider the following assumption.

\vspace{5pt}

\emph{Assumption $\mathrm{D}$.} (D1) Suppose that $X$ is compact and $\mathbf{J} F(x)$ exists for any $x\in X$. (D2) Each $\Tilde{\theta}_{c}(\bullet, y(\bullet), \xi)$ is $L_{0}$-Lipschitz and $L_{1}$-smooth on $X_{\eta}$. (D3) Each $\Tilde{\theta}_{c}(x, \bullet, \xi)$ is $L^{y}_{0}$-Lipschitz on $X_{\eta}$. (D4) For any $x\in X$, $\|\nabla \theta_{c}(x)\| \leq M_{\theta}$ holds for some $M_{\theta} > 0$. (D5) Consider Algorithm \ref{SA-lower-level}. For any $k, t \geq 0$, we have (i) $\mathbb{E}[\nabla_{y}\widehat{F}(\hat{x}^{k}, y^{t}, \xi^{t}) \mid \hat{x}^{k}, y^{t}] = \nabla_{y}\mathbb{E}[\widehat{F}(\hat{x}^{k}, y^{t}, \bxi) \mid \hat{x}^{k}, y^{t}]$; (ii) $\mathbb{E}[ \| \nabla_{y}\widehat{F}(\hat{x}^{k}, y^{t}, \xi^{t}) - \nabla_{y}\mathbb{E}[\widehat{F}(\hat{x}^{k}, y^{t}, \bxi)] \|^{2} \mid \hat{x}^{k}, y^{t}]\leq v^{2}_{F}$ holds a.s. for some $v_{F} > 0$; and (iii) the bound $\| \nabla_{y}\mathbb{E}[\widehat{F}(x, y, \bxi)] \|$ $\leq c_{F}$ holds for some $c_{F} > 0$ and any $x, y\in X$. $\hfill \Box$

\vspace{5pt}

\begin{remark}
    For simplicity, we assume that each $\xi$-realization, given by $\widehat{F}(\bullet,\xi)$, is $L$-Lipschitz continuous for any $\xi\in \Xi$. We may obtain analogs of our results by instead supposing that each $\xi$-realization is $L(\xi)$-Lipschitz continuous under a suitable integrability condition on $L(\bxi)$. $\hfill \Box$
\end{remark}

\begin{algorithm}[htbp]
    \caption{~Stochastic Asynchronous Modified GR ({\bf SAMGR}) Scheme}
    Set $k = 0$. Initialize $x^{0}\in X$, a constant stepsize $\gamma < \min \{ \tfrac{N}{L_{1}}, N \}$, a sequence of parameter triples $\{\eta_{k}, N_{k}, \epsilon_{k}\}_{k\geq 0}$. Iterate until $k\geq K$.

    \textbf{Player selection.} Pick player $i(k)\in \{ 1, \cdots, N \}$ with probability $\tfrac{1}{N}$.

    \textbf{Single ZO estimator.} Set $j = 1$. Iterate (1)-(3) until $j \geq N_{k}$.

    (1) Generate an i.i.d. realization $v^{k}_{j}\in \eta_{k} \mathbb{S}$ of $\bv$ and an i.i.d. realization $\xi^{k}_{j}$ of $\bxi$.
    
    (2) Call Algorithm \ref{SA-lower-level} to obtain inexact solutions $y^{\epsilon_{k}}_{c}(x^{k}+v^{k}_{j})$ and $y^{\epsilon_{k}}_{c}(x^{k}-v^{k}_{j})$.

    (3) If $v_{j, i(k)}^{k}$ is the $i(k)$-th block of $v^{k}_{j}\in \eta_{k} \mathbb{S}$,
    compute $g^{\eta_{k}, \epsilon_{k}}_{j, i(k)}(x^{k})$ as follows
    \begin{equation*}
        g^{\eta_{k}, \epsilon_{k}}_{j, i(k)}(x^{k}) \triangleq \tfrac{n \left( \Tilde{\theta}_{c}(x^{k}+v^{k}_{j}, y^{\epsilon_{k}}_{c}(x^{k}+v^{k}_{j}), \xi^{k}_{j}) - \Tilde{\theta}_{c}(x^{k}-v^{k}_{j}, y^{\epsilon_{k}}_{c}(x^{k}-v^{k}_{j}), \xi^{k}_{j}) \right) v_{j, i(k)}^{k}}{2\eta_{k}\|v^{k}_{j}\|}.
    \end{equation*}

	\textbf{Mini-batch ZO estimator.} Compute $g^{\eta_{k}, \epsilon_{k}}_{N_{k}, i(k)}(x^{k}) \triangleq \tfrac{1}{N_{k}} \sum_{j=1}^{N_{k}} g^{\eta_{k}, \epsilon_{k}}_{j, i(k)}(x^{k})$.

    \textbf{Block update.} $x^{k+1}_{i} \triangleq \Pi_{X_{i}} [ x^{k}_{i} - \gamma g^{\eta_{k}, \epsilon_{k}}_{N_{k}, i(k)}(x^{k}) ]$ if $i = i(k)$ otherwise $x^{k+1}_{i} \triangleq x^{k}_{i}$.

    \textbf{Return.}~ $x^{R_{K}}$ as the final estimate, where $R_{K}$ is a random variable uniformly distributed over $\{\lceil \lambda K\rceil, \cdots, K\}$ with $\lambda\in(0,1)$.
    \label{SAMGR}
\end{algorithm}

\begin{algorithm}[htbp]
    \caption{~\textbf{SA} Scheme (at $k$th iteration of \textbf{SAMGR})}
    Set $t = 0$, $y^{0}\in X$, $\hat{x}^{k}\in X_{\eta}$, stepsizes $\{ \alpha_{t} \}_{t\geq 0}$ defined as $\alpha_{t} = \tfrac{\alpha_{0}}{t+\Gamma}$ where $\alpha_{0} > \tfrac{1}{2c}$ and $\Gamma>0$. Iterate until $t\geq t_{k}$.
    
    \textbf{Update.}  $y^{t+1} = \Pi_{X}[y^{t}-\alpha_{t} \nabla_{y}\widehat{F}(\hat{x}^{k}, y^{t}, \xi^{t})]$ for solving \eqref{yc-x-2},  where $\xi^{t}$ denotes the i.i.d. random realization of $\bxi$.

    \textbf{Return.} $y^{t_{k}}$ as final estimate and let $y^{\epsilon_{k}}_{c}(\hat{x}^{k}) \triangleq y^{t_{k}}$.
    \label{SA-lower-level}
\end{algorithm}

We observe that the Lipschitz continuity properties required in assumptions $\mathrm{(D2)}$ and $\mathrm{(D3)}$ follow from suitable requirements on $\mathrm{VI}\:(X, F)$. We formalize the relations in the following lemma and its proof can be found in the appendix.

\begin{lemma}\label{VI-requirement-deduces-Lipschitz-continuity}\em
    Suppose $X$ is compact with diameter $D_{X} > 0$ and $\mathbf{J} F(x)$ exists for all $x\in X$. Assume that there exists some $L_{F} > 0$, $M_{F} > 0$, $M_{J} > 0$, and $L_{J} > 0$ such that for any $x, y \in X$ and any realization $\xi$, we have that $\| \Tilde{F}(x, \xi) - \Tilde{F}(y, \xi) \| \leq L_{F} \| x - y \|$, $\| \Tilde{F}(x, \xi) \| \leq M_{F}$, $\| \mathbf{J} \Tilde{F}(x, \xi) \| \leq M_{J}$, and $\| \mathbf{J} \Tilde{F}(x, \xi) - \mathbf{J} \Tilde{F}(y, \xi) \| \leq L_{J} \| x - y \|$. Then the following hold: (i) $y_{c}(\bullet)$ is $(1 + \tfrac{L_{F}}{c})$-Lipschitz; (ii) Each $\Tilde{\theta}_{c}(\bullet, y(\bullet), \xi)$ is $L_{0}$-Lipschitz, where $L_{0} \triangleq 2D_{X}(L_{F}+c) + M_{F}(2 + \tfrac{L_{F}}{c})$; (iii) Each $\Tilde{\theta}_{c}(\bullet, y(\bullet), \xi)$ is $L_{1}$-smooth, where $L_{1} \triangleq 2(L_{F}+M_{J}+c) + D_{X}L_{J} + \tfrac{M_{J}L_{F}}{c}$; and (iv) Each $\Tilde{\theta}_{c}(x, \bullet, \xi)$ is $L^{y}_{0}$-Lipschitz, where $L^{y}_{0} \triangleq M_{F} + c D_{X}$. $\hfill \Box$
\end{lemma}

Based on Assumption $\mathrm{D}$, we next analyze the \textbf{SAMGR} scheme and provide rate and complexity statements.

\subsubsection{Convergence and complexities}

Before we analyze \textbf{SAMGR}, we first analyze the \textbf{SA} scheme, which provides inexact solutions $y^{\epsilon_{k}}_{c}(x^{k} + v^{k}_{j})$ and $y^{\epsilon_{k}}_{c}(x^{k} - v^{k}_{j})$. For convenience, we use the notation $\hat{x}^{k}$ to unify $x^{k} - v^{k}_{j}$ and $x^{k} + v^{k}_{j}$, i.e., when we say $\hat{x}^{k}$, it may refer to either $x^{k} - v^{k}_{j}$ or $x^{k} + v^{k}_{j}$. Below we clarify the error bound between the exact solution $y_{c}(\hat{x}^{k})$ and its inexact counterpart $y_{c}^{\epsilon_{k}}(\hat{x}^{k})$.

\begin{lemma}\label{lower-level-error}\em
Consider Algorithm \ref{SA-lower-level} for solving the strongly convex stochastic optimization problem \eqref{yc-x-2}. Suppose assumption $\mathrm{(D5)}$ holds. Given $k > 0$ and $\hat{x}^{k}\in X$, let $y_{c}(\hat{x}^{k})$ be the unique exact solution and $y^{\epsilon_{k}}_{c}(\hat{x}^{k})$ be the output of Algorithm \ref{SA-lower-level} after $t_{k}$ iterations. Then for any $k \geq 0$, we have that
\begin{equation*}
    \mathbb{E}[\|y^{\epsilon_{k}}_{c}(\hat{x}^{k})-y_{c}(\hat{x}^{k})\|^{2} \mid \hat{x}^{k}]\leq \epsilon_{k} \triangleq \tfrac{c_{\epsilon}}{t_{k}+\Gamma},
\end{equation*}
where $\Gamma > 0$ and $c_{\epsilon} \triangleq \max \left\{\tfrac{(c^{2}_{F}+v^{2}_{F})\alpha^{2}_{0}}{2c\alpha_{0}-1}, \: \Gamma\sup_{y\in X}\|y_{0}-y\|^{2} \right\} > 0$.
\end{lemma}
\begin{proof}
    The proof is similar to that of \cite[Theorem 2-(a)]{cui-shanbhag-yousefian-2023}, except that the cited result solves a strongly monotone stochastic VI, whereas here we solve the strongly convex stochastic optimization problem \eqref{yc-x-2}. We therefore omit the proof.
\end{proof}

To facilitate our analysis, besides $g^{\eta_{k}, \epsilon_{k}}_{j, i(k)}(x^{k})$ and $g^{\eta_{k}, \epsilon_{k}}_{N_{k}, i(k)}(x^{k})$ defined in Algorithm \ref{SAMGR}, we further define $g^{\eta_{k}}_{j}(x^{k})$ and $g^{\eta_{k}, \epsilon_{k}}_{j}(x^{k})$ as
\begin{align*}
    g^{\eta_{k}}_{j}(x^{k}) &\triangleq \tfrac{n \left( \Tilde{\theta}_{c}(x^{k}+v^{k}_{j}, y_{c}(x^{k}+v^{k}_{j}), \xi^{k}_{j}) - \Tilde{\theta}_{c}(x^{k}-v^{k}_{j}, y_{c}(x^{k}-v^{k}_{j}), \xi^{k}_{j}) \right) v_{j}^{k}}{2\eta_{k}\|v^{k}_{j}\|}, \\
    g^{\eta_{k}, \epsilon_{k}}_{j}(x^{k}) &\triangleq \tfrac{n \left( \Tilde{\theta}_{c}(x^{k}+v^{k}_{j}, y^{\epsilon_{k}}_{c}(x^{k}+v^{k}_{j}), \xi^{k}_{j}) - \Tilde{\theta}_{c}(x^{k}-v^{k}_{j}, y^{\epsilon_{k}}_{c}(x^{k}-v^{k}_{j}), \xi^{k}_{j}) \right) v_{j}^{k}}{2\eta_{k}\|v^{k}_{j}\|}.
\end{align*}
The block update step in \textbf{SAMGR} can be equivalently rewritten as
\begin{equation*}
    x^{k+1} \triangleq \Pi_{X} \left[ x^{k} - N^{-1} \gamma (\nabla \theta_{c}(x^{k}) + e^{k}_{1} + e^{k}_{2} + e^{k}_{3} + e^{k}_{4}) \right],
\end{equation*}
where errors $e^{k}_{1}$, $e^{k}_{2}$, $e^{k}_{3}$, and $e^{k}_{4}$ are defined as
\begin{equation}\label{three-errors}
    \begin{aligned}
        e^{k}_{1} &\triangleq \nabla \theta^{\eta_{k}}_{c}(x^{k}) - \nabla \theta_{c}(x^{k}), \quad e^{k}_{2} \triangleq \tfrac{1}{N_{k}} \sum_{j=1}^{N_{k}} e^{k}_{2, j}, \; e^{k}_{2, j} \triangleq g^{\eta_{k}}_{j}(x^{k}) - \nabla \theta^{\eta_{k}}_{c}(x^{k}), \\
        e^{k}_{3} &\triangleq \tfrac{1}{N_{k}} \sum_{j=1}^{N_{k}} e^{k}_{3, j}, \; e^{k}_{3, j} \triangleq g^{\eta_{k}, \epsilon_{k}}_{j}(x^{k}) - g^{\eta_{k}}_{j}(x^{k}), \\
        e^{k}_{4} &\triangleq \tfrac{1}{N_{k}} \sum_{j=1}^{N_{k}} e^{k}_{4, j}, \; e^{k}_{4, j} \triangleq N \mathbf{U}_{i(k)} g^{\eta_{k}, \epsilon_{k}}_{j, i(k)}(x^{k}) - g^{\eta_{k}, \epsilon_{k}}_{j}(x^{k}) \textrm{ with } N > 1,
    \end{aligned}
\end{equation}
where $\mathbf{U}_{i(k)}\in \mathbb{R}^{n\times n_{i(k)}}$ is the $i(k)$-th block of the identity matrix.

\begin{lemma}[SAMGR recursion]\label{theta-function-descent-lemma}\em 
    Suppose Assumption $\mathrm{D}$ holds. Consider the sequence $\{x^{k}\}_{k=0}^{\infty}$ generated by {\bf SAMGR} with a constant stepsize $\gamma < \min \{ \tfrac{N}{L_{1}}, N \}$. Then for any $k\geq 0$, we have that
    \begin{equation}\label{SAMGR-recursion}
        (1-\tfrac{\gamma L_{1}}{N})\tfrac{\gamma}{4N}\|G_{N/\gamma}(x^{k})\|^{2} \leq \theta_{c}(x^{k}) - \theta_{c}(x^{k+1}) + (1-\tfrac{\gamma L_{1}}{2N})\tfrac{\gamma}{N} \| e^{k}_{1} + e^{k}_{2} + e^{k}_{3} + e^{k}_{4} \|^{2}.
    \end{equation}
\end{lemma}
\begin{proof}
    The proof is similar to that of \cite[Lemma 5]{shanbhag-yousefian-2021}, and hence we omit it.
\end{proof}

We denote the \textbf{SAMGR} history at iteration $k$ by $\mathcal{F}_{k} \triangleq \cup_{t=0}^{k-1} \{ i(t), \cup_{j=1}^{N_{t}} \{ \xi^{t}_{j}, v^{t}_{j} \} \}$. Inspired by the Robbins-Siegmund lemma (see Lemma \ref{three-useful-lemma}-(i)), we know that as long as $\mathbb{E}[ \| e^{k}_{1} + e^{k}_{2} + e^{k}_{3} + e^{k}_{4} \|^{2} \!\mid\! \mathcal{F}_{k}]$ is summable, we can show that $\|G_{N/\gamma}(x^{k})\| \to 0$ a.s. as $k\to \infty$. To this end, we first analyze the bias and moment properties of errors defined in \eqref{three-errors}. The proof of the following lemma can be found in the appendix.

\begin{lemma}[Bias and moment properties]\label{bias-moment-lemma}\em
    Suppose Assumption $\mathrm{D}$ holds. Consider the errors defined in \eqref{three-errors}. Then the following hold a.s. for any $k\geq 0$.

    (i) $\mathbb{E}[ e^{k}_{2, j} \!\mid\! \mathcal{F}_{k} ] = \mathbb{E}[ e^{k}_{4, j} \!\mid\! \mathcal{F}_{k} ] = 0$ for any $j = 1, \cdots, N_{k}$.
    
    (ii) $\| e^{k}_{1} \|^{2} \leq \eta^{2}_{k} L^{2}_{1} n^{2}$, $\mathbb{E}[ \| e^{k}_{2} \|^{2} \!\mid\! \mathcal{F}_{k} ] \leq \tfrac{16\sqrt{2\pi}L^{2}_{0}n}{N_{k}}$, $\mathbb{E}[ \| e^{k}_{3} \|^{2} \!\mid\! \mathcal{F}_{k} ] \leq \tfrac{(L^{y}_{0})^{2}n^{2}\epsilon_{k}}{\eta^{2}_{k}}$, and
    \begin{equation*}
        \mathbb{E}[ \|e^{k}_{4}\|^{2} \!\mid\! \mathcal{F}_{k} ] \leq \tfrac{4(N-1) \left[ M^{2}_{\theta} + \eta^{2}_{k} L^{2}_{1} n^{2} + 16\sqrt{2\pi}L^{2}_{0}n + \eta^{-2}_{k}(L^{y}_{0})^{2}n^{2}\epsilon_{k} \right]}{N_{k}},
    \end{equation*}
    where $L_{0} > 0, L_{1} > 0, L^{y}_{0} > 0, M_{\theta} > 0$ are given in Assumption $\mathrm{D}$. $\hfill \Box$
\end{lemma}

\begin{theorem}[Almost sure convergence of SAMGR]\label{a.s.-SAMGR}\em
    Suppose Assumption $\mathrm{D}$ holds. Consider the sequence $\{ x^{k} \}_{k\geq 0}$ generated by {\bf SAMGR} with a constant stepsize $\gamma < \min \{ \tfrac{N}{L_{1}}, N \}$. For any $k\geq 0$, we define $\{ \eta_{k}, N_{k}, t_{k} \}_{k\geq 0}$ as
    \begin{equation}\label{SAMGR-parameters-choices}
        \eta_{k} \triangleq n^{-b}(k+1)^{-(\frac{1}{2}+\frac{1}{2}\delta)}, \; N_{k} \triangleq \lceil n^{a}(k+1)^{1+\delta} \rceil, \; \textrm{and } \; t_{k} \triangleq \lceil n^{e}(k+1)^{2+2\delta} \rceil
    \end{equation}
    for some $a, b, e \ge 0$ such that $e\geq 2b$ and $\delta > 0$, where $t_{k} > 0$ is the number of \textbf{SA} iterations at the $k$th iteration of \textbf{SAMGR}. Then we have $\|G_{N/\gamma}(x^{k})\|\to 0$ a.s. as $k\to \infty$. Almost surely, if the strict copositivity condition holds at the limiting point $x^\infty$ of the sequence $\{x^k\}_{k\geq 0}$, then $x^\infty$ is an isolated solution of $\mathrm{VI}\:(X, F)$.
\end{theorem}
\begin{proof}
    It follows from Lemma \ref{lower-level-error} and the expression for $t_{k}$ in \eqref{SAMGR-parameters-choices} that $\epsilon_{k} = \mathcal{O}(t^{-1}_{k}) = \mathcal{O}(n^{-e}(k+1)^{-(2+2\delta)})$. Therefore, by Lemma \ref{bias-moment-lemma}, we may claim the summability of $\mathbb{E}[ \| e^{k}_{1} + e^{k}_{2} + e^{k}_{3} + e^{k}_{4} \|^{2} \!\mid\! \mathcal{F}_{k} ]$. By invoking Lemma \ref{three-useful-lemma}-(i) on \eqref{SAMGR-recursion}, we have $\|G_{N/\gamma}(x^{k})\|\to 0$ a.s. as $k\to \infty$. The convergence to an isolated solution of $\mathrm{VI}\:(X, F)$ under strict copositivity follows immediately from Lemma \ref{copositivity}.
\end{proof}

Next we derive the complexity results of \textbf{SAMGR}. The following lemma gives an upper bound on the summation $\sum_{k=\lceil \lambda K \rceil}^{K} \mathbb{E}[ \| e^{k}_{1} + e^{k}_{2} + e^{k}_{3} + e^{k}_{4} \|^{2} \!\mid\! \mathcal{F}_{k} ]$, where $\lambda \in (0, 1)$ is given in Algorithm \ref{SAMGR}. We defer the proof to the appendix.

\begin{lemma}\label{SAMGR-summation-upper-bound}\em
    Suppose Assumption $\mathrm{D}$ holds. Consider the sequence $\{x^{k}\}_{k=0}^{\infty}$ generated by {\bf SAMGR} with a constant stepsize $\gamma < \min \{ \tfrac{N}{L_{1}}, N \}$. Consider the errors defined in \eqref{three-errors}. Suppose the total number of iterations $K$ satisfies $K\geq \tfrac{2}{1-\lambda}$, where $\lambda \in (0, 1)$ is given in Algorithm \ref{SAMGR}. Let $\{\eta_k, N_k, t_k\}_{k\geq 0}$ be given by \eqref{SAMGR-parameters-choices}. Define $h(a, b, e) \triangleq \min\{ 2b, 1+a, e-2b \} \geq 0$. Then we have
    \begin{equation*}
        \sum_{k=\lceil \lambda K \rceil}^{K} \mathbb{E}[ \| e^{k}_{1} + e^{k}_{2} + e^{k}_{3} + e^{k}_{4} \|^{2} \!\mid\! \mathcal{F}_{k} ] \leq \Theta(\lambda, N) \: n^{2-h(a, b, e)},
    \end{equation*}
    where $\Theta(\lambda, N) > 0$ is defined in \eqref{Theta-lambda} (see Appendix \ref{proof-SAMGR-summation-upper-bound}). $\hfill \Box$
\end{lemma}

Based on the above lemma, we next provide rate and complexity statements where the dependence on the dimension $n$ is qualified.

\begin{theorem}[Rate and complexity results of SAMGR]\label{rate-complexities-SAMGR}\em
    Suppose Assumption $\mathrm{D}$ holds. Consider the sequence $\{x^{k}\}_{k=0}^{\infty}$ generated by {\bf SAMGR} with a constant stepsize $\gamma < \min \{ \tfrac{N}{L_{1}}, N \}$. Suppose the total number of iterations $K$ satisfies $K\geq \tfrac{2}{1-\lambda}$, where $\lambda \in (0, 1)$ is given in Algorithm \ref{SAMGR}. Let $\{\eta_k, N_k, t_k\}_{k\geq 0}$ be given by \eqref{SAMGR-parameters-choices}. Then the following rate and complexity results hold.

    (i) (sublinear rate) We have that
    \begin{equation}\label{expected-residual-sublinear}
        \mathbb{E} \left[ \| G_{N/\gamma}(x^{R_{K}}) \|^{2} \right] \leq \tfrac{\mathbb{E} \left[ \theta_{c}(x^{\lceil \lambda K \rceil}) \right] \,-\,\theta^{\ast}_{c} \,+\, \Theta(\lambda, N)n^{2-h(a, b, e)}}{(1-\frac{\gamma L_{1}}{N})\frac{\gamma}{4N}(1-\lambda)K},
    \end{equation}
    where $\theta^{\ast}_{c} \triangleq \min_{x\in X} \theta_{c}(x)$, $h(a, b, e)$ and constant $\Theta(\lambda, N)$ are defined in Lemma \ref{SAMGR-summation-upper-bound}.

    (ii) (iteration and sample complexities) Suppose that the stepsize is chosen as $\gamma \triangleq \tfrac{N}{\theta_{\gamma} L_{1}} < \min \{ \tfrac{N}{L_{1}}, N \}$ for some $\theta_{\gamma} > 1$. Then the iteration and sample complexities required to guarantee $\mathbb{E} [\| G_{N/\gamma}(x^{R_{K}}) \|^{2}] \leq \epsilon$ are as follows.

    (ii-a) The iteration complexity is $K_\epsilon = \mathcal{O}( n^{2-h(a, b, e)} \epsilon^{-1} )$.

    (ii-b) The sample complexity is $S_{\epsilon} = \mathcal{O}( n^{(4+a+2\delta)-(2+\delta)h(a, b, e)}\epsilon^{-(2+\delta)} )$.

    (ii-c) Both iteration and sample complexities of \textbf{SA} are given by
    \begin{equation*}
        K^{\textrm{SA}}_{\epsilon} = S^{\textrm{SA}}_{\epsilon} = \mathcal{O}( n^{(8+a+e+6\delta)-(4+3\delta)h(a, b, e)} \epsilon^{-(4+3\delta)} ).
    \end{equation*}

    In particular, for the choice $(a, b, e) = (1, 1, 4)$, we have that $h(a, b, e) = 2$. Consequently, the above complexity bounds reduce to $\mathcal{O}(\epsilon^{-1})$, $\mathcal{O}(n\epsilon^{-(2+\delta)})$, and $\mathcal{O}(n^{5}\epsilon^{-(4+3\delta)})$, respectively.
\end{theorem}
\begin{proof}
    (i) By summing \eqref{SAMGR-recursion} over $k = \lceil \lambda K \rceil, \cdots, K$ and taking expectations on both sides, it leads to
    \allowdisplaybreaks
    \begin{align*}
        &(1-\tfrac{\gamma L_{1}}{N})\tfrac{\gamma}{4N} \sum_{k=\lceil \lambda K \rceil}^{K} \mathbb{E} \left[ \|G_{N/\gamma}(x^{k})\|^{2} \right] \\
        &\leq \mathbb{E} \left[ \theta_{c}(x^{\lceil \lambda K \rceil}) \right] - \theta^{\ast}_{c} + (1-\tfrac{\gamma L_{1}}{2N})\tfrac{\gamma}{N} \sum_{k=\lceil \lambda K \rceil}^{K} \mathbb{E}[ \| e^{k}_{1} + e^{k}_{2} + e^{k}_{3} + e^{k}_{4} \|^{2} ],
    \end{align*}
    where $\theta^{\ast}_{c} \triangleq \min_{x\in X} \theta_{c}(x)$. By the definition of $R_{K}$ and Lemma \ref{SAMGR-summation-upper-bound}, as well as the fact that $\gamma < \min \{ \tfrac{N}{L_{1}}, N \}$, and hence $(1-\tfrac{\gamma L_{1}}{2N})\tfrac{\gamma}{N} < 1$, we have that
    \begin{align*}
        & (K-\lceil \lambda K \rceil+1) \mathbb{E} \left[ \|G_{N/\gamma}(x^{R_{K}})\|^{2} \right] \leq \tfrac{\mathbb{E} \left[ \theta_{c}(x^{\lceil \lambda K \rceil}) \right] \,-\,\theta^{\ast}_{c} \,+\, \Theta(\lambda, N)n^{2-h(a, b, e)}}{(1-\frac{\gamma L_{1}}{N})\frac{\gamma}{4N}}.
    \end{align*}
    Since $(K-\lceil \lambda K \rceil+1) \geq (1-\lambda) K$, the result follows.

    (ii) Let the stepsize be $\gamma \triangleq \tfrac{N}{\theta_{\gamma} L_{1}} < \min \{ \tfrac{N}{L_{1}}, N \}$ for some $\theta_{\gamma} > 1$. Hence we have $(1-\tfrac{\gamma L_{1}}{N})\tfrac{\gamma}{4N} = \tfrac{\theta_{\gamma}-1}{4 \theta_{\gamma}^{2} L_{1}} = \mathcal{O} ( \tfrac{1}{L_{1}} )$. It follows that
    \begin{equation}\label{rate-complexities-SAMGR-proof-eqn1}
        \mathbb{E} \left[ \| G_{N/\gamma}(x^{R_{K}}) \|^{2} \right] \leq \mathcal{O} \left( \tfrac{n^{2-h(a, b, e)}}{K} \right).
    \end{equation}
    Then the iteration complexity $K_{\epsilon}$ in (ii-a) follows from \eqref{rate-complexities-SAMGR-proof-eqn1} immediately. The overall sample complexity is
    \allowdisplaybreaks
    \begin{align*}
        S_{\epsilon} = 2 \sum_{k=0}^{K_{\epsilon}} N_{k} &\overset{\eqref{SAMGR-parameters-choices}}{=} 2 \sum_{k=0}^{K_{\epsilon}} \lceil n^{a}(k+1)^{1+\delta} \rceil \leq n^{a} \mathcal{O}(K^{2+\delta}_{\epsilon}) \\
        &= \mathcal{O}(n^{(4+a+2\delta)-(2+\delta)h(a, b, e)}\epsilon^{-(2+\delta)}),
    \end{align*}
    where the factor of two accounts for sampling both $\bv \in \eta_{k} \mathbb{S}$ and $\bxi \in \Xi$ at each iteration. Then we complete the proof of (ii-b). To show the overall complexities of \textbf{SA} in (ii-c), we note that
    \allowdisplaybreaks
    \begin{align*}
        &K^{\textrm{SA}}_{\epsilon} = S^{\textrm{SA}}_{\epsilon} = 2\sum_{k=0}^{K_{\epsilon}} N_{k} t_{k} \overset{\eqref{SAMGR-parameters-choices}}{\leq} \mathcal{O} \left( \sum_{k=0}^{K_{\epsilon}} n^{a+e} (k+1)^{3+3\delta} \right) \\
        &\leq n^{a+e} \mathcal{O}(K^{4+3\delta}_{\epsilon}) = \mathcal{O}( n^{(8+a+e+6\delta)-(4+3\delta)h(a, b, e)} \epsilon^{-(4+3\delta)} ),
    \end{align*}
    which completes the proof. By setting $(a, b, e) = (1, 1, 4)$, we obtain the stated specific complexity results.
\end{proof}

\section{Applications and numerics}\label{Sec-6}

In this section, we apply \textbf{SSGR}, \textbf{SAGR}, and \textbf{SAMGR} schemes to three distinct applications.

\subsection{Stochastic network congestion problems}

Consider $N$ players competing on a directed graph $\mathcal{G} = (\mathcal{V}, \mathcal{L})$, where $\mathcal{V}$ and $\mathcal{L}$ denote the sets of nodes and directed links, respectively. Suppose that $|\mathcal{V}| = N$ and each player is located at a distinct node, where the flow of the $i$th player on the link $\ell \in \mathcal{L}$ is denoted by $x^{\ell}_{i}$. Let $x_i \triangleq (x_i^\ell)_{\ell\in\mathcal L}$ denote the $i$th player's flow over different links, subject to the lower and upper bounds $\mathbf{lb}\in\mathbb R_+^{|\mathcal L|}$ and $\mathbf{ub}\in\mathbb R_+^{|\mathcal L|}$. Suppose $X_{i} \subseteq \mathbb{R}^{|\mathcal{L}|}_{+}$ captures the $i$th player's flow constraints. Let $x^{\ell} \triangleq (x^{\ell}_{i})_{i=1}^{N}$ denote the vector of flows of $N$ different players on the link $\ell \in \mathcal{L}$. We denote the random flow capacity on link $\ell \in \mathcal{L}$ by $b^\ell(\bxi)$ and assume that it is uniformly upper bounded by $b^{\ell}(\bxi) \leq b^{\ell}_{\max}$. The $i$th player solves the following parameterized problem:
\begin{equation}\label{G-net}
    \min_{x_i \in X_i} f_i(x_i,x_{-i}) \triangleq \sum_{\ell \in \mathcal{L}} \mathbb{E}\left[ \tfrac{M_{G}}{b^{\ell}(\bxi)-\|x^\ell\|_2} - \tfrac{\beta(\bxi)(x_i^\ell)^2}{2} \right], ~ \forall i\in [N], \tag{\ensuremath{G^{\mathrm{net}}}}
\end{equation}
where $M_{G} > 0$ and each player faces a stochastic optimization problem involving a concave quadratic utility term with $\mathbb{E}[\beta(\bxi)] > 0$. The following lemma shows that the stochastic network congestion game \eqref{G-net} satisfies the \ref{QG} property. The proof can be found in the appendix.

\begin{lemma}\label{QG-network-game}\em 
    Consider the stochastic $N$-player game \eqref{G-net}. Suppose that each $X_i$ is nonempty, convex, and compact. Suppose that for any $\ell \in \mathcal{L}$, the bound $d^{\ell}_{\min} \leq \| x^\ell \|_2 \leq d^{\ell}_{\max}$ holds for some constants $d^{\ell}_{\min}, d^{\ell}_{\max} > 0$. If we have that
    \begin{equation*}
        \tfrac{M_{G}}{d^{\ell}_{\max} (b_{\rm max}^{\ell}-d_{\min}^{\ell})^{2}} - \mathbb{E}[\beta(\bxi)] \triangleq \eta^{\ell} > 0, ~ \forall \ell \in \mathcal{L},
    \end{equation*}
    then the map $F$ satisfies the \ref{QG} property with $\eta \triangleq \min_{\ell \in \mathcal{L}} \eta^{\ell} > 0$. $\hfill \Box$
\end{lemma}

\begin{figure}[htp]
    \centering
    \begin{minipage}[c]{0.38\textwidth}
        \centering
        \resizebox{0.6\linewidth}{!}{
        \begin{tikzpicture}[->, >=Stealth, node distance=2cm]
            \node[circle, draw, minimum size=1cm] (1) {1};
            \node[circle, draw, minimum size=1cm, right=of 1] (2) {2};
            \node[circle, draw, minimum size=1cm, below=of 1] (3) {3};
            \node[circle, draw, minimum size=1cm, below=of 2] (4) {4};

            \draw (1) -- (2) node[midway, above] {1};
            \draw (2) -- (4) node[midway, right] {3};
            \draw (3) -- (4) node[midway, below] {4};
            \draw (1) -- (3) node[midway, left] {2};
            \draw (4) -- (1) node[pos=0.3, below] {6};
            \draw (2) -- (3) node[pos=0.3, below] {5};
        \end{tikzpicture}
        }
    \end{minipage}
    \hfill
    \begin{minipage}[c]{0.6\textwidth}
        \vspace{-10pt}
        \begin{equation*}
        \begin{aligned}
            X_1 &\triangleq \{ \mathbf{lb}\leq x_1\leq \mathbf{ub} \mid x^6_1 = x^1_1 + x^2_1 \}, \\
            X_2 &\triangleq \{ \mathbf{lb}\leq x_2\leq \mathbf{ub} \mid x^1_2 = x^3_2 + x^5_2 \}, \\
            X_3 &\triangleq \{ \mathbf{lb}\leq x_3\leq \mathbf{ub} \mid x^2_3 + x^5_3 = x^4_3 \}, \\
            X_4 &\triangleq \{ \mathbf{lb}\leq x_4\leq \mathbf{ub} \mid x^3_4 + x^4_4 = x^6_4 \}.
        \end{aligned}
        \end{equation*}
    \end{minipage}
    \caption{The directed graph and flow constraints for the network congestion problem.}
    \label{graph-flow-network-congestion-game}
    \vspace{-15pt}
\end{figure}

Let us consider the directed graph $\mathcal{G} = (\mathcal{V}, \mathcal{L})$ in Figure \ref{graph-flow-network-congestion-game}. There are $N = 4$ players, and each is located at a distinct node. Suppose that the incoming flow equals the outgoing flow at each node, as specified by the flow constraints, where \textbf{lb} and \textbf{ub} denote the lower and upper bounds on each player’s flow, respectively. In this example, we assume that $M_{G} = 5 \times 10^{4}$, $b^{\ell}(\bxi) \sim \mathrm{U}\:[44, 45]$ for any $\ell \in \mathcal{L}$, $\beta(\bxi) = 2 + 0.2\bxi$ where $\bxi \sim \mathrm{U}\:[-1, 1]$, $\mathbf{lb} = 4\mathbf{e}_{6}$, and $\mathbf{ub} = 12\mathbf{e}_{6}$. We compare our \textbf{SSGR} and \textbf{SAGR} schemes with the stochastic Popov method \cite{vankov-nedic-sankar-2023} (requiring two projections and a single sample for each step) and the stochastic extragradient (\textbf{SEG}) method \cite{kannan-shanbhag-2019} (requiring two projections and samples for each step), as well as their asynchronous versions on such a stochastic networked equilibrium problem. All curves are averaged over $20$ runs, and each scheme is terminated after $300$ projection evaluations.

\begin{figure}[htp]
    \centering
    \begin{subfigure}{0.24\textwidth}
        \centering
        \includegraphics[width=\textwidth]{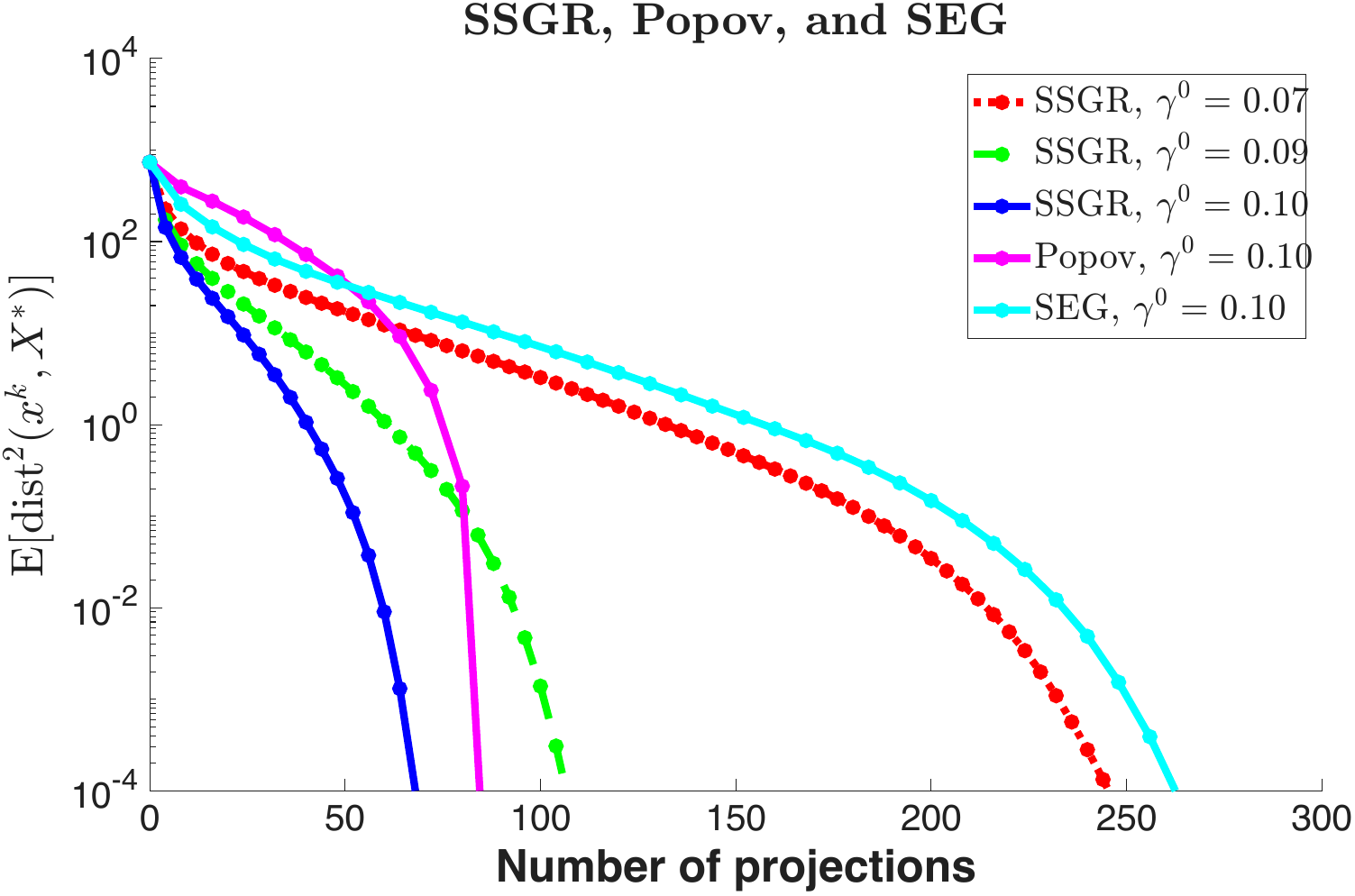}
        \caption{}
    \end{subfigure}
    \hfill
    \begin{subfigure}{0.24\textwidth}
        \centering
        \includegraphics[width=\textwidth]{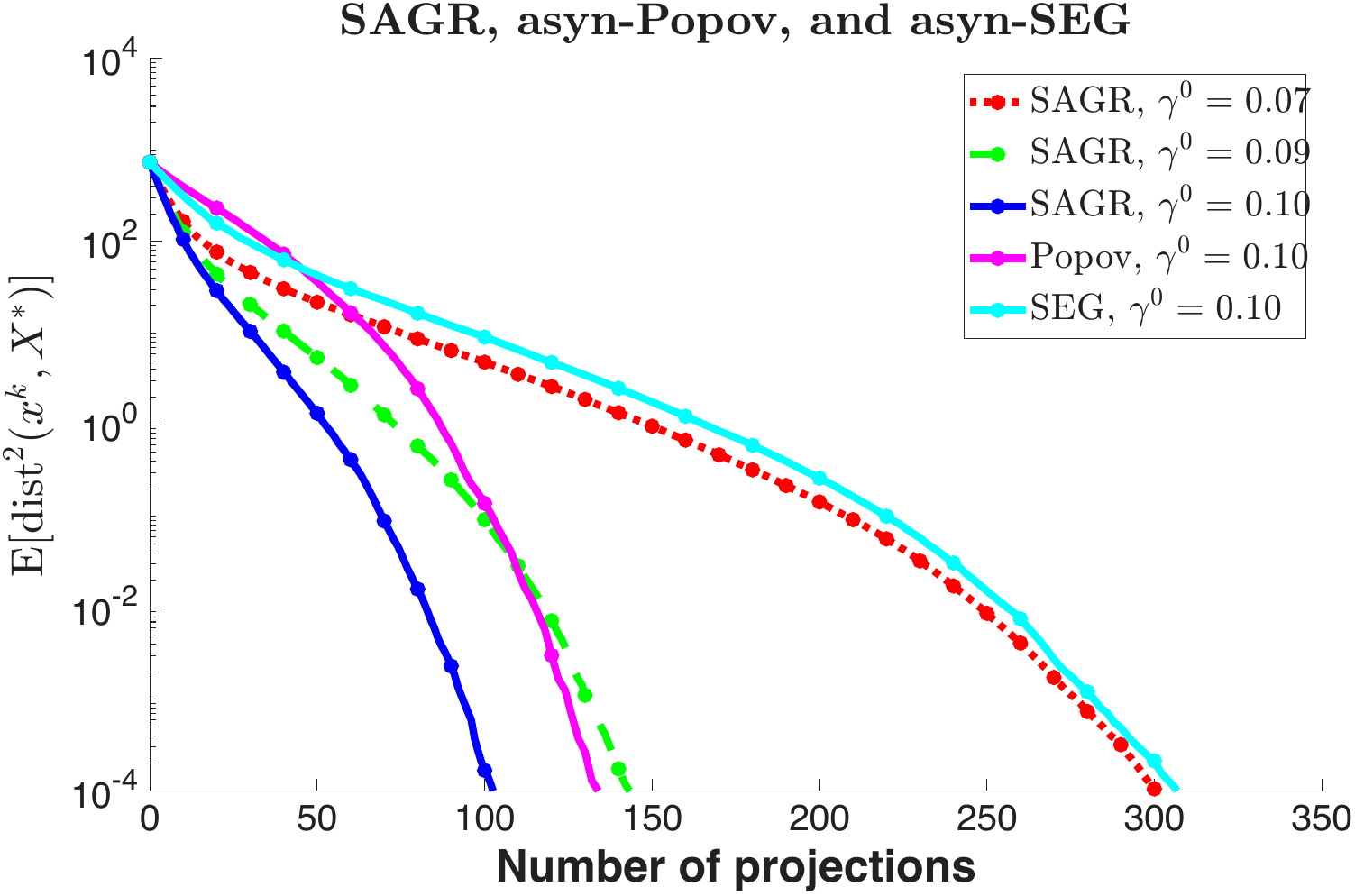}
        \caption{}
    \end{subfigure}
    \hfill
    \begin{subfigure}{0.24\textwidth}
        \centering
        \includegraphics[width=\textwidth]{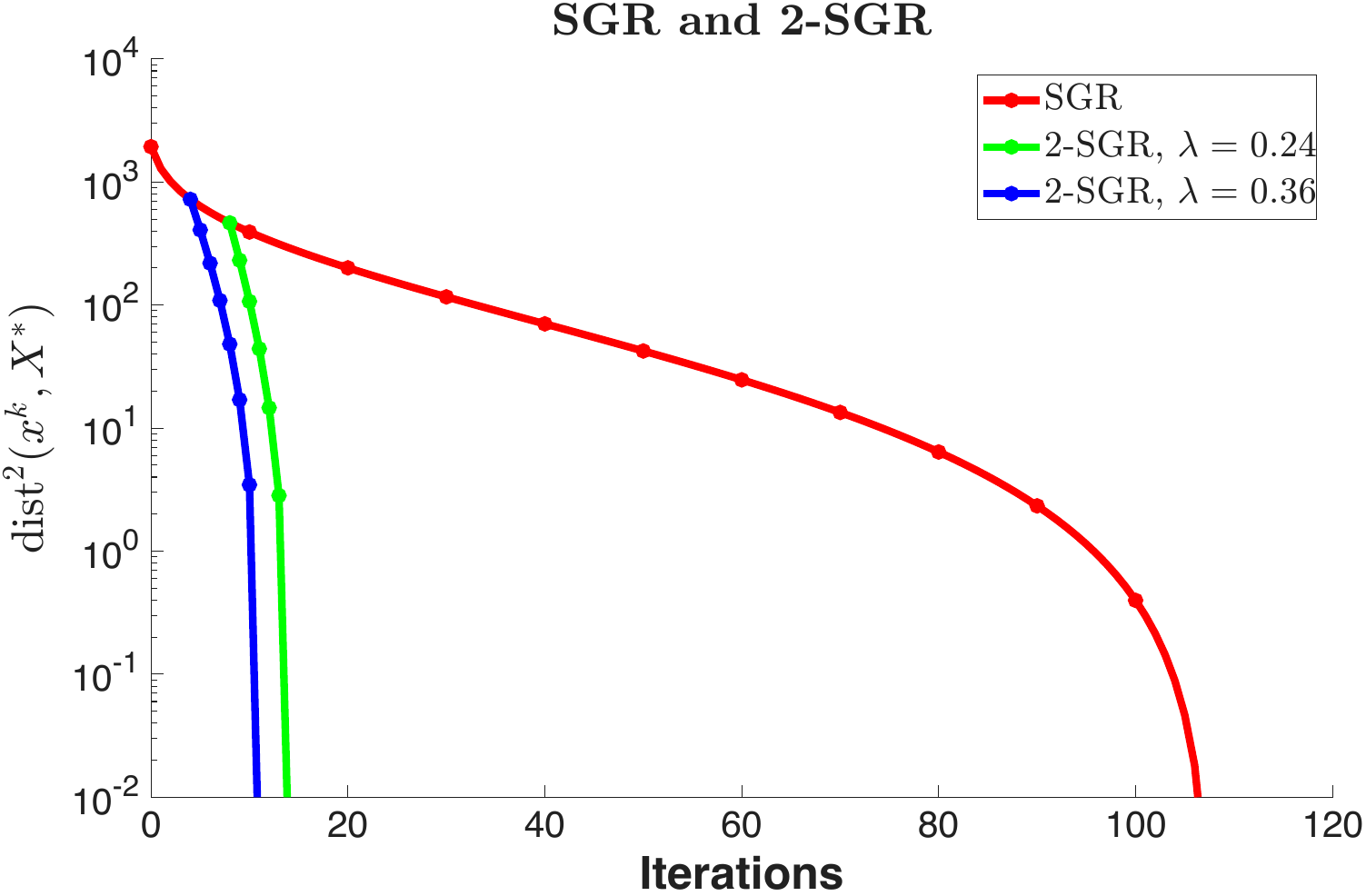}
        \caption{}
    \end{subfigure}
    \hfill
    \begin{subfigure}{0.24\textwidth}
        \centering
        \includegraphics[width=\textwidth]{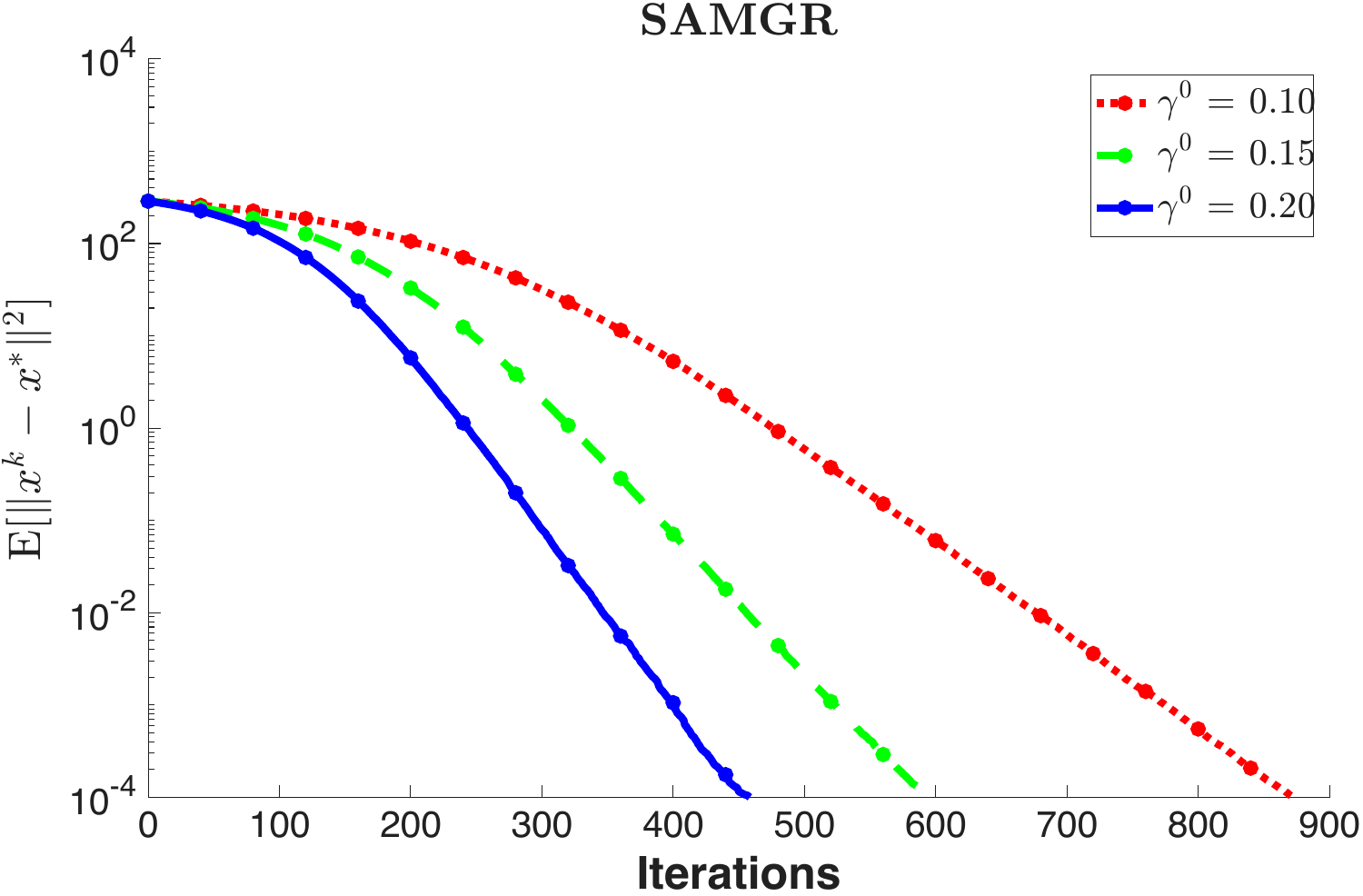}
        \caption{}
    \end{subfigure}
    \vspace{-0.1in}
    \caption{(a): \textbf{SSGR}, Popov, and \textbf{SEG}; (b): \textbf{SAGR}, asyn-Popov, and asyn-\textbf{SEG}; (c): \textbf{SGR} and \textbf{2-SGR}; and (d): \textbf{SAMGR}.}
    \label{numerics}
\end{figure}
\vspace{-0.1in}

The comparison results in Figure \ref{numerics}-(a) and Figure \ref{numerics}-(b) show that, under the same choice of $\gamma^{0}$, \textbf{SSGR} competes well with the Popov and SEG schemes and similar behavior emerges in the asynchronous setting.

\subsection{Nonconvex-nonconcave minimax problems} 
Consider the following smooth nonconvex-nonconcave minimax problem denoted by \eqref{minimax}:
\begin{equation*}\label{minimax}
    \min_{x\in X} \max_{y\in Y} L(x, y) \triangleq 50x + (x-20) \bigg[ \tfrac{1}{20} (y-20) - \tfrac{1}{50}(x-20) + \tfrac{1}{500}(y-20)^{2} \bigg], \tag{MM}
\end{equation*}
where $X = Y \triangleq [20, 70]$. We say that $(x^{\ast}, y^{\ast})^{\top}$ is a stationary point of \eqref{minimax} if
\begin{align*}
    \nabla_{x} L(x^{\ast}, y^{\ast})^{\top} (x - x^{\ast}) &\geq 0, ~ \forall x \in X, \\
    -\nabla_{y} L(x^{\ast}, y^{\ast})^{\top} (y - y^{\ast}) &\geq 0, ~ \forall y \in Y.
\end{align*}
We may observe that the above stationarity condition is equivalent to finding a $z^{\ast} = (x^{\ast}, y^{\ast})^{\top}$ such that $F(z^{\ast})^{\top}(z-z^{\ast}) \geq 0$ holds for any $z \in Z \triangleq X \times Y$, where the concatenated map $F$ is defined as
\begin{equation*}
    F(z) \triangleq \begin{bmatrix}
        \nabla_{x} L(x, y) \\
        -\nabla_{y} L(x, y)
    \end{bmatrix}.
\end{equation*}
By computation, we obtain that $Z^{\ast} = \{20\} \times [20,70]$, $\|F(z)\| \leq M_F=\tfrac{115}{2}$ holds for any $z \in Z$, and the \ref{QS} property holds with $\beta = 43$. We compare \textbf{2-SGR} with \textbf{SGR} on such a problem in Figure \ref{numerics}-(c). Observe that \textbf{2-SGR} significantly outperforms \textbf{SGR} in terms of the number of iterations required to achieve similar accuracy. Moreover, choosing a smaller value of $\lambda$ delays entry of the iterate sequence into the second stage, while maintaining a faster linear convergence rate with a smaller value of $q$.

\subsection{Strictly copositive congestion games}
Consider the $N$-player congestion game \eqref{G-con}, where the $i$th player wants to maximize her utility function:
\begin{equation}\label{G-con}
    \max_{x_i \in X_i} f_{i}(x_{i}, x_{-i}) \triangleq U_i(x_i) - x_i \sum_{j \ne i} g_j(x_j), ~ \forall i\in [N]. \tag{\ensuremath{G^{\mathrm{con}}}}
\end{equation}
The concatenation map $F$ associated with \eqref{G-con} may be potentially non-monotone, complicating the computation of NE. Yet when such a game satisfies the strict copositivity condition, NE can be efficiently computed via the \textbf{SAMGR} scheme.

\begin{lemma}\em
    Consider the $N$-player game \eqref{G-con}. Suppose that there are $N = 8$ players, with $X_i \triangleq [0,20]$ for each $i\in [N]$. For each player $i\in [N]$, let $U_{i}(x_{i}) \triangleq -\tfrac{1}{2} x_{i}^2 + \mathbb{E}[80+\bxi]x_{i}$, where $\bxi \sim \mathrm{U} \: [-4, 4]$, and let $g_{j}(x_{j}) \triangleq \tfrac{1}{100}(x_{j}-10)^{2} + 10$ for each $j\neq i$. Then $x^{\ast} = (10, \cdots, 10)^{\top}$ is a stationary point of $\theta_{c}$ on $X$ and an NE of \eqref{G-con}.
\end{lemma}
\begin{proof}
    The $i$th player's parameterized optimization problem is
    \begin{equation*}
        \min_{x_{i}\in X_{i}} f_{i}(x_{i}, x_{-i}) \triangleq \left( \tfrac{1}{2}x^{2}_{i} - \mathbb{E}[80 + \bxi]x_{i} \right) + x_{i}\sum_{j\neq i} \left( \tfrac{1}{100}(x_{j}-10)^{2} + 10 \right), ~ \forall i\in [N].
    \end{equation*}
    We may verify that such a game is convex since $\nabla^{2}_{x_{i}x_{i}}f_{i}(x_{i}, x_{-i}) = 1 > 0$ for any $i\in [N]$. Consider $x^{\ast} = (10, \cdots, 10)^{\top}$. Next we show that $F$ is not monotone but the strict copositivity condition in Lemma \ref{copositivity} holds at $x^{\ast}$.  Recall that $F$ is monotone on $X$ if and only if $\mathbf{J}F(x) \succeq 0$  for all $x\in X$~\cite[Proposition~2.3.2]{facchinei-pang-2003}. By computation, we may obtain the Jacobian matrix $\mathbf{J}F(x)$ as follows.
    \begin{equation*}
        \mathbf{J}F(x) = \begin{bmatrix}
        1 & 0.02(x_{1}-10) & \cdots & 0.02(x_{1}-10) \\
        0.02(x_{2}-10) & 1 & \cdots & 0.02(x_{2}-10) \\
        \vdots & \vdots & \ddots & \vdots \\
        0.02(x_{8}-10) & 0.02(x_{8}-10) & \cdots & 1
        \end{bmatrix}.
    \end{equation*}
    At $x'=(0,\ldots,0)^{\top}$, the Jacobian $\mathbf{J}F(x')$ has an eigenvalue equal to $-0.4$. Consequently, $\mathbf{J}F(x') \not\succeq 0$, which implies that $F$ is not monotone on $X$.  Next, we verify the strict copositivity condition. It is not difficult to see that $0\in \nabla \theta_{c}(x^{\ast}) + \mathcal{N}_{X}(x^{\ast})$ by \eqref{regularized-gap-function-gradient}, i.e., $x^{\ast}$ is a stationary point. We want to show that $\mathbf{J}F(x^*)$ is strictly copositive on $\mathcal{T}_{X}(x^{\ast})\cap (-F(x^{\ast}))^{\ast}$. Indeed, $\mathbf{J}F(x^{\ast})$ is an identity matrix $\textrm{\bf I}_{8\times 8}$. Therefore, $\mathbf{J}F(x^{\ast})$ is strictly copositive on the whole space $\mathbb{R}^{8}$, and hence also strictly copositive on $\mathcal{T}_{X}(x^{\ast}) \cap (-F(x^{\ast}))^{\ast}$. By the definition of an NE, we may show that $x^{*}$ is an NE of \eqref{G-con}.
\end{proof}

Due to the unavailability of $\nabla \theta_{c}(x)$, we use $\mathbb{E}[\| x^{k} - x^{\ast} \|^{2}]$ as the convergence measure and test three different constant stepsizes in Figure \ref{numerics}-(d). All curves are averaged over $20$ runs. We may observe that a larger step size performs better than a smaller step size given the same number of iterations. 

\section{Concluding remarks}\label{Sec-7} 

The consideration of nonconvexity in continuous-strategy static noncooperative games is at a relatively nascent stage. Few efficient schemes exist with last-iterate convergence guarantees for contending with games with smooth expectation-valued and potentially nonconvex objectives. To this end, we develop \textbf{SSGR} and \textbf{SAGR} schemes with a.s. convergence and sublinear rate guarantees for computing a QNE under the \ref{QG} property. Surprisingly, this claim can be strengthened to computing an NE when the game admits a smooth and pseudoconvex potential function. Notably, in a deterministic setting, linear rates can be derived under the \ref{QS} property. We then consider the \textbf{SAMGR} scheme for computing equilibria in convex but possibly non-monotone games with expectation-valued objectives, which allows us to derive a sublinear rate for the expected residual and establish convergence to a solution under the strict copositivity condition. We show that the prescribed properties emerge in applications and provide promising numerics. 

\appendix
\section{Auxiliary proofs}

\subsection{Proof of Lemma \ref{SAGR-stepsize-asymptotics}}\label{proof-SAGR-stepsize-asymptotics}
\begin{proof}
    (i) By definition, we know that $\Gamma_{k}(i) = \sum_{t=0}^{k} \mathbf{1}_{E_{i, t}}$, where $E_{i, t}$ is the event that player $i$ is selected at iteration $t$ and $\mathbf{1}_{E_{i, t}}$ is the indicator function of the event $E_{i, t}$. The events $\{ E_{i, t} \}_{t=0}^{k}$ are i.i.d. with mean $\mathbb{E}[\mathbf{1}_{E_{i, t}}] = p_{i}$ for any $i\in [N]$. By the law of iterated logarithm \cite[pp.\,476--479]{dudley-2002}, for any deterministic $\Tilde{d} \in (0, \tfrac{1}{2})$, we have
    \begin{equation*}
        \lim_{k\to \infty} \tfrac{|\Gamma_{k}(i) - (k+1)p_{i}|}{(k+1)^{1/2+\Tilde{d}}} = 0, ~ \textrm{a.s.} ~ \forall i\in [N].
    \end{equation*}
    Therefore, for some sufficiently large $\Tilde{k} \triangleq \Tilde{k}(\Tilde{d}, N)$, for any $i\in [N]$ and any $k \geq \Tilde{k}$, 
    \begin{equation}\label{proof-SAGR-stepsize-asymptotics-eqn1}
        \tfrac{|\Gamma_{k}(i) - (k+1)p_{i}|}{(k+1)^{1/2+\Tilde{d}}} \leq 1, ~ \textrm{a.s.} ~ \forall i\in [N],
    \end{equation}
    implying that for any $k \geq \Tilde{k}$,
    \begin{equation*}
        \Gamma_{k}(i) \geq (k+1)p_{i} - (k+1)^{1/2+\Tilde{d}} = \left( (k+1)^{1/2-\Tilde{d}}p_{i} - 1 \right) (k+1)^{1/2+\Tilde{d}}.
    \end{equation*}
    Since $\Tilde{d}\in (0, \tfrac{1}{2})$, the term $(k+1)^{1/2-\Tilde{d}}p_{i}$ tends to infinity as $k$ goes to infinity. Thus, for any $k\geq \Tilde{k}$ (we can continue to choose a larger $\Tilde{k}$, if needed), we have that
    \begin{equation*}
        (k+1)^{1/2-\Tilde{d}}p_{i} - 1 \geq \tfrac{1}{2} (k+1)^{1/2-\Tilde{d}}p_{i}, ~ \textrm{a.s.} ~ \forall i\in [N].
    \end{equation*}
    By combining the above two relations, it follows that for any $k\geq \Tilde{k}$, for any $i\in [N]$,
    \begin{equation*}
        \Gamma_{k}(i) \geq \tfrac{1}{2} (k+1) p_{i}, ~ \textrm{a.s.} ~ \forall i\in [N].
    \end{equation*}
    Therefore, for any $\Tilde{d}\in (0, \tfrac{1}{2})$, we have for any $k\geq \Tilde{k}$ and any $i\in [N]$,
    \begin{equation}\label{proof-SAGR-stepsize-asymptotics-eqn2}
        \gamma^{k}_{i} = \tfrac{1}{\Gamma_{k}(i)} \leq \tfrac{2}{(k+1)p_{i}} \leq \tfrac{2}{k p_{i}}, ~ \textrm{a.s.} ~ \forall i\in [N],
    \end{equation}
    as desired, where the first equality follows from $\gamma_{i}^{k} = \tfrac{1}{\Gamma_k(i)}$ for sufficiently large $k$.
    
    The result (ii) follows directly from \eqref{proof-SAGR-stepsize-asymptotics-eqn2}. For (iii),  we have
    \begin{align*}
        | \gamma^{k}_{i} - \tfrac{1}{k p_{i}} | = | \gamma^{k}_{i} - \tfrac{1}{(k+1) p_{i}} + \tfrac{1}{(k+1) p_{i}} - \tfrac{1}{k p_{i}} | \leq | \gamma^{k}_{i} - \tfrac{1}{(k+1) p_{i}} | + | \tfrac{1}{(k+1) p_{i}} - \tfrac{1}{k p_{i}} |.
    \end{align*}
    Moreover, we further have (recall that $\gamma_{i}^{k} = \tfrac{1}{\Gamma_k(i)}$ holds for sufficiently large $k$)
    \allowdisplaybreaks
    \begin{align*}
        | \gamma^{k}_{i} - \tfrac{1}{(k+1) p_{i}} | &= | \tfrac{1}{\Gamma_k(i)} - \tfrac{1}{(k+1) p_{i}} | = \tfrac{1}{(k+1)p_{i}} \tfrac{1}{\Gamma_k(i)} | (k+1)p_{i} - \Gamma_{k}(i) | \\
        &\overset{\eqref{proof-SAGR-stepsize-asymptotics-eqn2}}{\leq} \tfrac{2}{(k+1)^{2}p^{2}_{i}} | (k+1)p_{i} - \Gamma_{k}(i) | \overset{\eqref{proof-SAGR-stepsize-asymptotics-eqn1}}{\leq} \tfrac{2}{(k+1)^{2}p^{2}_{i}} (k+1)^{1/2+\Tilde{d}} \\
        &= \tfrac{2}{(k+1)^{3/2-\Tilde{d}} p^{2}_{i}} \leq \tfrac{2}{k^{3/2-\Tilde{d}} p^{2}_{\min}} \\
        \mbox{ and } & | \tfrac{1}{(k+1) p_{i}} - \tfrac{1}{k p_{i}} | = \tfrac{1}{p_{i}} \tfrac{1}{(k+1)k} \leq \tfrac{1}{k^{2} p^{2}_{\min}} \leq \tfrac{1}{k^{3/2-\Tilde{d}} p^{2}_{\min}}.
    \end{align*}
    By combining the above two inequalities, we obtain that $| \gamma^{k}_{i} - \tfrac{1}{k p_{i}} | \leq \tfrac{3}{k^{3/2-\Tilde{d}} p^{2}_{\min}}$ for any $k\geq \Tilde{k}$, which completes the proof.
\end{proof}

\subsection{Proof of Lemma \ref{Chung}}\label{proof-Chung}
\begin{proof}
    By Lemma $3$ in \cite[Chapter 2.2]{polyak-1987}, we know that $e^{k}\to 0$ as $k\to \infty$. Therefore, there exists a sufficiently large $\Tilde{k}$ such that for $k\geq \Tilde{k}$ we have $e^{k} \leq 1$. It follows that for $k\geq \Tilde{k}$ we have $e^{k+1} \leq (1-\tfrac{C}{k})e^{k} + \tfrac{A+B}{k^{t+1}}$. By Chung's lemma~\cite[Chapter 2.2]{polyak-1987}, we can derive the desired rate statement.
\end{proof}

\subsection{Proof of Lemma \ref{VI-requirement-deduces-Lipschitz-continuity}}\label{proof-VI-requirement-deduces-Lipschitz-continuity}
\begin{proof}
    (i) By Jensen's inequality, we have that $\| F(x) - F(y) \| = \| \mathbb{E}[\Tilde{F}(x, \xi)] - \mathbb{E}[\Tilde{F}(y, \xi)] \| \leq \mathbb{E}[\| \Tilde{F}(x, \xi) - \Tilde{F}(y, \xi) \|] \leq L_{F}\|x - y\|$. Recall that $y_{c}(x) = \Pi_{X}[x - \tfrac{1}{c}F(x)]$. The Lipschitz continuity of $y_c(\bullet)$ follows from the nonexpansiveness of the Euclidean projection, and we have that
    \begin{align*}
        \| y_{c}(x) - y_{c}(x') \| &\leq \| (x - \tfrac{1}{c}F(x)) - (x' - \tfrac{1}{c}F(x')) \| \leq \| x - x' \| + \tfrac{1}{c} \| F(x) - F(x') \| \\
        &\leq \| x - x' \| + \tfrac{L_{F}}{c} \| x - x' \| = (1 + \tfrac{L_{F}}{c}) \| x - x' \|.
    \end{align*}
    holds for any $x, x' \in X$. 

    (ii) For any $x, x' \in X$ and any realization $\xi$, it follows from \eqref{gap-function} that
    \allowdisplaybreaks
    \begin{align*}
        & \quad | \Tilde{\theta}_{c}(x, y_{c}(x), \xi) - \Tilde{\theta}_{c}(x', y_{c}(x'), \xi) |  \\
        & \leq | \tilde{F}(x, \xi)^{\top} (x - y_{c}(x)) - \tilde{F}(x',\xi)^{\top} (x - y_c(x)) | \\ 
        & \quad + |\tilde{F}(x', \xi)^{\top} (x - y_c(x)) - \tilde{F}(x', \xi)^{\top} (x'-y_c(x')) | \\  
        & \quad + | \tfrac{c}{2}(x - y_c(x))^{\top} (x-y_c(x)) - \tfrac{c}{2}(x-y_c(x))^{\top} (x' - y_c(x')) | \\
        & \quad + | \tfrac{c}{2}(x-y_c(x))^{\top} (x' - y_c(x')) - \tfrac{c}{2}(x'-y_c(x'))^{\top} (x'-y_c(x')) | \\
        &\leq ( 2D_{X}(L_{F}+c) + M_{F}(2 + \tfrac{L_{F}}{c}) ) \| x - x' \|.
    \end{align*}
    Then we may set $L_{0} \triangleq 2D_{X}(L_{F}+c) + M_{F}(2 + \tfrac{L_{F}}{c})$, as desired.

    (iii) For any $x, x' \in X$ and any realization $\xi$, it follows from \eqref{gap-function} that
    \allowdisplaybreaks
    \begin{align*}
        & \quad | \nabla \Tilde{\theta}_{c}(x, y_{c}(x), \xi) - \nabla \Tilde{\theta}_{c}(x', y_{c}(x'), \xi) | \leq \| \Tilde{F}(x, \xi) - \Tilde{F}(x', \xi) \| \\
        & \quad + \| (\mathbf{J} \Tilde{F}(x, \xi) - c\mathbf{I})^{\top} (x - y_{c}(x)) - (\mathbf{J} \Tilde{F}(x', \xi) - c\mathbf{I})^{\top} (x' - y_{c}(x')) \| \\
        & \leq L_{F} \| x - x' \| + \| (\mathbf{J} \Tilde{F}(x, \xi) - c\mathbf{I})^{\top} (x - y_{c}(x)) - (\mathbf{J} \Tilde{F}(x', \xi) - c\mathbf{I})^{\top} (x - y_{c}(x)) \| \\
        & \quad + \| (\mathbf{J} \Tilde{F}(x', \xi) - c\mathbf{I})^{\top} (x - y_{c}(x)) - (\mathbf{J} \Tilde{F}(x', \xi) - c\mathbf{I})^{\top} (x' - y_{c}(x')) \| \\
        & \leq ( 2(L_{F}+M_{J}+c) + D_{X}L_{J} + \tfrac{M_{J}L_{F}}{c} ) \| x - x' \|.
    \end{align*}
    Then we may set $L_{1} \triangleq 2(L_{F}+M_{J}+c) + D_{X}L_{J} + \tfrac{M_{J}L_{F}}{c}$, as desired.

    (iv) For any $y, y'\in X$ and any realization $\xi$, it follows from \eqref{gap-function} that
    \allowdisplaybreaks
    \begin{align*}
        &\quad | \Tilde{\theta}_{c}(x, y, \xi) - \Tilde{\theta}_{c}(x, y', \xi) | \\
        &\leq | \Tilde{F}(x, \xi)^{\top}(x - y) - \Tilde{F}(x, \xi)^{\top}(x - y') | + \tfrac{c}{2} \left| \| x - y \|^{2} - \| x - y' \|^{2} \right| \\
        &\leq \| \Tilde{F}(x, \xi) \| \| y - y' \| + \tfrac{c}{2} \| y - y' \| \| (x - y) + (x - y') \| \leq (M_{F} + c D_{X}) \| \| y - y' \|.
    \end{align*}
    Then we may set $L^{y}_{0} \triangleq M_{F} + c D_{X}$, as desired.
\end{proof}

\subsection{Proof of Lemma \ref{bias-moment-lemma}}\label{proof-bias-moment-lemma}
\begin{proof}
    The proof of (i) is similar to that of \cite[Lemma 3-(i)]{shanbhag-yousefian-2021}. We focus on the proof of (ii). The first bound follows immediately from Lemma \ref{RS-property-grad}-(v), the second bound follows from Lemma \ref{RS-property-grad}-(vi) and the unbiasedness property (i), and the third bound follows from Assumption $\mathrm{(D3)}$ and Lemma \ref{lower-level-error}. Next we derive a second moment bound for $e^{k}_{4}$. Similar to the proof of \cite[Lemma 3-(ii)]{shanbhag-yousefian-2021}, we may obtain that
    \begin{equation*}
        \mathbb{E} \left[ \|e^{k}_{4, j}\|^{2} \!\mid\! \mathcal{F}_{k} \right] = (N-1) \mathbb{E} \left[ \| g^{\eta_{k}, \epsilon_{k}}_{j}(x^{k}) \|^{2} \!\mid\! \mathcal{F}_{k} \right],
    \end{equation*}
    implying that
    \allowdisplaybreaks
    \begin{align*}
        \mathbb{E} \left[ \|e^{k}_{4, j}\|^{2} \!\mid\! \mathcal{F}_{k}  \right] &\leq (N-1) \mathbb{E} \left[ \| \nabla \theta_{c}(x^{k}) + e^{k}_{1} + e^{k}_{2, j} + e^{k}_{3, j} \|^{2} \!\mid\! \mathcal{F}_{k} \right] \\
        &\leq (N-1) \mathbb{E} \left[ 4\| \nabla \theta_{c}(x^{k}) \|^{2} + 4\| e^{k}_{1} \|^{2} + 4\| e^{k}_{2, j} \|^{2} + 4\| e^{k}_{3, j} \|^{2} \!\mid\! \mathcal{F}_{k} \right] \\
        &\leq 4(N-1) \left[ \, M^{2}_{\theta} + \eta^{2}_{k} L^{2}_{1} n^{2} + 16\sqrt{2\pi}L^{2}_{0}n + \eta^{-2}_{k}(L^{y}_{0})^{2}n^{2}\epsilon_{k} \, \right],
    \end{align*}
    where $M_{\theta} > 0$ is defined in assumption $\mathrm{(D4)}$. Since the unbiasedness property holds, i.e.,  $\mathbb{E}[ e^{k}_{4, j} \!\mid\! \mathcal{F}_{k} ] = 0$, the required result follows.
    \begin{align*}
        \mathbb{E}[ \|e^{k}_{4}\|^{2} \!\mid\! \mathcal{F}_{k} ] \leq \tfrac{4(N-1) \left[ M^{2}_{\theta} + \eta^{2}_{k} L^{2}_{1} n^{2} + 16\sqrt{2\pi}L^{2}_{0}n + \eta^{-2}_{k}(L^{y}_{0})^{2}n^{2}\epsilon_{k} \right]}{N_{k}}.
    \end{align*}
\end{proof}

\subsection{Proof of Lemma \ref{SAMGR-summation-upper-bound}}\label{proof-SAMGR-summation-upper-bound}
\begin{proof}
    By Lemma \ref{bias-moment-lemma}, we have that
    \allowdisplaybreaks
    \begin{align*}
        &\sum_{k=\lceil \lambda K \rceil}^{K} \mathbb{E}[ \| e^{k}_{1} + e^{k}_{2} + e^{k}_{3} + e^{k}_{4} \|^{2} \!\mid\! \mathcal{F}_{k} ] \leq 4 \sum_{k=\lceil \lambda K \rceil}^{K} \mathbb{E}[ \| e^{k}_{1} \|^{2} \!\mid\! \mathcal{F}_{k} ] \\
        &\quad + 4 \sum_{k=\lceil \lambda K \rceil}^{K} \mathbb{E}[ \| e^{k}_{2} \|^{2} \!\mid\! \mathcal{F}_{k} ] + 4 \sum_{k=\lceil \lambda K \rceil}^{K} \mathbb{E}[ \| e^{k}_{3} \|^{2} \!\mid\! \mathcal{F}_{k} ] + 4 \sum_{k=\lceil \lambda K \rceil}^{K} \mathbb{E}[ \| e^{k}_{4} \|^{2} \!\mid\! \mathcal{F}_{k} ] \\
        &\leq \sum_{k=\lceil \lambda K \rceil}^{K} 4\eta^{2}_{k} L^{2}_{1} n^{2} + \sum_{k=\lceil \lambda K \rceil}^{K} \tfrac{64\sqrt{2\pi}L^{2}_{0}n}{N_{k}} + \sum_{k=\lceil \lambda K \rceil}^{K} \tfrac{4(L^{y}_{0})^{2}n^{2}\epsilon_{k}}{\eta^{2}_{k}} \\
        &\quad + \sum_{k=\lceil \lambda K \rceil}^{K} \tfrac{16(N-1) \left[ M^{2}_{\theta} \, + \, \eta^{2}_{k} L^{2}_{1} n^{2} \, + \, 16\sqrt{2\pi}L^{2}_{0}n \, + \, \eta^{-2}_{k}(L^{y}_{0})^{2}n^{2}\epsilon_{k} \right]}{N_{k}}.
    \end{align*}
    By the definitions of $\eta_{k}$, $N_{k}$, and $t_{k}$ given in \eqref{SAMGR-parameters-choices}, we have
    \allowdisplaybreaks
    \begin{align*}
        &\sum_{k=\lceil \lambda K \rceil}^{K} 4\eta^{2}_{k} L^{2}_{1} n^{2} = \sum_{k=\lceil \lambda K \rceil}^{K} \tfrac{4 L^{2}_{1} n^{2}}{n^{2b}(k+1)^{1+\delta}} \leq \sum_{k=\lceil \lambda K \rceil}^{K} \tfrac{4 L^{2}_{1} n^{2-2b}}{k+1}, \\
        &\sum_{k=\lceil \lambda K \rceil}^{K} \tfrac{64\sqrt{2\pi}L^{2}_{0}n}{N_{k}} \leq \sum_{k=\lceil \lambda K \rceil}^{K} \tfrac{64\sqrt{2\pi}L^{2}_{0}n}{n^{a}(k+1)^{1+\delta}} \leq \sum_{k=\lceil \lambda K \rceil}^{K} \tfrac{64\sqrt{2\pi}L^{2}_{0}n^{2-(1+a)}}{k+1}, \\
        &\sum_{k=\lceil \lambda K \rceil}^{K} \tfrac{4(L^{y}_{0})^{2}n^{2}\epsilon_{k}}{\eta^{2}_{k}} \overset{\textrm{Lemma } \ref{lower-level-error}}{\leq} \sum_{k=\lceil \lambda K \rceil}^{K} \tfrac{4(L^{y}_{0})^{2}n^{2}c_{\epsilon}}{\eta^{2}_{k}t_{k}} \leq \sum_{k=\lceil \lambda K \rceil}^{K} \tfrac{4(L^{y}_{0})^{2}c_{\epsilon}n^{2-(e-2b)}}{k+1},
    \end{align*}
    where $c_{\epsilon} \triangleq \max \left\{\tfrac{(c^{2}_{F}+v^{2}_{F})\alpha^{2}_{0}}{2c\alpha_{0}-1}, \: \Gamma\sup_{y\in X}\|y_{0}-y\|^{2} \right\} > 0$. Similarly, we also have that
    \allowdisplaybreaks
    \begin{align*}
        &\sum_{k=\lceil \lambda K \rceil}^{K} \tfrac{16(N-1) \left[ M^{2}_{\theta} \, + \, \eta^{2}_{k} L^{2}_{1} n^{2} \, + \, 16\sqrt{2\pi}L^{2}_{0}n \, + \, \eta^{-2}_{k}(L^{y}_{0})^{2}n^{2}\epsilon_{k} \right]}{N_{k}} \\
        &\leq \sum_{k=\lceil \lambda K \rceil}^{K} \tfrac{16(N-1)M^{2}_{\theta}n^{2-(2+a)}}{k+1} + \sum_{k=\lceil \lambda K \rceil}^{K} \tfrac{16(N-1)L^{2}_{1}n^{2-(a+2b)}}{(k+1)^{2}} \\
        &+ \sum_{k=\lceil \lambda K \rceil}^{K} \tfrac{256(N-1)\sqrt{2\pi}L^{2}_{0}n^{2-(1+a)}}{k+1} + \sum_{k=\lceil \lambda K \rceil}^{K} \tfrac{16(N-1)(L^{y}_{0})^{2}c_{\epsilon}n^{2-(a+e-2b)}}{(k+1)^{2}}.
    \end{align*}
    Next, we derive two upper bounds. Noticing that $\lceil \lambda K \rceil \geq 1$ and that $K\geq \tfrac{2}{1-\lambda}$, and hence $K-1\geq \lceil \lambda K \rceil$, we know from \cite[Lemma 8.26]{beck-2017} that
    \allowdisplaybreaks
    \begin{align*}
        & \sum_{k=\lceil \lambda K \rceil}^{K}\tfrac{1}{k+1} = \tfrac{1}{\lceil \lambda K \rceil+1} + \cdots + \tfrac{1}{K+1} \leq \frac{1}{2} + \int_{\lceil \lambda K \rceil}^{K}\tfrac{1}{t+1}dt \leq \tfrac{1}{2} + \log{\tfrac{K+1}{\lambda K+\lambda}} = \tfrac{1}{2} - \log{\lambda}, \\
        & \sum_{k=\lceil \lambda K \rceil}^{K}\tfrac{1}{(k+1)^{2}} = \tfrac{1}{(\lceil \lambda K \rceil+1)^{2}} + \cdots + \tfrac{1}{(K+1)^{2}} \leq \tfrac{1}{4} + \int_{\lceil \lambda K \rceil}^{K}\tfrac{1}{(t+1)^{2}}dt \leq \tfrac{1}{4} + \tfrac{(1-\lambda)^{2}}{\lambda(3-\lambda)}.  
    \end{align*}
    We define the function $h(a, b, e) \triangleq \min \{ 2b, 1+a, e-2b, 2+a, a+2b, 1+a, a+e-2b \}$ $= \min \{ 2b, 1+a, e-2b \}$. Let $\Theta(\lambda, N) > 0$ be defined as
\allowdisplaybreaks
    \begin{equation}\label{Theta-lambda}
        \begin{aligned}
            \Theta(\lambda, N) &\triangleq \big[ (16N-12) L^{2}_{1} + (256N-192)\sqrt{2\pi} L^{2}_{0} + (16N-12)c_{\epsilon}(L^{y}_{0})^{2} \\
            &\qquad + (16N-16)M^{2}_{\theta} \big] \times \max \left\{ \tfrac{1}{2} - \log{\lambda}, \tfrac{1}{4} + \tfrac{(1-\lambda)^{2}}{\lambda(3-\lambda)} \right\} > 0.
        \end{aligned}
    \end{equation}
    By combining the above bounds, we have
    \begin{align*}
        \sum_{k=\lceil \lambda K \rceil}^{K} \mathbb{E}[ \| e^{k}_{1} + e^{k}_{2} + e^{k}_{3} + e^{k}_{4} \|^{2} \!\mid\! \mathcal{F}_{k} ] \leq \Theta(\lambda, N) \: n^{2-h(a, b, e)}.
    \end{align*}
\end{proof}

\subsection{Proof of Lemma \ref{QG-network-game}}\label{proof-QG-network-game}
\begin{proof}
    By computation, we obtain that for any $i\in [N]$,
    \begin{equation*}
        F_i(x) = \nabla_{x_{i}} f_{i}(x_{i}, x_{-i}) = \left( \mathbb{E}\left[\tfrac{M_{G}}{\|x^{\ell}\|_{2}(b^{\ell}(\bxi)-\|x^{\ell}\|_{2})^{2}} - \beta(\bxi)\right]x^{\ell}_{i} \right)_{\ell\in \mathcal{L}}.
    \end{equation*}
    For any $\ell \in \mathcal{L}$, we define $F^{\ell}(x^{\ell})$ as
    \begin{equation}\label{proof-QG-network-game-eqn1}
        F^{\ell}(x^{\ell}) \triangleq \left( \mathbb{E}\left[\tfrac{M_{G}}{\|x^{\ell}\|_{2}(b^{\ell}(\bxi)-\|x^{\ell}\|_{2})^{2}} - \beta(\bxi)\right]x^{\ell}_{i} \right)_{i=1}^{N} = \mathbb{E}\left[\tfrac{M_{G}}{\|x^{\ell}\|_{2}(b^{\ell}(\bxi)-\|x^{\ell}\|_{2})^{2}} - \beta(\bxi)\right] x^{\ell}.
    \end{equation}
    Then we can write $F(x) = (F_{i}(x))_{i=1}^{N}$ or $F(x) = (F^{\ell}(x^{\ell}))_{\ell \in \mathcal{L}}$. We claim that for each $\ell \in \mathcal{L}$, the mapping $F^{\ell}$ satisfies the \ref{SP} property on $X^{\ell}$, where $X^{\ell}$ denotes the feasible region of $x^\ell$. Indeed, by our assumption, we have that
    \begin{equation}\label{proof-QG-network-game-eqn2}
        \min_{x^{\ell}\in X^{\ell}} \left( \mathbb{E}\left[\tfrac{M_{G}}{\|x^{\ell}\|_{2}(b^{\ell}(\bxi)-\|x^{\ell}\|_{2})^{2}} - \beta(\bxi)]\right] \right) \geq \left(  \tfrac{M_{G}}{d^{\ell}_{\max} (b_{\rm max}^{\ell}-d_{\min}^{\ell})^{2}} - \mathbb{E}[\beta(\bxi)] \right) \triangleq \eta^{\ell} > 0.
    \end{equation}
    By the \ref{SP} definition, we first assume that $F^{\ell}(x^{\ell})^{\top}(y^{\ell}-x^{\ell}) \geq 0$ for any $x^{\ell}, y^{\ell} \in X^{\ell}$. Then we may deduce that
    \begin{align*}
        0 \leq F^{\ell}(x^{\ell})^{\top}(y^{\ell}-x^{\ell}) \overset{\eqref{proof-QG-network-game-eqn1}}{=} \mathbb{E}\left[\tfrac{M_{G}}{\|x^{\ell}\|_{2}(b^{\ell}(\bxi)-\|x^{\ell}\|_{2})^{2}} - \beta(\bxi)\right] \langle x^{\ell}, y^{\ell} - x^{\ell} \rangle.
    \end{align*}
    Then by \eqref{proof-QG-network-game-eqn2}, it follows that $\langle x^{\ell}, y^{\ell} - x^{\ell} \rangle \geq 0$. Therefore, we have that
    \allowdisplaybreaks
    \begin{align*}
        & F^{\ell}(y^{\ell})^{\top}(y^{\ell}-x^{\ell}) = \mathbb{E}\left[\tfrac{M_{G}}{\|y^{\ell}\|_{2}(b^{\ell}(\bxi)-\|y^{\ell}\|_{2})^{2}} - \beta(\bxi)\right] \langle y^{\ell}, y^{\ell} - x^{\ell} \rangle \\
        &\geq \mathbb{E} \left[ \tfrac{M_{G}}{\|y^{\ell}\|_{2}(b^{\ell}(\bxi)-\|y^{\ell}\|_{2})^{2}} - \beta(\bxi) \right] \langle y^{\ell}, y^{\ell} - x^{\ell} \rangle \\
        &- \underbrace{ \mathbb{E} \left[ \tfrac{M_{G}}{\|y^{\ell}\|_{2}(b^{\ell}(\bxi)-\|y^{\ell}\|_{2})^{2}} - \beta(\bxi) \right] \langle x^{\ell}, y^{\ell} - x^{\ell} \rangle }_{\geq 0} \\
        &= \mathbb{E} \left[ \tfrac{M_{G}}{\|y^{\ell}\|_{2}(b^{\ell}(\bxi)-\|y^{\ell}\|_{2})^{2}} - \beta(\bxi) \right] \langle y^{\ell} - x^{\ell}, y^{\ell} - x^{\ell} \rangle \overset{\eqref{proof-QG-network-game-eqn2}}{\geq} \eta^{\ell} \| y^{\ell} - x^{\ell} \|^{2},
    \end{align*}
    implying that $F^{\ell}$ satisfies the \ref{SP} property on $X^{\ell}$ with $\eta^{\ell} > 0$. By Theorem \ref{QNE-existence}, we know that $\mathrm{SOL}\:(X, F)$ is nonempty. For any $x^{\ast} \in X^{\ast}$, we have
    \begin{equation}\label{proof-QG-network-game-eqn3}
        0 \leq (x - x^{\ast})^{\top} F(x^{\ast}) = \sum_{\ell \in \mathcal{L}} (x^{\ell} - x^{\ast, \ell})^{\top} F^{\ell}(x^{\ast, \ell}), ~ \forall x\in X.
    \end{equation}
    Since \eqref{proof-QG-network-game-eqn3} holds for any $x\in X$, we have that $(x^{\ell}-x^{\ast, \ell})^{\top} F^{\ell}(x^{\ast, \ell})\geq 0$ holds for any $x^{\ell} \in X^{\ell}$ and any $\ell\in \mathcal{L}$. By the \ref{SP} property for $F^{\ell}$, $(x^{\ell}-x^{\ast, \ell})^{\top} F^{\ell}(x^{\ell})\geq {\eta^{\ell}} \|x^{\ell}-x^{\ast, \ell}\|^{2}$ holds for any $x^{\ell} \in X^{\ell}$ and any $\ell\in \mathcal{L}$. Therefore, we have that
    \begin{equation*}
        (x - x^{\ast})^{\top} F(x) = \sum_{\ell \in \mathcal{L}} (x^{\ell} - x^{\ast, \ell})^{\top} F^{\ell}(x^{\ell}) \geq \eta \sum_{\ell \in \mathcal{L}} \| x^{\ell} - x^{\ast, \ell} \|^{2} = \eta \| x - x^{\ast} \|^{2}, ~ \forall x\in X,
    \end{equation*}
    where $\eta \triangleq \min_{\ell \in \mathcal{L}} \eta^{\ell} > 0$. Therefore, the map $F$ satisfies the \ref{GSM} property with $\eta > 0$. As discussed in subsection \ref{Sec-2.2}, \ref{GSM} implicitly implies that $X^{\ast}$ is a singleton, and hence the \ref{QG} property holds. Then we complete the proof.
\end{proof}

\bibliographystyle{siamplain}
\bibliography{references}
\end{document}

%% file: ex_shared.tex
\usepackage{colortbl}
\usepackage{lipsum, tcolorbox}
\usepackage{tikz}
\usetikzlibrary{matrix, arrows.meta, positioning}
\usepackage{float}
\usepackage{makecell}
\usepackage{pifont}
\usepackage{caption}
\usepackage{subcaption}
\usepackage{mdframed}
\usepackage[table]{xcolor}
\usepackage{amsfonts}
\usepackage{graphicx}
\usepackage{epstopdf}
\usepackage{amssymb}
\usepackage{array}
\usepackage{multirow}
\usepackage{amsopn}
\usepackage{algorithmic}
\newcommand{\e}{\mathbf{e}}
\newcommand{\bv}{\mathbf{v}}
\newcommand{\uu}{{\mathbf{u}}}
\newcommand{\bxi}{\boldsymbol{\xi}}

\ifpdf
  \DeclareGraphicsExtensions{.eps,.pdf,.png,.jpg}
\else
  \DeclareGraphicsExtensions{.eps}
\fi

\newsiamremark{remark}{Remark}
\newsiamremark{hypothesis}{Hypothesis}
\crefname{hypothesis}{Hypothesis}{Hypotheses}
\newsiamthm{claim}{Claim}
\newsiamremark{fact}{Fact}
\crefname{fact}{Fact}{Facts}

\def\argmin{\mathop{\rm argmin}}

\def\argmax{\mathop{\rm argmax}}

\newcommand{\pmat}[1]{\begin{pmatrix} #1 \end{pmatrix}}

\newmdenv[
  backgroundcolor=gray!20,
  linecolor=white,
  linewidth=0pt,
  innertopmargin=5pt,
  innerbottommargin=5pt,
  innerleftmargin=8pt,
  innerrightmargin=8pt,
]{custombox}

\headers{STOCHASTIC GRADIENT-RESPONSE SCHEMES FOR COMPUTING QNE}{ZHUOYU XIAO and UDAY V. SHANBHAG}

\title{Computing Equilibria in Stochastic Nonconvex and Non-monotone Games via Gradient-Response Schemes\thanks{Submitted to the editors \today. A very preliminary version \cite{xiao-shanbhag-2025-GR-conference} appeared in the 7th Annual Learning for Dynamics \& Control Conference, 2025. This work was funded in part by the ONR under grants N$00014$-$22$-$1$-$2589$, AFOSR Grant FA9550-24-1-0259,  DOE grant DE-SC$0023303$.}}

\author{Zhuoyu Xiao \and Uday V. Shanbhag\thanks{University of Michigan, Ann Arbor, MI, 48109. (\email{zyxiao, udaybag@umich.edu}).}}
